\documentclass[a4paper,10pt]{amsart}

\DeclareRobustCommand{\SkipTocEntry}[5]{} 

\usepackage[utf8]{inputenc}
\usepackage[T1]{fontenc}
\usepackage[english]{babel}
\usepackage{amsmath,amsfonts,amssymb}
\usepackage{enumitem}
\usepackage{mathabx}
\usepackage{dsfont}
\usepackage{color}
\usepackage{xcolor} 
\usepackage{tikz, pgfplots}
\usepackage[scale = .75]{geometry}
\numberwithin{equation}{section}

\allowdisplaybreaks
\newcommand{\NN}{\mathbb{N}}
\newcommand{\ZZ}{\mathbb{Z}}
\newcommand{\TT}{\mathbb{T}}
\newcommand{\RR}{\mathbb{R}}
\newcommand{\cF}{\mathcal{F}}

\newcommand{\cS}{\mathcal{S}}
\newcommand{\cA}{\mathcal{A}}
\newcommand{\dd}{\mathrm{d}}
\newcommand{\ve}{\varepsilon}
\newcommand{\eps}{\varepsilon}

\newcommand{\lla}{\left\langle}
\newcommand{\rra}{\right\rangle}

\newcommand{\ds}{\displaystyle}

\definecolor{vert}{rgb}{0.0, 0.5, 0.0}

\newtheorem{theo}{Theorem}[section]
\newtheorem{prop}[theo]{Proposition}
\newtheorem{cor}[theo]{Corollary}
\newtheorem{lem}[theo]{Lemma}
\newtheorem{rema}[theo]{Remark}

\newcommand{\be}{\begin{equation}}
    \newcommand{\ee}{\end{equation}}
\newcommand{\beqn}{\begin{equation}}
    \newcommand{\eeqn}{\end{equation}}

\newcommand{\mua}{\mu_\alpha}

\newcommand{\Lah}{L_\alpha^{h}}
\newcommand{\Laht}{L_\alpha^{h,J,K}}
\newcommand{\Dh}{\Lambda_\alpha^h}
\newcommand{\Dht}{\Lambda_\alpha^{h,J,K}}
\newcommand{\Ght}{\Gamma_\alpha^{h,J,K}}

\newcommand{\Symh}{\mathcal{S}_\alpha^h}
\newcommand{\Antsymh}{\mathcal{A}_\alpha^h}

\newcommand\pare[1]{\left(#1\right)}

\newcommand{\dt}{{\Delta t}}
\newcommand{\dx}{{\Delta x}}
\newcommand{\dv}{{\Delta v}}

\newcommand{\dpi}{{\Pi_\dv}}
\newcommand{\lm}{{\ell^2_{\dx,\dv}(M^{-1})}}
\newcommand{\hm}{{H^1_{\dx,\dv}(M^{-1})}}

\definecolor{lapislazuli}{rgb}{0.15, 0.38, 0.61}

\usetikzlibrary{patterns}
\newcommand{\logLogSlopeTriangle}[5]
{
    \pgfplotsextra
    {
        \pgfkeysgetvalue{/pgfplots/xmin}{\xmin}
        \pgfkeysgetvalue{/pgfplots/xmax}{\xmax}
        \pgfkeysgetvalue{/pgfplots/ymin}{\ymin}
        \pgfkeysgetvalue{/pgfplots/ymax}{\ymax}
        
        \pgfmathsetmacro{\xArel}{#1}
        \pgfmathsetmacro{\yArel}{#3}
        \pgfmathsetmacro{\xBrel}{#1-#2}
        \pgfmathsetmacro{\yBrel}{\yArel}
        \pgfmathsetmacro{\xCrel}{\xArel}
        
        \pgfmathsetmacro{\lnxB}{\xmin*(1-(#1-#2))+\xmax*(#1-#2)} 
        \pgfmathsetmacro{\lnxA}{\xmin*(1-#1)+\xmax*#1} 
        \pgfmathsetmacro{\lnyA}{\ymin*(1-#3)+\ymax*#3} 
        \pgfmathsetmacro{\lnyC}{\lnyA+#4*(\lnxA-\lnxB)}
        \pgfmathsetmacro{\yCrel}{\lnyC-\ymin)/(\ymax-\ymin)}
        
        \coordinate (A) at (rel axis cs:\xArel,\yArel);
        \coordinate (B) at (rel axis cs:\xBrel,\yBrel);
        \coordinate (C) at (rel axis cs:\xCrel,\yCrel);
        
        \draw[#5]   (A)-- node[pos=0.5,anchor=north] {\scriptsize{1}}
        (B)-- 
        (C)-- node[pos=0.,anchor=west] {\scriptsize{#4}} 
        (A);
    }
}

\allowdisplaybreaks

\usepackage[hidelinks]{hyperref}

\title[Numerical method for fractional Fokker-Planck]{On a structure-preserving numerical method \\ for fractional Fokker-Planck equations}

\author{Nathalie Ayi}
\address[N. Ayi]{Sorbonne Universit\'e, Universit\'e de Paris, CNRS, Laboratoire Jacques-Louis Lions, 4 place Jussieu, 75005 Paris, France}
\email{nathalie.ayi@sorbonne-universite.fr}
\author{Maxime Herda}
\address[M. Herda]{Inria, Univ. Lille, CNRS, UMR 8524 - Laboratoire Paul Painlev\'e, F-59000 Lille, France}
\email{maxime.herda@inria.fr}
\author{H\'el\`ene Hivert}
\address[H. Hivert]{Univ. Lyon, \'Ecole centrale de Lyon, CNRS UMR 5208, Institut Camille Jordan, F-69134 \'Ecully, France}
\email{helene.hivert@ec-lyon.fr}
\author{Isabelle Tristani}
\address[I. Tristani]{D\'epartement de math\'ematiques et applications, \'Ecole normale sup\'erieure, CNRS, PSL Research University, 45 rue d'Ulm, 75005 Paris, France }
\email{isabelle.tristani@ens.fr}

\begin{document}
    
    \begin{abstract}
        In this paper, we introduce and analyse numerical schemes for the homogeneous and the kinetic Lévy-Fokker-Planck equation. The discretizations are designed to preserve the main features of the continuous model such as conservation of mass, heavy-tailed equilibrium and (hypo)coercivity properties. We perform a thorough analysis of the numerical scheme and show exponential stability {and convergence of the scheme}. Along the way, we introduce new tools of discrete functional analysis, such as discrete nonlocal Poincaré and interpolation inequalities adapted to fractional diffusion. Our theoretical findings are illustrated and complemented with numerical simulations.\\[1em]
    \end{abstract}
    
    \maketitle

    \section{Introduction}
    
    In this paper, we are interested in the numerical discretization of the kinetic Lévy-Fokker-Planck equation (or fractional Fokker-Planck equation). The continuous model describes the evolution of a distribution function $f\equiv f(t,x,v)$ which depends on time $t\geq0$, position in a periodic domain $x\in\TT^d = \RR^d/\ZZ^d$ and velocity~$v\in\RR^d$, and solves
    \begin{equation}\label{eq:kinfracFP}
        \partial_t f + v\cdot\nabla_x f\ =\ \nabla_v\cdot(vf)-(-\Delta_v)^{\alpha/2}f =: L_\alpha f\,,
    \end{equation}
    supplemented with the initial condition $f(0,\cdot,\cdot) = f^0$. For $\alpha=2$, the model coincides with the so-called kinetic Fokker-Planck equation. Here we are interested in the case $\alpha\in(0,2)$. In this range of the parameter, the fractional Laplacian $-(-\Delta_v)^{\alpha/2}$ is defined in the following way. For any Schwartz function~$g:\RR^d\to\RR$, one has $\cF((-\Delta_v)^{\alpha/2}g)(\xi)\ =\ |\xi|^{\alpha}\cF(g)(\xi)$ where $\cF(\cdot)$ denotes the Fourier transform and $|\cdot|$ is the Euclidean norm on $\mathbb{R}^d$. Another equivalent definition \cite{kwasnicki_2017_fractional} is given by the singular integral
    \begin{equation}\label{eq:fraclap_int1}
        (-\Delta_v)^{\alpha/2}g(v)\ =\ C_{d,\alpha}\,\mathrm{P.V.}\int_{\RR^d}\frac{g(v)-g(w)}{|v-w|^{d+\alpha}}\, \dd w\,,
    \end{equation}
    where $\mathrm{P.V.}$ stands for the principal value, and the multiplicative constant is given by \[C_{d,\alpha}=\frac{2^\alpha\Gamma(\frac{d+\alpha}{2})}{(\pi^{d/2}|\Gamma(-\frac{\alpha}{2})|)}\,,\] where $\Gamma(\cdot)$ is the Gamma function.

    When $\alpha<2$, the velocity density which generates the kernel of $L_\alpha$, called local equilibrium, is heavy tailed. Indeed, by passing to Fourier variables one has $\cF(L_\alpha g)(\xi) = -\xi\cdot\nabla_\xi\cF(g)(\xi)-|\xi|^{\alpha}\cF(g)(\xi)$. From this formula, one has that the only probability distribution satisfying $L_\alpha\mu_\alpha = 0$ is given by
    \begin{equation}\label{eq:defmua}
        \mu_\alpha(v)\ =\  \frac{1}{(2\pi)^d}\int_{\mathbb{R}^d}\exp\left(i\,v\cdot\xi-\frac{|\xi|^\alpha}{\alpha}\right)\,\dd\xi\,.
    \end{equation}
    The density $\mu_\alpha$ is known as a symmetric stable density and is related to the theory of  Levy processes (see~\cite{applebaum_2009_levy}). Away from the origin, the Fourier transform of $\mu_\alpha$ is smooth and rapidly decaying at infinity. In particular $\mu_\alpha$ is smooth. However, at~$\xi=0$ the Fourier transform is as regular as $\xi\mapsto|\xi|^\alpha$, which implies that $\mu_\alpha(v)$ decays as $|v|^{-\alpha-d}$ when~$|v|\to\infty$. This qualitative behavior can be estimated more precisely through pointwise bounds from above and below on $\mu_\alpha$ and its derivatives (see \eqref{eq:mualpha_est1}, \eqref{eq:mualpha_est2} and \eqref{eq:mualpha_est3} in Appendix~\ref{sec:boundsmua}). Non-Maxwellian, algebraically decaying velocity densities arise in the modelling of astrophysical plasmas \cite{summers1991modified, pierrard2010kappa}. More specifically, Vlasov-Lévy-Fokker-Planck models \cite{cesbron2012anomalous, aceves_2019_fractional, cesbron_2020_fractional} such as~\eqref{eq:kinfracFP} as well as other kinetic models with heavy tailed local equilibrium \cite{mellet2010fractional, mellet2011ARMA, benabdallah2011fractional, bouin2020quantitative} have attracted attention in the recent years because of their asymptotic properties. Indeed, because of the slow decay in velocity,  the macroscopic diffusion limits of these kinetic equations are fractional diffusion equations in the space variable. These asymptotics are obtained after an anomalous rescaling of the kinetic equation. The design of appropriate asymptotic preserving numerical schemes in these limits has been investigated in \cite{crouseilles2016numerical1, crouseilles2016numerical2, wang2016asymptotic}.

    Another natural question concerning asymptotic behaviors in these models is the long-time behavior of solutions. Using conservation of mass and space pedriodic boundary conditions, one easily infers that the global equilibrium is given by $(x,v)\mapsto\lla f^0\rra\,\mu_\alpha(v)$, where $\lla f^0\rra:=\int f^0\dd x\dd v$ denotes the initial mass.  The long time behavior can be quantified by the time evolution of well-chosen norms. While the Fourier transform of the Green kernel of \eqref{eq:kinfracFP} can be expressed explicitly, the Fourier inversion is not explicit when $\alpha\neq2$. Therefore, it is not easy to derive estimates from this representation of solutions. Instead of explicit representations, energy method allows to quantify the time evolution of appropriate norms. For dissipative kinetic equations, which are usually degenerate in the sense that diffusion or relaxation happens only in the velocity variable, hypocoercivity methods \cite{herau_2004_isotropic, herau_2006_hypo, villani_2009_hypocoercivity, guo_2010_boltzmann, dolbeault2015hypocoercivity, herau_2018_introduction} are well-suited energy methods which allows one to recover dissipation properties in the whole phase-space. They exploit the interaction between transport and collision operators, respectively $v\cdot\nabla_x$ and $L_\alpha$ in the present case. Recently, the adaptation of hypocoercivity methods has allowed to derive quantitative long-time behavior estimates for \eqref{eq:kinfracFP} and related models \cite{ayi_2019_note, bouin_2019_fractional}. In \cite{bouin_2019_fractional}, the authors generalize the so-called $L^2$ method (which was introduced in \cite{herau_2006_hypo} and developed in~\cite{guo_2010_boltzmann, dolbeault2015hypocoercivity}) to the fractional case. In \cite{ayi_2019_note}, the $H^1$ method of \cite{herau_2004_isotropic, villani_2009_hypocoercivity} is adapted by the authors of the present contribution. Let us briefly recall the strategy and the results of~\cite{ayi_2019_note}. The $H^1$ method relies on the estimation of the functional
    \begin{equation} \label{def:triplenorm}
        \mathcal{H}(h,h)
        \ 
        :=\  \|h\|^2_{L^2_{x,v}(\mu_\alpha^{-1})} \,+\, a\|\nabla_x h \|^2_{L^2_{x,v}(\mu_\alpha^{-1})} \,+\, b\, \|\nabla_v h \|^2_{L^2_{x,v}(\mu_\alpha^{-1})}  + 2\,c\, \langle \nabla_x h , \nabla_v h \rangle_{L^2_{x,v}(\mu_\alpha^{-1})}\,,
    \end{equation}
    where $L^2_{x,v}(\mu_\alpha^{-1})$ is the Hilbert space of functions which are square integrable against the weight $\mu_\alpha^{-1}$, with canonical norm $\|\cdot\|_{L^2_{x,v}(\mu_\alpha^{-1})}$ and scalar product $\lla\cdot,\cdot\rra_{L^2_{x,v}(\mu_\alpha^{-1})}$. The function $h$ denotes the difference between the transient and steady solutions, that is $h(t,x,v) = f(t,x,v)-\lla f^0\rra\,\mu_\alpha(v)$.

    For well-chosen  positive constants $a,b,c>0$, \eqref{def:triplenorm} is equivalent to the weighted Sobolev norm $\|h\|_{H^1_{x,v}(\mu_\alpha^{-1})}^2\,:=\,\|h\|_{L^2_{x,v}(\mu_\alpha^{-1})}^2 + \|\nabla_x h\|_{L^2_{x,v}(\mu_\alpha^{-1})}^2 + \|\nabla_v h\|_{L^2_{x,v}(\mu_\alpha^{-1})}^2$ and is dissipated along the dynamics. As a consequence of this dissipation and well-suited non-local Poincaré inequalities \cite{gentil_2008_levy, wang_2014_simple}, one can prove the following result on the long-time behavior of the solutions.
    \begin{theo}[\cite{ayi_2019_note}]\label{t:maincont}
        Let $f$ be a solution of the kinetic L\'evy-Fokker-Planck equation \eqref{eq:kinfracFP} with initial data~$f^0\in H^1_{x,v}(\mu_\alpha^{-1})$.  Then, for all $t\geq0$ one has
        \[
        \|f(t) - \lla f^0\rra\,\mu_\alpha\|_{H^1_{x,v}(\mu_\alpha^{-1})}\ \leq\ C\,\|f^0 - \lla f^0\rra\,\mu_\alpha\|_{H^1_{x,v}(\mu_\alpha^{-1})}\,e^{-\lambda t}\,
        \]
        for some constants $C\geq1$ and $\lambda>0$ depending only on $d$ and $\alpha$.
    \end{theo}
    While the proof of Theorem~\ref{t:maincont} (see~\cite{ayi_2019_note}) follows the classical $H^1$ hypocoercivity strategy described above (see also \cite{villani_2009_hypocoercivity, herau_2018_introduction} and references therein), there are challenges which are specific to the fractional case. The main difficulties can already be seen when estimating the propagation of weighted Sobolev norms for the space homogeneous version fractional Fokker-Planck equation \cite{gentil_2008_levy, tristani_2015_FFP, mischler2017uniform, lafleche_2020_FFP}, which is
    \begin{equation}\label{eq:fracFP}
        \partial_t f\ =\  L_\alpha f\ = \ \nabla_v\cdot(vf)-(-\Delta_v)^{\alpha/2}f\,.
    \end{equation}
    
    First, the fractional Fokker-Planck operator $L_\alpha$ is not symmetric in its natural Hilbert space. More precisely for (say) Schwartz functions $f,g$, the operator $-L_\alpha$ admits the decomposition
    \begin{equation} \label{eq:splitLa}
        -\lla L_\alpha f,\,g\rra_{L^2_v(\mu_\alpha^{-1})}\ =\ \cS_\alpha^v(f,g)\,+\,\cA_\alpha^v(f,g)
    \end{equation}
    with $\cS_\alpha^v$ symmetric and $\cA_\alpha^v$  skew-symmetric. While $\cA_\alpha^v = 0$ when $\alpha=2$, it is non-trivial in the fractional case. It turns out that this decomposition is crucial in the proof of (hypo)coercivity estimates for the equation~\eqref{eq:kinfracFP}. { More precisely, given any $g =e^{tL_\alpha}g_0$ and operators $A$ and $B$, one has formally that~$\frac{\dd}{\dd t}\lla Ag,Bg\rra = \lla [A,L_\alpha]g,Bg\rra + \lla Ag,[B,L_\alpha]g\rra -2\cS_{\alpha}^v(Ag,Bg)$. Therefore the skew symmetric part of the operator $L_\alpha$ only appears in commutators. This observation enables us to avoid loss of moments in velocities in our estimates which one would face with bad rearrangements of the terms (see \cite{ayi_2019_note} for details). }
    
    { Secondly, in our energy estimates, it is important to quantify the regularization properties of the Lévy-Fokker-Planck operator coming from the dissipation terms thanks to specific functional inequalities in order to bound remainder terms which involve full derivatives in velocity.} One important example of these inequalities is the following (see \cite[Proposition~4.1]{ayi_2019_note} for a proof). For all $\ve>0$, there is a constant $K(\ve)\equiv K(\ve,\alpha,d)>0$ such that 
    \begin{equation}\label{eq:interp}
        \|\nabla_vf\|_{L^2_v(\mu_\alpha^{-1})}^2\ \lesssim\ K(\ve) \left(\cS_\alpha^v(f,f)\,+\,\|\Pi_v f\|_{L^2_v(\mu_\alpha^{-1})}^2\right)+\, \ve\,\cS_\alpha^v(\nabla_vf,\nabla_vf)\, ,
    \end{equation}
    where $(\Pi_v f)(v)\  =\  \left(\int_{\RR^d}f(w)\,\dd w\right)\,\mu_\alpha(v)$ is the orthogonal projection of $f$ onto the kernel of $L_\alpha$ in the space $L^2_v(\mu_\alpha^{-1})$.
    
\addtocontents{toc}{\SkipTocEntry}
\subsection*{Goal of the paper and main results} 
In this paper, we are interested in the numerical discretization of \eqref{eq:fracFP} and \eqref{eq:kinfracFP}. Our main results are the design of a consistent, stable and structure preserving numerical method for these equations in dimension $d=1$ (\emph{i.e.} two-dimensional in the phase-space), as well as its analysis, implementation and simulation. The keystone is the discretization of the  operator~$L_\alpha$. In terms of preservation of the structure, our numerical method satisfies the properties of
    \begin{itemize}
        \item conservation of mass;
        \item preservation of the heavy-tailed local equilibrium $\mu_\alpha$;
        \item preservation of coercivity properties in the homogeneous case;
        \item preservation of the hypocoercivity properties in the inhomogeneous case;
        \item approximation of the fractional Fokker-Planck operator $L_\alpha$ on the whole line with a discretization on a truncated domain;
        \item preservation of the asymptotics $\alpha\to 2^-$.
    \end{itemize}
    Let us mention that while we do not prove it, the preservation of non-negativity of solutions is also observed numerically. 
    
    In the recent years, there has been several works around the numerical analysis of hypocoercivity properties for discretizations of kinetic equations. The first contribution \cite{poretta_2017_numerical} concerns a finite difference discretization of the Kolmogorov equation. For the same model, hypocoercivity properties for other types of schemes were studied in \cite{foster_2017_structure, georgoulis_2021_hypocoercivity}. Concerning the classical kinetic Fokker-Planck equation, there has been to our knowledge two main contributions \cite{dujardin_2020_coercivity, bessemoulin_2020_hypocoercivity} dealing respectively with $H^1$ and $L^2$ hypocoercivity. 
In the present contribution, the range of models is extended to the fractional-Fokker-Planck case. {Notice also that using a perturbative argument, in~\cite{mischler2017uniform}, the exponential stability of a model which in a sense is akin to the discretized fractional Fokker-Planck equation is proven. 
}

    Our discretization of the Lévy-Fokker-Planck operator $L_\alpha$ is based on a conservative finite difference / finite volume approach. The fractional Laplacian is discretized following the Huang and Oberman finite difference-quadrature method \cite{huang_2014_discretization}. This method is based on the integral representation \eqref {eq:fraclap_int1}. While they may seem natural, spectral methods based on the Fourier definition of $(-\Delta)^{\alpha/2}$ are not efficient and lead to aliasing errors because of the slow decay of functions at infinity \cite{huang_2014_discretization}. From the discrete version of the fractional Laplacian, we introduce an equilibrium and mass preserving discretization of the drift term $\nabla_v\cdot(vf)$. 
    
    { From the discretization of $L_\alpha$, the complete scheme for the homogeneous Fokker-Planck equation is readily obtained with a fully implicit Euler scheme in time (see \eqref{eq:def_scheme_hom}).  In the inhomogeneous case, the scheme is also fully implicit in time with a centered discretization of the transport term (see \eqref{eq:def_scheme_kin}-\eqref{eq:def_transport}). A large part of this paper is devoted to the numerical and asymptotic analysis of these Eulerian schemes in unbounded velocity domain. For their practical implementation the velocity domain is truncated and the discrete Lévy-Fokker-Planck operator is adapted accordingly  thanks to consistent and structure preserving numerical boundary conditions (see Section~\ref{sec:truncation}). At the very end of the paper (see Section~\ref{sec:semilag}) we also propose an alternative scheme in the inhomogeneous case. It is based on a semi-Lagrangian version of our scheme with time splitting, which improves computational efficiency, but for which the rigorous numerical and asymptotic analysis is out of the scope of the present paper.}
    
    For the Eulerian schemes in unbounded domains, we rigorously show coercivity and hypocoercivity properties leading to exponential stability of the discrete solution. These results are stated formally in Theorem~\ref{theo:coerc_discrLFP} for the homogeneous case and Theorem~\ref{theo:hypo_discrLFP} and Theorem~\ref{theo:hypo_fulldiscrLFP} for the kinetic case. In order to prove these results we need to adapt many of the continuous properties to the (more challenging) discrete setting. Apart from the exponential stability results, many intermediate results of discrete functional analysis have their own importance. They are gathered in Section~\ref{sec:funcanal} and include a discrete version of nonlocal Poincaré inequalities as well as many new interpolation and embedding inequalities involving discrete fractional operators and norms.
    
    Compared to the continuous setting, there are two main challenges concerning the coercivity and hypocoercivity analysis in the discrete setting. The first one is that Fourier analysis is not tractable anymore for easily proving intermediate functional inequalities (interpolation, embeddings...). The second challenge concerns commutators between discrete operators, which are essential in hypocoercivity estimates.  They contain remainder terms which vanish when the mesh size goes to $0$ but need to be dealt with in order to close estimates.
    
     { Let us mention that the extension of our discretization strategy  to the multi-dimensional case $d>1$ is not trivial and not covered by the present paper. The main difficulty is to have a satisfactory discretization of the fractional Laplace operator (\emph{e.g.} in terms of convergence properties). While there are methods adapted to extended Dirichlet boundary conditions \cite{siwei2019}, its discretization for algebraically decaying densities in the whole space is difficult and up to our knowledge, still largely open \cite{huang_2014_discretization, huang2016finite}.}

    Apart from the theoretical analysis of the schemes, we provide several numerical test cases. These simulations illustrate the theoretical results of global stability, long-time behavior, conservation of mass and preservation of heavy-tails in velocity. Additionally, we illustrate the experimental convergence of the schemes in both the homogeneous and inhomogeneous cases and discuss computational time performances. 
    
    \addtocontents{toc}{\SkipTocEntry}
    \subsection*{Outline} The plan of the paper is as follows. In Section~\ref{sec:pres_num}, the discretization of the Lévy-Fokker-Planck  operator $L_\alpha$ is introduced in an unbounded domain and basic properties are derived. {In particular the consistency, stability and convergence of the scheme are proved.} In Section~\ref{sec:truncation}, the numerical method is adapted to a truncated velocity domain: the discretization which is used in practice is introduced. In Section~\ref{sec:funcanal}, discrete functional analysis results are derived. They include interpolation and embedding inequalities, and non-local Poincaré inequalities. Let us recall that in these results, in the discretized setting, the goal is uniformity of constants with respect to the mesh size. Then, in Section~\ref{sec:stablong}, we show global exponential stability following a (hypo)coercivity strategy. In Section~\ref{sec:numsim}, we perform several numerical simulations which illustrate the structural properties of the numerical method as well as convergence of the schemes. Finally, we gathered in the appendix various necessary but technical results concerning bounds on the equilibrium $\mu_\alpha$ and its discretization.

    \section{Presentation and basic properties of the numerical method in unbounded velocity domain}\label{sec:pres_num}
    In the following, unless explicitly stated otherwise, the velocity space $\mathbb{R}$ is discretized by a regular subdivision $(v_j = jh)_{j\in\mathbb{Z}}$ with $h>0$ a given step size. For a velocity distribution $f:\RR\rightarrow \RR$, $f_j$ denotes an approximation of $f(v_j)$ and with a slight abuse of notation, we write $f = (f_j)_{j\in\ZZ}$. In Section~\ref{sec:discr_FL}, we present the discretization $\Dh:\RR^{\ZZ}\rightarrow\RR^{\ZZ}$ of the fractional Laplacian, such that $(\Dh f)_j$ approaches $-(-\Delta)^{\alpha/2}f(v_j)$. Then, in Section~\ref{sec:discr_FP}, we introduce a discretization of the full fractional Fokker-Planck operator, denoted $\Lah\ :=\ \Gamma_\alpha^h  + \Dh$ where $\Gamma_\alpha^h $ is the discretization of the drift operator $\nabla_v\cdot(v \, \cdot)$. Finally, in Section~\ref{sec:full_num_scheme}, we write the numerical schemes for \eqref{eq:fracFP} and \eqref{eq:kinfracFP}.

    \subsection{Discretization of the fractional Laplacian}\label{sec:discr_FL}
    
    In order to discretize the fractional Laplacian operator in dimension one, we follow the finite difference-quadrature approach of Huang and Oberman~\cite{huang_2014_discretization}. This method preserves the convolution structure of \eqref{eq:fraclap_int1}, and therefore { the properties of monotony} of the operator, and it has a theoretical $O(h^{3-\alpha})$ accuracy.  Let us briefly recall the method here and derive some complementary results.
    
    The integral formulation of the fractional Laplace operator given in~\eqref{eq:fraclap_int1} may be symmetrized and split into the sum of a singular part and a tail part, respectively
    \begin{multline*}
        (-\Delta_v)^{\alpha/2}f(v_j)\ =\ C_{1,\alpha}\,\int_{0}^h\,\frac{f(v_j+w)+f(v_j-w)-2f(v_j)}{w^{1+\alpha}}\,\mathrm{d}w\\ \,+\, C_{1,\alpha}\,\int_{h}^{\infty}\,\frac{f(v_j+w)+f(v_j-w)-2f(v_j)}{w^{1+\alpha}}\,\mathrm{d}w
    \end{multline*}
    where we recall that $C_{1,\alpha} = 2^{\alpha} \Gamma((\alpha+1)/2)/(\pi^{\frac12} |\Gamma(-\alpha/2)|)$. The singular part is approached by using that $f(v_j+w)+f(v_j-w)-2f(v_j)\approx w^2 (f_{j+1}+f_{j-1}-2f_j)/h^2$  for $w\in[0,h]$ and integrating in $w$. The tail part is approached by replacing $f$ with a piecewise quadratic interpolation  of the values $(f_j)_{j\in\ZZ}$, and again integrating (explicitly) in $w$. Altogether (see \cite{huang_2014_discretization} for details), the discrete fractional Laplace operator is given by
    \begin{equation}\label{eq:def_discrLap}
        (\Dh f)_j\ =\  \sum_{k =1}^\infty\beta^h_{k}\,(f_{j+k}+f_{j-k}-2f_j)\,h\ =\ \sum_{k\in\ZZ}\beta^h_{k}\,(f_{j-k}-f_j)\,h\,,
    \end{equation}
    with weights 
    \begin{equation}\label{eq:def_beta}
        \beta^h_{k}:=
        \frac{C_{1,\alpha}}{h^{1+\alpha}}
        \left\{ 
        \begin{array}{l l }
            \ds \frac{1}{2-\alpha}-\varphi_\alpha''(1) - \frac{\varphi_\alpha'(3)+3\varphi_\alpha'(1)}{2}+ \varphi_\alpha(3)-\varphi_\alpha(1) & \mathrm{\;if\;}   k=  1\,, \\[.75em]
            \ds 2\left[ \varphi_\alpha'(k+1)+\varphi_\alpha'(k-1)-\varphi_\alpha(k+1)+\varphi_\alpha(k-1) \right] &\mathrm{\; if\;}  k= 2, 4, 6, \dots\,,\\[.75em]
            \ds -\frac{\varphi_\alpha'(k+2)+6\varphi_\alpha'(k)+\varphi_\alpha'(k-2)}{2}+\varphi_\alpha(k+2)-\varphi_\alpha(k-2) & \mathrm{\; if\;}  k=3, 5, 7, \dots\,,\\
            \ds \beta^h_{-k}& \mathrm{\; if\;}  k<0\,.
        \end{array}
        \right.
    \end{equation}
    where 
    \[
    \varphi_\alpha(t):=\left\{ 
    \begin{array}{l c l}
        \ds\frac{t^{2-\alpha}}{(2-\alpha)(\alpha-1){\alpha}} & \mathrm{\;if\;} & \alpha\neq 1,  \\[.75em]
        \ds t-t\ln(t) &  \mathrm{\;if\;} & \alpha=1.
    \end{array}
    \right.
    \]
    \begin{rema}
        The  approximation of the singular part of the integral appears in the first term in the expression of $\beta_{\pm1}^h$. Observe also that the value assigned to $\beta^h_{0}$ is arbitrary as it does not appear in~\eqref{eq:def_discrLap}.
    \end{rema}
    We have the following estimates on the coefficients $\beta_{k}^h$.
    \begin{lem} \label{lem:bounds_beta}
        There exist positive constants $b_{\alpha}$ and $B_{\alpha}$ depending only on $\alpha\in(0,2)$ such that 
        \begin{equation}\label{eq:bounds_beta}
            \frac{b_{\alpha}}{|hk|^{1+\alpha}} \ \le\  \beta_k^h \ \le\  \frac{B_{\alpha}}{|hk|^{1+\alpha}}\,,\quad\forall k \in \ZZ\setminus\{0\}\,.
        \end{equation}
    \end{lem}
    \begin{proof}
        We only deal with the case of $k>0$. 
        For $k=1$, a direct computation yields that 
        \[
        \beta_{1}^h=\frac{C_{1,\alpha}}{2{\alpha}(2-\alpha) h^{1+\alpha}}\,\left[{ 8}+(4+\alpha)\,(3^{1-\alpha}-1)(\alpha-1)^{-1}\right]\,,\quad \text{if }\alpha\neq1\,,
        \]
        and 
        \[
        \beta_{1}^h=\frac{C_{1,1}}{2h^{2}}\,\left(8-5\ln(3)\right)\,,\quad \text{if }\alpha=1\,.
        \] 
        Then for $k \ge 2$, we observe that
        \begin{equation} \label{eq:defbeta_even}
            \beta_{k}^h\ =\ \frac{C_{1,\alpha}}{h^{1+\alpha}}\int_{-1}^1\,(1-t^2)\,\varphi_\alpha^{(3)}(k+t)\,\dd t\,,\quad\text{ if }k\text{ is even}\,,
        \end{equation}
        and 
        \begin{equation} \label{eq:defbeta_odd}
            \beta_{k}^h\ =\ \frac{C_{1,\alpha}}{2h^{1+\alpha}}\int_{-2}^2\,(t^2-3|t|+2)\,\varphi_\alpha^{(3)}(k+t)\dd t\,,\quad\text{ if }k\text{ is odd}\,.
        \end{equation}
        From there, since $\varphi_\alpha^{(3)}(t) = t^{-1-\alpha}$, the upper bounds in the even and odd cases and the lower bound in the even case are easily derived from \eqref{eq:defbeta_even} and \eqref{eq:defbeta_odd}. For the last bound we use that for $k>2$ odd 
        \[
        \frac{2(hk)^{1+\alpha}}{C_{1,\alpha}}\beta_{k}^h\geq \frac{1}{(1-2/k)^{1+\alpha}}\int_{-2}^{-1}\,(t^2-3|t|+2)\,\dd t + \frac{1}{(1+1/k)^{1+\alpha}}\int_{-1}^{2}\,(t^2-3|t|+2)\,\dd t = \psi_\alpha(1/k)
        \]
        with $\psi_\alpha(k^{-1}) = \tfrac32(1+k^{-1})^{-1-\alpha}-\tfrac16(1-2k^{-1})^{-1-\alpha}$. Clearly, $\psi_\alpha(k^{-1})\geq\psi_\alpha(1/5)$ when $k\geq5$ and the right-hand side is positive, uniformly in $\alpha\in(0,2)$. In the last case $k=3$, one has $2(hk)^{1+\alpha}\beta_{3}^h/C_{1,\alpha} = (15^{-\alpha} (4 (-10 3^\alpha + 9 5^\alpha + 15^\alpha) - (5 3^\alpha + 18 5^\alpha + 15^\alpha) \alpha))/((\alpha-2) (\alpha-1) \alpha)$. The right-side is bounded from below by $\log(5)-8/5>0$ for $\alpha\in(0,2)$.
    \end{proof}
    
    \begin{rema}[Additional properties]
        Let us state here some additional properties of the operator $\Dh$.  The discrete fractional Laplacian satisfies the conservation of mass, namely
        \begin{equation}\label{eq:mass_discrLap}
            \sum_{j\in\ZZ}(\Dh u)_j\ =\ 0\,.
        \end{equation}
        Moreover, it is self-adjoint in the space of square summable sequences, namely
        \begin{equation}\label{eq:sym_discrLap}
            \sum_{j\in\ZZ}(\Dh u)_jv_j\ =\ \sum_{j\in\ZZ}(\Dh v)_ju_j\,.
        \end{equation}
        The discretization of the fractional Laplace operator is consistent with the usual centered finite difference approximation of the Laplacian, namely, { using that $C_{1,\alpha} \sim_{\alpha \to 2^-} 2-\alpha$, one can prove that}
        \begin{equation}\label{eq:lim_coeff_alpha2}
            \lim_{\alpha\to2^-}h\,\beta_k^h\ =\ \left\{\begin{array}{ll}{h^{-2}}&\text{if  }{|k|}=1\\0&\text{otherwise }\end{array}\right.
        \end{equation} 
        and thus one recovers the usual discrete Laplacian at the limit
        \begin{equation}\label{eq:lim_oper_alpha2}
            \lim_{\alpha\to2^-}(\Dh u)_j\ =\ \frac{u_{j+1}+u_{j-1}-2u_j}{h^2}\,,
        \end{equation}
        for all $j\in\ZZ$. Finally it is consistent at order $3-\alpha$. When $h\to0$, one has for any $u\in\mathcal{C}_b^4(\RR)$ that
        \begin{equation}\label{eq:conv_discrfracLap}
            \sup_{j\in\ZZ}\left|-(-\Delta)^{\alpha/2}u(hj)-(\Dh u)_j\right|\ \leq\ K_\alpha\,\|u\|_{\mathcal{C}_b^4(\RR)}\,h^{3-\alpha}\,,
        \end{equation}
        with $K_\alpha$ a positive constant depending only on $\alpha$. { A careful examination of the consistency analysis in \cite[Section 2]{huang_2014_discretization} shows that a more precise estimate is $O(h^{4-\alpha} + C_{1,\alpha} h^{3-\alpha})$ where the $O(h^{4-\alpha})$ comes from the singular part of the integral and $O(C_{1,\alpha} h^{3-\alpha})$ is the tail part of the integral and one has $C_{1,\alpha}\to 0$ when $\alpha\to 2^-$. This explains how one recovers the second order accuracy for the classical Laplacian at the limit $\alpha\to 2^-$. Let us mention that more recently, methods have been proposed \cite{Duo2018} to improve the convergence to $h^{2}$ uniformly in $\alpha$.}
    \end{rema}

    \subsection{Discretization of the Lévy-Fokker-Planck operator}\label{sec:discr_FP}
    
    We now turn to the discretization of the full non-local fractional Fokker-Planck operator $L_\alpha$. It is discretized as follows
    \begin{equation}\label{def_LFPdiscr}
        \Lah\ =\ \Gamma_\alpha^h + \Dh
    \end{equation}
    where $\Gamma_\alpha^h$ is the discrete equivalent of $\partial_v(v\, \cdot)$ to be defined. The goal is to define a consistent approximation that preserves exactly the discrete equilibrium $(M_j)_{j\in\ZZ}$ defined by
    \begin{equation}\label{eq:def_discr_equilibrium}
        M_j\ :=\ \mu_\alpha(v_j)\,.
    \end{equation}
    The design relies on the identities 
    $
    \partial_v(v\, f)\ =\ \partial_v\left(v\, \mu_\alpha\,f/\mu_\alpha\right)
    $
    and
    \[
    v\mu_\alpha(v)\ :=\ \frac{1}{2}\int_{-v}^v(-\Delta_w)^{\alpha/2}\mu_\alpha(w)\dd w\,,
    \]
    which is easily obtained using that $L_\alpha\mu_\alpha=0$ and that $\mu_\alpha$ is symmetric. { The  non-local reformulation of $v\mu_\alpha(v)$ will be used in the design of the scheme to allow for the exact preservation of equilibrium. In terms of numerical analysis, the drawback of this choice is that (hypo)coercivity estimates (see Section~\ref{sec:stablong}) in Sobolev norms lead to new technical challenges in the discrete setting and do not rely on the mere adaptation of the continuous case (see \cite{ayi_2019_note}).}

    In order to preserve the divergence structure and conserve mass, the operator $\Gamma_\alpha^h$ is discretized in the finite volume fashion
    \begin{equation} \label{def:Gamma}
        (\Gamma_\alpha^h f)_j  \ :=\  \frac{\mathcal{F}_{j+\frac{1}{2}} - \mathcal{F}_{j-\frac{1}{2}}}{h}
    \end{equation}
    with the numerical flux defined by the centered approximation 
    \begin{equation} \label{def:flux}
        \mathcal{F}_{j+\frac{1}{2}} \ :=\ \pare{VM}_{j+\frac{1}{2}}\,\left(\frac{f_j}{2M_j} + \frac{f_{j+1}}{2M_{j+1}}\right)\,,
    \end{equation} 
    and 
    \begin{equation}\label{def:VMhalf_simple}
        \pare{VM}_{j+\frac{1}{2}} = -\pare{VM}_{-j-\frac{1}{2}} := -\frac{1}{2}\sum_{k=-j}^j(\Dh M)_k\,h\,, \quad\text{for }j\geq0\,.
    \end{equation}
    
    \begin{rema}
        From the conservation of mass for the discrete fractional Laplacian \eqref{eq:mass_discrLap} and symmetries, one can derive other formulas for $(VM)$ which will be useful for the analysis. In particular for any odd $m \in \ZZ$,
        \begin{equation} \label{def:VM}
            \pare{VM}_{j+\frac{m}{2}}\ =\ \frac{1}{2} \sum_{k \in \mathbb{Z}}\mathrm{sgn}\left(k-(j+\tfrac{m}{2})\right)\,(\Dh M)_k\,h
        \end{equation}
        where $\mathrm{sgn}(\cdot)$ is the sign function. Using~\eqref{eq:def_discrLap}, it can for example be rewritten as
        \[
        \pare{VM}_{j+\frac{m}{2}}\ =\ \frac{1}{4} \sum_{k \in \mathbb{Z}}\sum_{\ell \in \ZZ} \mathrm{sgn}\left(k-(j+\tfrac{m}{2})\right)\,  {\beta_\ell^h}(M_{k+\ell}+M_{k-\ell}-2M_k)\,h^2 \, ,
        \]
 see Appendix~\ref{subsec:BoundsOnTheDiscreteVM} for more details. 
    \end{rema}
    
    The following properties are direct consequences of the definition of the operators $\Dh$ in \eqref{eq:def_discrLap}-\eqref{eq:def_beta} and  $\Lah$ in~\eqref{def_LFPdiscr}-\eqref{def:VMhalf_simple}.
    \begin{lem}[Basic properties] \label{lem:loceq}
        The operator $\Lah$ satisfies the following properties. 
        \begin{itemize}
            \item[\emph{i)}] Mass conservation: for any suitably summable sequence $u$, one has 
            \begin{equation}\label{eq:mass_conserv}
                \sum_{j\in\ZZ}(\Lah u)_j\,h\ =\ 0\,.
            \end{equation}
            \item[\emph{ii)}] Preservation of local equilibrium:
            \begin{equation}\label{eq:preserv_eq}
                (\Lah M)_j\ =\ 0\,,\quad\forall j\in\ZZ \, .
            \end{equation}
            \item[\emph{iii)}] Consistency: for any $u\in\mathcal{C}_b^4(\RR)$, one has that
            \begin{equation}\label{eq:conv_discrLFP}
                \sup_{j\in\ZZ}|(L_\alpha u)(hj)-(\Lah u)_j|\ \leq\ K_\alpha\,\|u\|_{\mathcal{C}_b^4(\RR)}\, {h^{\min(3-\alpha,2)}}\,,
            \end{equation}
            for some $K_\alpha>0$.
        \end{itemize}
    \end{lem}
    \begin{proof}
        The first property follows from \eqref{eq:mass_discrLap} and \eqref{def:Gamma}. The second property is readily obtained using that $\pare{VM}_{j+\frac{1}{2}} - \pare{VM}_{j-\frac{1}{2}} = -(\Dh M)_jh$. The last property follows from \eqref{eq:conv_discrfracLap} and the centered discretization of the drift \eqref{def:Gamma}-\eqref{def:VMhalf_simple}, yielding $O(h^2)$ part of the estimate.
    \end{proof}

    The discrete Lévy-Fokker-Planck operator can be split into the following symmetric / skew-symmetric decomposition.
    \begin{prop}[Bilinear decomposition]
        Given $(f_j)_{j \in \mathbb{Z}}$ and $(g_j)_{j \in \mathbb{Z}}$, we introduce $F_j := f_j/M_j$ and $G_j := g_j/M_j$ for any $j\in\ZZ$. One has the following decomposition
        \begin{equation} \label{eq:Lalphabilinear}
            - \sum_{j\in\ZZ} (\Lah f)_j\,g_j\,M_j^{-1} h\ =\ \Symh(f,g) \,+\, \Antsymh(f,g)
        \end{equation}
        where $\Symh$ and $\Antsymh$ are respectively symmetric and skew-symmetric bilinear forms defined by 
        \begin{equation}\label{eq:def_sym_discr}
            \Symh (f,g) := \frac{1}{2} \sum_{(j,k) \in \ZZ^2} \beta_k^h {(F_j - F_{j+k})\,(G_j - G_{j+k})}\,M_j\,h^2\,,
        \end{equation}
        and 
        \begin{equation}\label{eq:def_skewsym_discr}
            \Antsymh (f,g) := - \frac{1}{2} \sum_{(j,k) \in \ZZ^2} {\beta_k^h} {(F_j G_{j+k}- G_j  F_{j+k})}\,M_j\,h^2 - {1 \over 2} \sum_{j \in \mathbb{Z}}  {\pare{VM}_{j+\frac{1}{2}}}  (F_{j+1} G_j - F_j G_{j+1}).
        \end{equation}
    \end{prop}
    \begin{proof}
        Observe that using the definition of $\Dh$ in~\eqref{eq:def_discrLap}, \eqref{eq:def_sym_discr} and \eqref{eq:def_skewsym_discr} rewrite 
        \[
        \Symh (f,g)\ =\  \frac12\sum_{j \in \mathbb{Z}} \left[(\Dh (FG))_j M_j  \,-\,   f_j (\Dh G)_j - g_j (\Dh F)_j\right] h\,,
        \]
        and
        \[
        \Antsymh (f,g)\ =\ -\frac12\sum_{j \in \mathbb{Z}} \left(f_j (\Dh G)_j -   g_j (\Dh F)_j\right)h - \frac{1}{2} \sum_{j \in \mathbb{Z}}  (VM)_{j+\frac{1}{2}}  (F_{j+1}G_j - F_j G_{j+1})\,.
        \]
        Therefore, with a change of index in the last term we get
        \begin{multline*}
            \Symh (f,g) + \Antsymh(f,g)\ =\ \frac12\sum_{j \in \mathbb{Z}} (\Dh (FG))_j M_j\,h -  \sum_{j \in \mathbb{Z}}  f_j (\Dh G)_j h\\
            - \frac{1}{2}  \sum_{j \in \mathbb{Z}}  (VM)_{j+\frac{1}{2}}  F_{j+1}G_j + \frac{1}{2}  \sum_{j \in \mathbb{Z}}  (VM)_{j-\frac{1}{2}} F_{j-1} G_{j} \, .
        \end{multline*}
        Thanks to the symmetry property \eqref{eq:sym_discrLap}, the first term of the right-hand side rewrites 
        \[
        \frac12\sum_{j \in \mathbb{Z}} (\Dh M)_j F_j\,G_j\,h\ =\ {-}\frac12\sum_{j \in \mathbb{Z}} \left((VM)_{j+\frac{1}{2}} - (VM)_{j-\frac{1}{2}}\right)\,F_j\,G_j
        \]
        By combining this equality with the previous one, we get the claim.
    \end{proof}
    
    \begin{cor} \label{cor:kerim}
        The discrete Lévy-Fokker-Planck operator $\Lah$, as an operator on $\{(f_j)_j\,|\,\sum_jf_j^2M_j^{-1}<+\infty\}$,  has the following properties:
        \begin{itemize}
            \item[\emph{i)}] $\mathrm{Ker}(\Lah)\ =\ \mathrm{span}\{(M_j)_{j\in\ZZ}\}$\,,
            \item[\emph{ii)}] $\mathrm{Im}(\Lah)\ \subset\ \{(g_j)_j\,|\,\sum_j\,g_j= 0\}$\,.
        \end{itemize}
        
    \end{cor}
    \begin{proof}
        For the first property just observe that $\Symh (f,f)=0$ if and only if $f\in\mathrm{span}\{(M_j)_{j\in\ZZ}\}$. The second property follows from the computation $\Symh (f,M) = \Antsymh(f,M) = 0$.
    \end{proof}

    \begin{lem} In the limit $\alpha\to2^-$, one has that
        \[
        \lim_{\alpha\to2^-}(VM)_{j+1/2}\ =\ \frac{M_j - M_{j+1}}{h}\,,
        \]
        and
        \[
        \lim_{\alpha\to2^-}\Lah f\ =\ \frac{\mathcal{G}_{j+1/2}-\mathcal{G}_{j-1/2}}{h}\,,\quad\text{with}\quad \mathcal{G}_{j+1/2}\ =\ \frac{M_j+M_{j+1}}{2\,h}\,\left(\frac{f_{j+1}}{M_{j+1}}-\frac{f_{j}}{M_{j}}\right)\,.
        \]
    \end{lem}
    \begin{rema}
        We recall that in the case $\alpha=2$, the equilibrium $\mu_2$ is a standard Gaussian. It satisfies the relation $v\mu_2(v) = -\mu_2'(v)$ whose discrete counterpart is given by the first limit in the Lemma. When $\alpha\to2^-$, the scheme degenerates in a conservative finite difference / finite volume scheme which is clearly consistent with the equivalent reformulation
        \[
        L_2f\ =\ \partial_v\left(\mu_2\,\partial_v\left(\frac{f}{\mu_2}\right)\right)\,.
        \]
All the aforementioned good properties stated above (conservation of mass, local equilibrium, ...) still hold at the limit. A particular property is that unlike the fractional case, $L_{2}^h$ is symmetric in its natural Hilbert space. This limit discretization is in the same fashion as the famous Chang-Cooper \cite{ChangCooper:1970, BuetDellacherie:2010}, Il'In~\cite{IlIn1969} and Scharfetter-Gummel discretizations \cite{Scharfetter1969}. It is also close to the discretization adopted in \cite{bessemoulin_2020_hypocoercivity} and where the counterpart of this paper's results were proved for classical Fokker-Planck equations. 
    \end{rema}

    \subsection{Numerical schemes in unbounded velocity domain}\label{sec:full_num_scheme}
    The discrete Lévy-Fokker-Planck operator $L_\alpha^{h}$, acting only on the velocity index $j\in\ZZ$, is defined by \eqref{def_LFPdiscr}-\eqref{def:VMhalf_simple} and \eqref{eq:def_discrLap}-\eqref{eq:def_beta} in the previous section. From there, we can now define the numerical approximation of the homogeneous (in space) Lévy-Fokker-Planck equation $\partial_t f = L_\alpha f$. For a time discretization $t_n = n\Delta t$ with time step $\Delta t>0$ an approximation $f^n_j$ of $f(t_n,v_j)$ is computed by solving the implicit in time scheme 
    \begin{equation}\label{eq:def_scheme_hom}
        \frac{f^{n+1}_{j}-f^{n}_{j}}{\Delta t}\ =\ (L_\alpha^{h} f)_{j}^{n+1}\,,\quad\forall(n,j)\in\NN\times\ZZ\,,
    \end{equation}
    and starts at some given initial data $(f_j^0)_j$.

    In the inhomogeneous case, namely for the kinetic Lévy-Fokker-Planck equation $\partial_t f + v\partial_xf = L_\alpha f$, set in the phase space $\TT\times\RR$, we need some additional discretization parameters. The space and velocity step are $\Delta x = N_x^{-1}$ with $N_x$ an odd positive integer, and $\Delta v$ respectively (instead of $h$). From there we write $t_n = n\Delta t$, $x_i = i\Delta x$ and $v_j = j\Delta v$ for any $(n,i,j)\in\NN\times\ZZ/N_x\ZZ\times\ZZ$. The scheme computes the approximation of $f(t_n,x_i,v_j)$ is denoted $f^n_{i,j}$. It is implicit in time and writes
    \begin{equation}\label{eq:def_scheme_kin}
        \frac{f^{n+1}_{i,j}-f^{n}_{i,j}}{\Delta t}\,+\,(T^{\Delta x}f)_{i,j}^{n+1}\ =\ (L_\alpha^{\Delta v} f)_{i,j}^{n+1}\,,\quad\forall(n,i,j)\in\NN\times\ZZ/N_x\ZZ\times\ZZ\,,
    \end{equation}
    with given initial data $(f_{i,j}^0)_{i,j}$. The discrete transport operator writes
    \begin{equation}\label{eq:def_transport}
        \ds (T^{\Delta x}f)_{i,j}^{n}\ =\  {v_j}\frac{f^{n}_{i+1,j} - f^{n}_{i-1,j}}{2\Delta x}\,.
    \end{equation}
    
    Let us end by stating some properties of the scheme.
    
    \begin{prop}\label{prop:elementary}
        The scheme \eqref{eq:def_scheme_kin} satisfies the following properties.
        \begin{itemize}
            \item[\emph{i)}] A solution $(f^{n}_{i,j})_{i,j,n}$ is a stationary solution, \emph{i.e.} $f^{n+1}_{i,j} = f^{n}_{i,j}$ for all $n\geq0$, $i\in\ZZ/N_x\ZZ$ and $j\in\ZZ$, if and only if for some constant $C\in\RR$,
            \[
            f^{n}_{i,j}\ =\ CM_j\ =\ C\mu_\alpha(v_j)\,,\quad\forall(n,i,j)\in\NN\times\ZZ/N_x\ZZ\times\ZZ\,.
            \]
            \item[\emph{ii)}] The total mass is preserved, namely for any suitably summable initial data $(f^{0}_{i,j})_{i,j}$
            \[
            \sum_{i\in\ZZ/N_x\ZZ}\sum_{j\in\ZZ} f_{i,j}^{n}\,\Delta v\,\Delta x \ =\  \sum_{i\in\ZZ/N_x\ZZ}\sum_{j\in\ZZ} f_{i,j}^{0}\,\Delta v\,\Delta x\,,\quad \forall n\in\NN.
            \]
            \item[\emph{iii)}] The solution satisfies the following global stability estimate
            \[
            \sum_{i\in\ZZ/N_x\ZZ}\sum_{j\in\ZZ} (f_{i,j}^{n})^2M_j^{-1}\,\Delta v\,\Delta x\ \leq\ \sum_{i\in\ZZ/N_x\ZZ}\sum_{j\in\ZZ} (f_{i,j}^{0})^2M_{j}^{-1}\,\Delta v\,\Delta x\,,\quad \forall n\in\NN.
            \]
            {\item[\emph{iv)}]  Let $\mathfrak{h}=(\Delta t,\dx,\dv)$, and $f_{\mathfrak{h}}(t,x,v)= f^n_{i,j}$, if $t\in [t_n,t_{n+1})$, $x\in[x_i-\dx/2,x_i+\dx/2)$, and $v\in [v_j-\dv/2,v_j+\dv/2)$. Then, for all $T\ge 0$, $(f_{\mathfrak{h}})_{\mathfrak{h}}$ converges weakly in $L^2(0,T; L^2(\mu_\alpha^{-1}\mathrm{d}x\mathrm{d}v))$ when $\mathfrak{h}\to 0$, and its limit is the weak solution of \eqref{eq:kinfracFP} in $L^2(0,T; L^2(\mu_\alpha^{-1}\mathrm{d}x\mathrm{d}v))$. }
        \end{itemize}
    \end{prop}
    \begin{proof}
        If one multiplies the scheme by $f^{n+1}_{i,j}M_j^{-1}$ and sums over all indices, one obtains
        \[
        \sum_{i\in\ZZ/N_x\ZZ}\sum_{j\in\ZZ} (f_{i,j}^{n})^2M_j^{-1}\,\Delta v\,\Delta x + \sum_{k=1}^n \sum_{i\in\ZZ/N_x\ZZ}\mathcal{S}_\alpha^{\Delta v} (f^k_i,f^k_i)\,\Delta t\ =\ \sum_{i\in\ZZ/N_x\ZZ}\sum_{j\in\ZZ} (f_{i,j}^{0})^2M_{j}^{-1}\,\Delta v\,\Delta x\,.
        \]
        From there, the stability estimate in iii) readily follows. Concerning the first item, observe that if $f^{n}_{i,j}$ is stationary then by the previous estimate $\mathcal{S}_\alpha^{\Delta v} (f^k_i,f^k_i)=0$ for all $k\in\NN$ and $i\in\ZZ/N_x\ZZ$. As a consequence of \eqref{eq:def_sym_discr}, the solution is of the form $f^n_{i,j} = \rho_iM_j$. Plugging this back in the steady version of \eqref{eq:def_scheme_kin} yields $\rho_{i+1} - \rho_{i-1} = 0$ for all $i$, therefore $\rho_i = C$ for all $i$, as $N_x$ is odd. It proves one implication in i) and the converse is trivial. Finally the conservation of mass in ii) is obtained by summing \eqref{eq:def_scheme_kin}  over $i$ and $j$.
        
        {For the sake of simplicity, the point iv) is only proved in the homogeneous case, we thus denote $\mathfrak{h}=(\dt, h)$. The extension to the inhomogeneous case is straightforward.
        Thanks to the global stability estimate in iii), $(f_{\mathfrak{h}})_{\mathfrak{h}}$ is bounded in $L^2(0,T; L^2(\mu_\alpha^{-1}\mathrm{d}v))$. Hence, it admits a subsequence (that we still denote $(f_{\mathfrak{h}})_{\mathfrak{h}}$) which converges weakly to some $f_*\in L^2(0,T; L^2(\mu_\alpha^{-1}\mathrm{d}v))$. We show in what follows that $f_*$ is a weak solution of \eqref{eq:fracFP} in $L^2(0,T; L^2(\mu_\alpha^{-1}\mathrm{d}v))$. The uniqueness of such a weak solution is a consequence of the linearity of \eqref{eq:fracFP} and of Theorem \ref{t:maincont}. 
        \\ Let $\varphi\in\mathcal{C}^\infty_0([0,+\infty)\times\mathbb{R})$ be a compactly supported test function, with support in $[0,T)\times[-V,V]$. A piecewise linear function $\varphi_{\mathfrak{h}}$ is introduced, such that for all $t\in[t_n,t_{n+1})$ and all $v\in[v_j-h/2,v_j+h/2)$,
        \[
        \varphi_{\mathfrak{h}}(t,v) = \varphi(t_n,v_j)+ \frac{\varphi(t_n,v_{j+1})-\varphi(t_n,v_j)}{h}(v-v_j)\,.
        \]
        Remark that $\varphi_{\mathfrak{h}}$ depends only on $v$ and not on $t$ on $[t_n,t_{n+1})\times[v_j-h/2,v_j+h/2)$, that it coincides with $\varphi(t_n,v_j)$ at $(t_n,v_j)$, and that it is linear in $v$ on $[v_j-h/2,v_j+h/2)$.
        Thanks to this construction, one has for all $n\in\mathbb{N}$, $j\in\mathbb{Z}$,
        \[
        \frac{1}{\dt\, h} \int_{t_n}^{t_{n+1}}\int_{v_j-h/2}^{v_j+h/2} \varphi_{\mathfrak{h}}(t,v)\mathrm{d}v\mathrm{d}t = \varphi(t_n,v_j)\,,
        \]
        and $|\varphi-\varphi_{\mathfrak{h}}|\lesssim (h^2+\dt) \mathds{1}_{[0,T]\times[-V,V]}$, where the constant in the inequality depends only on $\|\partial_v^2 \varphi\|_\infty$, and $\|\partial_t \varphi\|_\infty$. Note also that, 
        \begin{equation} 
        \label{eq:LocalErrApprox}
        \forall n\in\mathbb{N},\quad|\varphi(t_n,\cdot)-\varphi_{\mathfrak{h}}(t_n,\cdot)|\lesssim h^2 \mathds{1}_{[0,T]\times[-V,V]}(t_n,\cdot)\,,
        \end{equation}
        where the constant depends only on $\|\partial_v^2 \varphi(t_n,\cdot)\|_\infty$. The discrete weak formulation of \eqref{eq:def_scheme_hom} is obtained by multiplying \eqref{eq:def_scheme_hom} by $\varphi_{\mathfrak{h}}(t_{n+1},v_j)$, and summing in $n\in\mathbb{N}$ and $j\in\mathbb{Z}$. Denoting $\varphi^n_j=\varphi(t_n,v_j)=\varphi_{\mathfrak{h}}(t_n,v_j)$, it yields
        \begin{multline}
        \sum\limits_{n\in\mathbb{N}} \sum\limits_{j\in\mathbb{Z}} f^{n+1}_j \left( 
        \frac{\varphi^{n+2}_j - \varphi^{n+1}_j}{\dt} + \left(\Lambda_\alpha^h \varphi \right)^{n+1}_j - \frac{(VM)_{j+1/2}}{2M_j} \frac{\varphi^{n+1}_{j+1}-\varphi^{n+1}_j}{h}
        - \frac{(VM)_{j-1/2}}{2M_j} \frac{\varphi^{n+1}_{j}-\varphi^{n+1}_{j-1}}{h}
        \right)
        \\
        + \frac{1}{\dt} \sum\limits_{j\in\mathbb{Z}} f^0_j \varphi^1_j = 0\,, \label{eq:DiscrWeakFormulation}
        \end{multline}
        where
        \begin{align*}
        \dt \,h &\sum\limits_{n\in\mathbb{N}}\sum\limits_{j\in\mathbb{Z}} f^{n+1}_j \frac{\varphi^{n+2}_j - \varphi^{n+1}_j}{\dt} + h \sum\limits_{j\in\mathbb{Z}} f^0_j \varphi^1_j \\&\underset{\mathfrak{h}\to 0}\longrightarrow \int_0^{+\infty} \int_{\mathbb{R}} f_*(t,v)\; \partial_t \varphi(t,v) \mathrm{d}v \mathrm{d}t + \int_{\mathbb{R}} f_*(0,v) \varphi(0,v) \mathrm{d} v\,,
        \end{align*}
        using dominated convergence theorem, the fact that $f_{\mathfrak{h}}$ converges weakly to $f_*$ in $L^2(0,T;L^2(\mu_\alpha^{-1}\mathrm{d}v))$, and that $\varphi$ is smooth.
        Remark now that 
        \begin{align*}
        \dt \,h &\sum\limits_{n\in\mathbb{N}} \sum\limits_{j\in\mathbb{Z}} f^{n+1}_j \left( \Lambda_\alpha^h\varphi\right)^{n+1}_j 
        = -\int_{\dt}^{+\infty}\int_{\mathbb{R}} f_{\mathfrak{h}}(t,v) \left(-\Delta_v\right)^{\alpha/2}\varphi(t,v) \mathrm{d}v\mathrm{d}t  
        \\&\qquad 
        + \int_\dt^{+\infty}\int_{\mathbb{R}} f_{\mathfrak{h}}(t,v) \left[ \left( -\Delta_v\right)^{\alpha/2}\varphi(t,v) + \left(\Lambda_\alpha^h \varphi(t,\cdot)\right)(v) \right]\mathrm{d}v\mathrm{d}t 
        \\&\qquad + \int_{\mathbb{R}} \sum\limits_{n\in\mathbb{N}}f_{\mathfrak{h}}(t_{n+1},v) \left( 
        \Delta t  \left( \Lambda^h_\alpha \varphi(t_{n+1},\cdot) \right)(v)
        -\int_{t_n}^{t_{n+1}} \left( \Lambda^h_\alpha \varphi(t+\Delta t,\cdot) \right)(v) \mathrm{d}t
        \right) \mathrm{d}v
        \\&\qquad
        + \dt\sum\limits_{n\in\mathbb{N}} \int_{\mathbb{R}} f_{\mathfrak{h}}(t_{n+1},v) \left[ \left( \Lambda_\alpha^h \left(\varphi_{\mathfrak{h}}-\varphi\right) \left(t_{n+1},\cdot\right) \right)(v)  \right] \mathrm{d}v\,,
        \end{align*}
        where, for any function $\phi$ decreasing enough at infinity, and for all $v\in\mathbb{R}$
        \begin{equation}
        \label{eq:DiscrLap_continuous}
        \left(\Lambda_\alpha^h\phi\right)(v):=\sum\limits_{k\in\mathbb{Z}} \beta^h_k   \left( \phi(v-kh)-\phi(v)\right)h.
        \end{equation}
        Then, use Cauchy-Schwarz inequality and \eqref{eq:conv_discrfracLap} to get
        \begin{align*}
        &\left| \int_h^{+\infty}\int_{\mathbb{R}} f_{\mathfrak{h}}(t,v) \left[ \left( -\Delta_v\right)^{\alpha/2}\varphi(t,v) + \left(\Lambda_\alpha^h \varphi(t,\cdot)\right)(v) \right]\mathrm{d}v\mathrm{d}t \right| 
        \\
        & \le h^{3-\alpha} K_\alpha \|\varphi\|_{\mathcal{C}_b^4(\mathbb{R}_+\times\mathbb{R})} \|f_{\mathfrak{h}} \|_{L^2(0,T; L^2(\mu_\alpha^{-1}\mathrm{d}v))} \| \mu_\alpha \|_{L^2(0,T; L^2(\mu_\alpha^{-1}\mathrm{d}v))}
        \underset{\mathfrak{h}\to 0}\longrightarrow 0\,.
        \end{align*}
        In the third term, the difference between each integral in time and its evaluation at its right bound is the local quadrature error for a smooth function, hence for all $n\in\mathbb{N}$
        \begin{align*}
        \left|\Delta t  \left( \Lambda^h_\alpha \varphi(t_{n+1},\cdot) \right)(v)
        -\int_{t_n}^{t_{n+1}} \left( \Lambda^h_\alpha \varphi(t+\Delta t,\cdot) \right)(v) \mathrm{d}t \right| \lesssim \Delta t^2 \mathds{1}_{[0,T]}(t_{n+1})\,,
        \end{align*}
        since $\varphi$ is compactly supported in time. Cauchy-Schwarz inequality then yields 
        \begin{align*}
        &\left|  \int_{\mathbb{R}} \sum\limits_{n\in\mathbb{N}}f_{\mathfrak{h}}(t_{n+1},v) \left( 
        \Delta t  \left( \Lambda^h_\alpha \varphi(t_{n+1},\cdot) \right)(v)
        -\int_{t_n}^{t_{n+1}} \left( \Lambda^h_\alpha \varphi(t+\Delta t,\cdot) \right)(v) \mathrm{d}t
        \right) \mathrm{d}v   \right| 
        \\&\quad \le \Delta t \left\|f_{\mathfrak{h}}\right\|_{L^2(0,T;L^2(\mu_\alpha^{-1}\mathrm{d}v))} \left\|\mu_\alpha\right\|_{L^2(0,T,L^2(\mu_\alpha^{-1}\mathrm{d}v))} \underset{\mathfrak{h}\to 0}\longrightarrow 0\,.
        \end{align*}
        The fourth term is estimated using \eqref{eq:DiscrLap_continuous}, \eqref{eq:bounds_beta} and \eqref{eq:LocalErrApprox}. Let $n\ge 1$, one has
        \begin{align*}
        \left| \left( \Lambda_\alpha^h \left(\varphi_{\mathfrak{h}}-\varphi\right) \left(t_{n+1},\cdot\right) \right)(v)  \right| \lesssim h^2 \mathds{1}_{[0,T]}(t_{n+1}) \sum\limits_{k\in\mathbb{Z}} \beta^h_k h \lesssim h^{2-\alpha}\mathds{1}_{[0,T]}(t_{n+1})\,,
        \end{align*}
        so that
        \begin{align*}
        \dt\sum\limits_{n\in\mathbb{N}} \int_{\mathbb{R}} f_{\mathfrak{h}}(t_{n+1},v) \left[ \left( \Lambda_\alpha^h \left(\varphi_{\mathfrak{h}}-\varphi\right) \left(t_{n+1},\cdot\right) \right)(v)  \right] \mathrm{d}v\underset{\mathfrak{h}\to 0}\longrightarrow 0\,,
        \end{align*}
        and hence dominated convergence theorem yields
        \begin{align*}
        \dt \, h &\sum\limits_{n\in\mathbb{N}} \sum\limits_{j\in\mathbb{Z}} f^{n+1}_j \left( \Lambda_\alpha^h\varphi\right)^{n+1}_j   \underset{\mathfrak{h}\to 0}\longrightarrow -\int_0^{+\infty}\int_{\mathbb{R}} f_*(t,v) \left(-\Delta_v\right)^{\alpha/2}\varphi(t,v) \mathrm{d}v\mathrm{d}t\,.
        \end{align*}
        Coming back to \eqref{eq:DiscrWeakFormulation}, note that the two terms with $(VM)_{j\pm1/2}$ can be dealt with similarly. It can be reformulated as
        \begin{align*}
        &\dt \, h\sum\limits_{n\in\mathbb{N}} \sum\limits_{j\in\mathbb{Z}} f^{n+1}_j \frac{(VM)_{j+1/2}}{2M_j} \frac{\varphi^{n+1}_{j+1}-\varphi^{n+1}_j}{h}  
        =\dt \, h\sum\limits_{n\in\mathbb{N}} \sum\limits_{j\in\mathbb{Z}} f^{n+1}_j \frac{v_{j+1/2}\mu_\alpha(v_{j+1/2})}{2\mu_\alpha(v_j)} \frac{\varphi^{n+1}_{j+1}-\varphi^{n+1}_j}{h} 
        \\&\qquad \qquad
        + \dt \, h\sum\limits_{n\in\mathbb{N}} \sum\limits_{j\in\mathbb{Z}} f^{n+1}_j \frac{(VM)_{j+1/2}-v_{j+1/2}\mu_\alpha(v_{j+1/2})}{2\mu_\alpha(v_j)} \frac{\varphi^{n+1}_{j+1}-\varphi^{n+1}_j}{h}\,,
        \end{align*}
        where 
        \[
        \dt \, h\sum\limits_{n\in\mathbb{N}} \sum\limits_{j\in\mathbb{Z}} f^{n+1}_j \frac{v_{j+1/2}\mu_\alpha(v_{j+1/2})}{2\mu_\alpha(v_j)} \frac{\varphi^{n+1}_{j+1}-\varphi^{n+1}_j}{h} \underset{\mathfrak{h}\to 0}\longrightarrow \frac{1}{2} \int_0^{+\infty} \int_{\mathbb{R}} f_*(t,v) v \partial_v \varphi(t,v) \mathrm{d}v\mathrm{d}t,
        \]
        and the last term vanishes when $\mathfrak{h}\to 0$. Indeed, let us define $M_{\mathfrak{h}}$ by
        \[
        \forall v\in[v_j-h/2,v_j+h/2), \quad M_{\mathfrak{h}}(v):=M_j+ \frac{M_{j+1}-M_j}{h}(v-v_j)\,,
        \]
        so that the definition of $(VM)_{j+1/2}$ in \eqref{def:VMhalf_simple} yields
        \begin{align*}
        &2\left| (VM)_{j+1/2} - v_{j+1/2}\mu_\alpha(v_{j+1/2})\right|
        \\&\quad \le 
        \left| \int_{-v_{j+1/2}}^{v_{j+1/2}} \left( (-\Delta_w)^{\alpha/2} \mu_\alpha(w)+ \left(\Lambda^h_\alpha \mu_\alpha\right)(w) \right)\mathrm{d}w \right|
        + \left|
        \int_{-v_{j+1/2}}^{v_{j+1/2}}\left( \Lambda^h_\alpha (M_{\mathfrak{h}} - \mu_\alpha) \right)(w) \mathrm{d}{w}
        \right|
        \\&\quad \lesssim |v_j|( h^{3-\alpha}+ h^{2-\alpha})\,,
        \end{align*}
        where the last estimate was obtained thanks to \eqref{eq:conv_discrfracLap} and \eqref{eq:LocalErrApprox}. We conclude using Cauchy-Schwarz inequality 
        \begin{align*}
        &\left| \dt \, h\sum\limits_{n\in\mathbb{N}} \sum\limits_{j\in\mathbb{Z}} f^{n+1}_j \frac{(VM)_{j+1/2}-v_{j+1/2}\mu_\alpha(v_{j+1/2})}{2\mu_\alpha(v_j)} \frac{\varphi^{n+1}_{j+1}-\varphi^{n+1}_j}{h}\right|
        \\&\quad \lesssim h^{2-\alpha}\left( \Delta t \, h \sum\limits_{n\in\mathbb{N}} \sum\limits_{j\in\mathbb{Z}} \frac{|f^{n+1}_j|^2}{4M_j} \right)^{1/2}
        \left( \dt\,  h \sum\limits_{n\in\mathbb{N}} \sum\limits_{j\in\mathbb{Z}} |v_j|^2 \left(\frac{\varphi^{n+1}_{j+1}-\varphi^{n+1}_j}{h}\right)^2 \frac{1}{M_j} \right)^{1/2}
        \\&\quad \lesssim h^{2-\alpha} \underset{\mathfrak{h}\to 0}\longrightarrow 0\,,
        \end{align*}
        and $f_*$ is a weak solution of \eqref{eq:fracFP} in $L^2(0,T; L^2(\mu_\alpha^{-1}\mathrm{d}v))$.
        }
    \end{proof}
    
    \begin{rema}Observe that as a corollary of the results of Proposition~\ref{prop:elementary}, one can derive equivalent properties in the homogeneous case, that is for the scheme \eqref{eq:def_scheme_hom}.
    \end{rema}

    \section{Truncation of the velocity domain}\label{sec:truncation}
    
    For practical computations, the discrete velocity domain has to be reduced to a finite number of points. Despite the truncation, the discretization  still needs to approximate the Fokker-Planck operator on the whole real line. Moreover, it is desirable to preserve the structural properties of the discretization, such as preservation of mass and equilibrium.
    
    In this section, the domain in velocity is now reduced to a symmetric interval $v\in[-L,L]$ with~$L = Jh$ and $J$ a positive integer. The discrete velocities are
    \[
    v_j = jh\,,\quad j\in\{-J, \dots, J\}\,.
    \]
    Let us introduce a truncated version $\Laht$ of the operator $\Lah$ which satisfies the aforementioned requirements.
    
    \subsection{Truncated discrete fractional Laplacian}
    
    For the truncated version of the fractional Laplacian, we follow the method of Huang and Oberman \cite[Section~5]{huang_2014_discretization}. The singular integral version of $-(-\Delta_v f)^{\alpha/2}(v_j)$ may be decomposed into the sum of three contributions 
    \begin{multline*}
        -(-\Delta_v f)^{\alpha/2}(v_j)\ =\ \underbrace{C_{1,\alpha}\int_{-L_W}^{L_W}(f(v_j-w)-f(v_j))\,\frac{\dd w}{w^{1+\alpha}}}_{\text{(I)}}\\*\,-\,\underbrace{C_{1,\alpha}\int_{|w|>L_W}f(v_j)\,\frac{\dd w}{w^{1+\alpha}}}_{\text{(II)}}
        \,+\,\underbrace{C_{1,\alpha}\int_{|w|>L_W}f(v_j-w)\,\frac{\dd w}{w^{1+\alpha}}}_{\text{(III)}}\,,
    \end{multline*}
    where $L_W = Kh$ is a second truncation threshold which we assume to be such that
    \[
    L_W\geq 2L\,.
    \]
    For technical reason, we also  {assume that $K$ is an odd integer}. From there, each term is discretized in a specific way. For the first term, we just truncate the discretization of Section~\ref{sec:discr_FL}, namely
    \[
    \text{(I)}\,\approx\, \sum_{k=-K}^K(f_{j-k}-f_j)\,\beta_k^h\,h\,,
    \]
    where the coefficients $\beta_k^h$ are given by \eqref{eq:def_beta} for $k = -K+1, \dots, K-1$ and 
    \begin{equation}\label{eq:def_beta_last}
        \beta_{\pm K}\ =\ \frac{C_{1,\alpha}}{2h^{1+\alpha}}\left(2\varphi_\alpha''(K)+2\varphi_\alpha(K)-2\varphi_\alpha(K-2)-\varphi_\alpha'(K-2)-3\varphi_\alpha'(K)\right)\,.
    \end{equation}
    The integral in the second term can be computed exactly and this term is thus approximated by 
    \[
    \text{(II)}\ \approx\ \frac{2C_{1,\alpha}}{L_W^\alpha\,\alpha}\,f_j\,.
    \]
    For the third term, the function $f$ is evaluated outside of the truncation domain $[-L,L]$, since $L_W\geq2L$. As solutions of the fractional Fokker-Planck equation develop algebraic tails with known exponent, we approximate the function $f$ outside the domain {by an algebraically decaying function $f(\pm v) = f_{\pm J}(L/v)^{\gamma}$}, for $v\geq L$. It yields the approximation $\text{(III)} = \text{(IIIa)} + \text{(IIIb)}$ where {
    \[
    \text{(IIIa)}\ \approx\ f_J\,\frac{L^{\gamma}C_{1,\alpha}}{L_W^{\gamma+\alpha}(\gamma+\alpha)}\,{}_2F_1\left(\gamma,\gamma+\alpha;1+\gamma+\alpha;-\frac{j}{K}\right)\,,
    \]
    and
    \[
    \text{(IIIb)}\ \approx\ f_{-J}\,\frac{L^{\gamma}C_{1,\alpha}}{L_W^{\gamma+\alpha}(\gamma+\alpha)}\,{}_2F_1\left(\gamma,\gamma+\alpha;1+\gamma+\alpha;\frac{j}{K}\right)\,,
    \]}
    with ${}_2F_1$ the Gauss hypergeometric function. Observe that in the approximation  of (I), one needs $f_k$ outside of the domain $\{-J,\dots,J\}$. As for the approximation of the third term, one uses an algebraically decaying extension, that is { $f_{\pm k} = f_{\pm J}(J/k)^{\gamma}$}, for $k\geq J$. Altogether, the quantity $-(-\Delta_v f)^{\alpha/2}(v_j)$ is approximated by $(\Dht f)_j$ where the matrix $\Lambda_\alpha^{h,J,K}$, approximating the fractional Laplace operator on the domain $[-Jh,Jh]$ with integral representation truncated on $[-Kh, Kh]$ (with $K\geq2J)$, is given by

    \begin{equation}\label{eq:def_discrLap_trunc1}
        \Dht\ =\ -\left(\frac{2 C_{1,\alpha}}{\alpha(Kh)^\alpha} + \sum_{k = -K}^K\beta_k^h\,h\right)I + P_\alpha^{h,J,K}+ Q_\alpha^{h,J,K}\,,
    \end{equation}
    where $I$ is the identity matrix and for $j,k\in\{-J,\dots,J\}$
    \begin{equation}\label{eq:def_discrLap_trunc2}
        (P_\alpha^{h,J,K})_{jk}\ =\ \left\{
        \begin{array}{ll}
            \beta_{j-k}^h&\text{if }|k|<J\\[.5em]
            \ds\sum_{l = J}^{j+K}\beta_{j-l}^h    {\left(\frac{J}{l}\right)^{\gamma}} &\text{if }k=J\\[.5em]
            \ds\sum_{l = j-K}^{-J}\beta_{j-l}^h    {\left(\frac{J}{l}\right)^{\gamma}} &\text{if }k=-J
        \end{array}
        \right.
    \end{equation}
    \begin{equation}\label{eq:def_discrLap_trunc3}
        (Q_\alpha^{h,J,K})_{jk}\ =\ \left\{
        \begin{array}{ll}
            0&\text{if }|k|<J\\[.5em]
            \ds    {\frac{C_{1,\alpha}(Jh)^{\gamma}}{(Kh)^{\alpha+\gamma}(\alpha+\gamma)}\,{}_2F_1\left(\gamma,\alpha+\gamma,1+\alpha+\gamma,-\frac{j}{K}\right)}&\text{if }k=J\\[.5em]
            \ds    {\frac{C_{1,\alpha}(Jh)^{\gamma}}{(Kh)^{\alpha+\gamma}(\alpha+\gamma)}\,{}_2F_1\left(\gamma,\alpha+\gamma,1+\alpha+\gamma,\frac{j}{K}\right)}&\text{if }k=-J
        \end{array}
        \right.
    \end{equation}
    
    {
    \begin{rema}[Choice of $\gamma$] At the continuous level, as soon as $f^\text{in}/\mu_\alpha$ is bounded, the solution $f\equiv f(t,v)$ to the L\'evy-Fokker-Planck equation $\partial_tf = L_\alpha f$ decays exactly like $O(|v|^{-1-\alpha})$ at infinity for any positive time $t>0$. It is seen by expressing the solution with the fundamental solution of the equation (which is essentially $\mu_\alpha$ rescaled in time). In this case,
    \[
     \gamma = 1+\alpha\,.
    \]
    If the decay of the initial data is slower, say if $f^\text{in}$ behaves like $O(|v|^{-\beta})$ at infinity, with $\beta<1+\alpha$, then the same decay will hold for the solution at any positive time. In this case the parameter should be taken as  $\gamma = \beta$.
    \end{rema}}
    \subsection{Discretization of the drift and boundary fluxes}
    Now we turn to the discretization of the drift term $\partial_v(vf)$. The goal here is to propose a consistent discretization $\Ght$ which, despite the truncation of the domain, preserves the two important features that are preservation of mass and equilibrium for the full truncated discrete Fokker-Planck operator (matrix)
    \begin{equation}\label{q:def_discrLap_trunc2}
        \Laht\ =\ \Ght \,+\, \Dht\,.
    \end{equation}
    Concerning conservation of mass, let us remark that at the continuous level the mass is not preserved on the truncated domain $[-L,L]$. For the classical Fokker-Planck equation, when $\alpha = 2$, it makes sense to impose conservation of mass on the truncated domain at the discrete level because densities decay typically like Gaussian, so that up to a choice of a large truncation parameter $L>0$, the loss of information outside of the domain is comparable  to machine precision error. In the fractional case~$\alpha<2$, the mass outside of the truncation domain is non negligible and should be taken into account (see Remark~\ref{rem:Ialph} for a quantitative illustration). We shall do this by imposing well-chosen artificial boundary conditions.
    
    The truncated discretization of the drift term still writes
    \begin{equation}\label{def:Gamma_trunc1}
        (\Ght f)_j\ =\  \frac{1}{2h}\left(\pare{VM}_{j+\frac{1}{2}}^{J,K}\left(\frac{f_{j+1}}{M_{j+1}}+\frac{f_j}{M_j}\right) - \pare{VM}_{j-\frac{1}{2}}^{J,K}\left(\frac{f_{j}}{M_{j}}+\frac{f_{j-1}}{M_{j-1}}\right)\right)\,,\quad\text{for }|j|< J
    \end{equation}
    where the approximations $\pare{VM}_{j+\frac{1}{2}}^{J,K}$ of $v_{j+1/2}\mu_\alpha(v_{j+1/2})$ are given, as in the untruncated case, in terms of the discrete fractional Laplacian of the equilibrium
    \begin{equation}\label{def:VMhalf_trunc}
        \pare{VM}_{j+\frac{1}{2}}^{J,K} = -\pare{VM}_{-j-\frac{1}{2}}^{J,K} := -\frac{1}{2}\sum_{k=-j}^j(\Dht M)_k\,h\,, \quad\text{for }j\geq0\,.
    \end{equation}
    It remains to define the endrows of $\Ght$. We set
    \begin{equation}\label{def:Gamma_trunc2}
        (\Ght f)_J\ =\  \frac{1}{h}\left(\mathcal{F}_{J+1/2}(f) - \frac12\pare{VM}_{J-\frac{1}{2}}^{J,K}\left(\frac{f_{J}}{M_{J}}+\frac{f_{J-1}}{M_{J-1}}\right)\right)\,,
    \end{equation}
    and 
    \begin{equation}\label{def:Gamma_trunc3}
        (\Ght f)_{-J}\ =\  \frac{1}{h}\left(\frac12\pare{VM}_{-J+\frac{1}{2}}^{J,K}\left(\frac{f_{-J+1}}{M_{-J+1}}+\frac{f_{-J}}{M_{-J}}\right) - \mathcal{F}_{-J-1/2}(f)\right)\,,
    \end{equation}
    where the boundary fluxes are defined by 
    \begin{multline}\label{def:Gamma_trunc4}
        \mathcal{F}_{J+1/2}(f)\ =\ -h\,(\Dht f)_J-\frac{h}{2(h+I_\alpha^L)}\sum_{k = -J+1}^{J-1}(\Dht f)_k\,h\\\,+\,\frac{I_\alpha^L}{2(h+I_\alpha^L)}\pare{VM}_{J-\frac{1}{2}}^{J,K}\left(\frac{f_{J}}{M_{J}}+\frac{f_{J-1}}{M_{J-1}}\right)\,,
    \end{multline}
    and
    \begin{multline}\label{def:Gamma_trunc5}
        \mathcal{F}_{-J-1/2}(f)\ =\ h\,(\Dht f)_{-J}+\frac{h}{2(h+I_\alpha^L)}\sum_{k = -J+1}^{J-1}(\Dht f)_k\,h\\\,+\,\frac{I_\alpha^L}{2(h+I_\alpha^L)}\pare{VM}_{-J+\frac{1}{2}}^{J,K}\left(\frac{f_{-J+1}}{M_{-J+1}}+\frac{f_{-J}}{M_{-J}}\right)\,.
    \end{multline}
    
    The quantity $f_JI_\alpha^L$ (\emph{resp.} $f_{-J}I_\alpha^L$) is the mass of the extension of $f$ at the right (\emph{resp.} the left) of the truncation domain. In order to determine the parameter $I_\alpha^L$, we assume that at the right (\emph{resp.} left) of the domain  $f$ is well approximated by a $f_J\mu_\alpha(v)/ \mu_\alpha(L)$ (\emph{resp} $f_{-J}\mu_\alpha(v)/ \mu_\alpha(-L)$). It may be defined in several ways and we choose it to be such that $
    I_\alpha^L\ \approx\ \mu_\alpha(L)^{-1}\int_{L}^\infty\mu_\alpha(v)\dd v
    $. Observe that $2\int_{L}^\infty\mu_\alpha(v)\dd v = 1-\int_{-L}^L\mu_\alpha(v)\dd v$ so that a natural definition of $ I_\alpha^L$ is  
    \begin{equation}\label{eq:def_Ialph}
        I_\alpha^L = \frac{1}{2\,M_J}\left(1-\sum_{j = -J}^JM_j\,h\right)\,. 
    \end{equation}
    
    We now need to justify the expression of the boundary fluxes. Observe that the first two terms in both \eqref{def:Gamma_trunc4} and \eqref{def:Gamma_trunc5} are consistent with $0$ so these discrete fluxes are consistent with the continuous ones (if, say, $h\to0$ and $Jh$ remains constant). The purpose of these correction terms is illustrated by the following proposition.
    
    \begin{prop}
        The  truncated discrete fractional Fokker-Planck operator $\Laht$ defined in \eqref{eq:def_discrLap_trunc1}-\eqref{def:Gamma_trunc5} satisfies the following properties.
        \begin{itemize}
            \item[\emph{i)}] It preserves the equilibrium, namely
            \[
            (\Laht M)_j = 0 \,.
            \]
            \item[\emph{ii)}] It preserves the total mass: if
            \[
            \frac{f^{n+1} - f^{n}}{\Delta t}\ = \Laht f^{n+1}
            \]
            then 
            \[
            \sum_{j=-J}^J f_j^{n+1}\,h + I_\alpha^L(f_J^{n+1}+f_{-J}^{n+1})\ =\  \sum_{j=-J}^J f_j^{n}\,h + I_\alpha^L(f_J^n+f_{-J}^n)\,.
            \]
        \end{itemize}
    \end{prop}
    
    \begin{proof}
         The proof of i) relies on \eqref{def:Gamma_trunc1} and \eqref{def:VMhalf_trunc}. For ii), one can start from the desired equality, put all the terms from the right-hand side to the left and then replace $f^{n+1} - f^n$ by $(\Laht f^{n+1}) \Delta t$. Then, using the definition of $\Laht$ a lengthy but straightforward computation proves the equality. 
    \end{proof}

    \begin{rema}\label{rem:Ialph}
        Observe that if one takes $I_\alpha^L = 0$, the mass is preserved inside the truncation domain. As we stressed earlier in the section, this is however not an accurate way of discretizing the Fokker-Planck equation with fractional diffusion. As an illustration, if we truncate the domain at $L = 100$, then in the fractional case $I_1^{100} \approx 10^{2}$ whereas in the classical case $I_2^{100} \approx 5.10^{-3}$. To better interpret these values we recall that $2I_\alpha^L$ is the amount of mass of $\mu_\alpha$ outside of the domain relatively to the value of $\mu_\alpha$ at the boundary of the domain.
    \end{rema}

    \section{Discrete functional analysis} \label{sec:funcanal}
    In this section, we introduce discrete functional analysis tools, adapted to the discretization of fractional diffusion in dimension $1$ and which are going to be essential in order to derive the stability and asymptotic properties of the schemes  \eqref{eq:def_scheme_hom} and \eqref{eq:def_scheme_kin} in unbounded velocity domain.

    \subsection{Discrete functional spaces and notations}
    We start with some definitions and notations. Given a mesh size $h>0$ and a positive sequence $\gamma = (\gamma_j)_{j\in\ZZ}$, which may depends on $h$, we introduce several Hilbert spaces which we characterize by their norms. Given a norm (or a semi-norm) $\|\cdot\|_X$ on the space $X${, satisfying the parallelogram identity}, the  corresponding inner product is denoted and classically defined by $\lla\cdot,\ast\rra_X = (\|\cdot+\ast\|_X^2-\|\cdot-\ast\|_X^2)/4$. We start by the weighted discrete Lebesgue space $\ell^2_h(\gamma)$. For a sequence $(g_j)_{j\in\ZZ}$, we define
    \begin{equation}\label{eq:def_discrLeb}
        \|g\|_{\ell^2_h(\gamma)}^2\ =\ \sum_{j\in\ZZ}g_j^2\,\gamma_j\,h\,.
    \end{equation}
    The sequence $\gamma$ is to be thought of as the local equilibrium $M$ defined in \eqref{eq:def_discr_equilibrium} or its inverse depending on the context. We also introduce weighted discrete fractional Sobolev seminorms
    \begin{equation}\label{eq:def_discrFracSob}
        |g|_{\dot{H}_h^{s}(\gamma)}^2\ =\ \sum_{j\in\ZZ}\sum_{k\in\ZZ\setminus\{0\}}\frac{(g_j-g_{j+k})^2}{|hk|^{1+2s}}\,\gamma_j\,h^2\,,\quad s>0
    \end{equation}
    and norms
    \[
    \|g\|_{H_h^{s}(\gamma)}^2\ =\ \|g\|_{\ell^2_h(\gamma)}^2\,+\,|g|_{\dot{H}_h^{s}(\gamma)}^2\, ,\quad  \forall s \in (0,1) \,.
    \]

    We also introduce the finite difference operators 
    \[
    (D^{+}_hg)_{j}\ =\ \frac{g_{j+1}-g_j}{h},
    \]
    
    \[
    (D_hg)_{j} = \frac{g_{j+1}-g_{j-1}}{2h}\,,
    \]
    and thus \[(D^{2}_hg)_{j}\ =\ \frac{g_{j+2}+ g_{j-2}-2g_j}{4h^2}\,.\]
    Observe that one has the integration by part formula
    $
    \sum_{j\in\ZZ}(D_hg)_{j}f_j\ =\ -\sum_{j\in\ZZ}g_j(D_hf)_{j}\,.
    $
    
    Finally, we define weighted $H^1_h$-Sobolev spaces through the norm
    \begin{equation}\label{eq:def_discrSob}
        \|g\|^2_{H^1_h(\gamma)} = \|g\|^2_{\ell^2_h(\gamma)} + \|D_hg\|^2_{\ell^2_h(\gamma)}\,.
    \end{equation}
    
    For flat norms and semi-norms, that is when $\gamma_j = 1$ for all $j\in\ZZ$, we only write $\ell^2_h$ instead of $\ell^2_h(\gamma)$ and do the analogous modification of notation for the other spaces.
    
    \begin{rema}
        Beware that, with the above notation $|g|_{\dot{H}_h^{1}(\gamma)}\neq\|D_hg\|_{\ell^2_h(\gamma)}$. While there is no continuous equivalent of $|\cdot|_{\dot{H}_h^{s}(\gamma)}$ for $s\geq 1$ because of the singularity at $0$, it is a useful intermediate quantity in the discrete setting. 
    \end{rema}

    In the following, when we write $A\lesssim B$, we mean that there is a positive constant $C$ which may depend on $\alpha$, and other parameters, but never on the mesh size $h$ such that $A\leq C\, B$. If $A\lesssim B$ and $B\lesssim A$, we write $A\sim B$. If necessary, we write $\lesssim_{a,b,\dots}$ or $\sim_{a,b,\dots}$ to indicate the dependence of the constants on parameters $a,b,\dots$. Given sequences $A = (A_j)_j$, $B = (B_j)_j$ and a function $\phi$, products and application of a function have to be understood componentwise, namely $AB = (A_jB_j)_j$, $\phi(A) = (\phi(A_j))_j$.
    \begin{rema}
        Observe for instance that as a consequence of the definitions of $\Dh$ and $\Lah$ in Section~\ref{sec:pres_num}, the bounds of Lemma~\ref{lem:bounds_beta} and the previous definitions one has $ \lla-\Lah f,f\rra_{\ell^2_h(M^{-1})} = \Symh(f,f)\ \sim\ |f/M|^2_{\dot{H}_h^{\alpha/2}(M)}$.
    \end{rema}

    \subsection{Discrete non-local Poincaré inequalities}
    
    In this section, we establish a class of discrete functional inequalities which are the counterpart of nonlocal Poincaré inequalities \cite{gentil_2008_levy, wang_2014_simple}. The method of proof is directly inspired by \cite{wang_2014_simple}.
    
    We assume that $(\gamma_j^h)_{j\in\ZZ}$ and $(w_j^h)_{j\in\ZZ^*}$ are two sequences such that
    \begin{equation}\label{eq:symm}
        w_j^h\ =\ w_{-j}^h\,,\qquad \forall j\in\ZZ^*\,.
    \end{equation}
    We assume moreover that there is a constant $C_P$ such that
    \begin{equation}\label{eq:assump}
        \gamma_j^h\,\gamma_k^h\ \leq\ C_P\,(\gamma_j^h+\gamma_k^h)\,w_{j-k}^h\,,\quad \forall h>0\,,\ \forall j,k\in\ZZ\,,\ j\neq k\,.
    \end{equation}
    Then the following discrete functional inequality holds. 
    \begin{prop}\label{prop:general_inegfunc_discr}
        Under the assumptions~\eqref{eq:symm} and \eqref{eq:assump}, for any $h>0$ and any suitably summable sequence $g= (g_j)_j$ such that 
        \[
        \sum_{j\in\ZZ} g_j\,\gamma_j^h\,h\ =\ 0\,,
        \]
        one has 
        \[
        \|g\|_{\ell^2_h(\gamma^h)}^2\ \leq\ { \frac{C_P}{\sum_{j \in \ZZ} \gamma_j^h h}}\sum_{\substack{(j,k)\in\ZZ^2\\j\neq k}}w_{j-k}^h(g_j-g_k)^2\,\gamma_k^h\,h^2\,.
        \]
    \end{prop}
    \begin{proof} We use assumptions to get
        \[
        \begin{array}{rcl}
            \ds\|g\|_{\ell^2_h(\gamma^h)}^2&=&\ds\sum_{j\in\ZZ} g_j^2\,\gamma_j^h\,h 
            = \ds {\frac{1}{2\sum_{j \in \ZZ} \gamma_j^h h}}\,\sum_{j\in\ZZ}\sum_{k\in\ZZ}(g_j-g_k)^2\,\gamma_j^h\,\gamma_k^h\,h^2\\[1em]
            &\leq&\ds C_P\,{\frac{1}{2\sum_{j \in \ZZ} \gamma_j^h h}}\,\sum_{\substack{(j,k)\in\ZZ^2\\j\neq k}}w_{j-k}^h\,(g_j-g_k)^2\,(\gamma_k^h +\gamma_j^h)\,h^2\,.
        \end{array}
        \]
        After expanding the last sum and changing indices, one recovers the result.
    \end{proof}
    
    Thanks to the following lemma, we will show that the discrete equilibrium $M$ satisfies assumptions~\eqref{eq:symm} and~\eqref{eq:assump}.
    
    \begin{lem}\label{lem:inegwang}
        Let $\alpha\in(0,2)$ and $\nu:\RR^d\to\RR$ be a probability density function such that for some constant $c_{d,\alpha}$ one has 
        \begin{equation}\label{bounds}
            \nu(v)\ \leq \frac{c_{d,\alpha}}{1+|v|^{d+\alpha}}\,\quad \text{for all}\  v\in\RR^d\,.
        \end{equation}
        Then for all $v,w\in\RR^d$, one has 
        \begin{equation}\label{ineg}
            \nu(v)\,\nu(w)\ \leq\ 2^{d+\alpha-1}c_{d,\alpha}\,\frac{\nu(v)+\nu(w)}{|v-w|^{d+\alpha}}
        \end{equation}
    \end{lem}
    \begin{proof}
        By using \eqref{bounds} one has 
        \[
        \ds\frac{\nu(v)^{-1}+\nu(w)^{-1}}{|v-w|^{d+\alpha}} \geq \ds c_{d,\alpha}^{-1}\,\frac{|v|^{d+\alpha} + |w|^{d+\alpha}}{|v-w|^{d+\alpha}}\geq\ds c_{d,\alpha}^{-1}\,\frac{|v|^{d+\alpha} + |w|^{d+\alpha}}{(|v|+|w|)^{d+\alpha}}\geq\ds 2^{1-(d+\alpha)}c_{d,\alpha}^{-1}
        \]
        where we used the convexity of $x\mapsto x^{d+\alpha}$ in the last inequality.
    \end{proof}
    We can now state the main result of this section. We recall that $M_j = \mu_\alpha(v_j)$ and  $\cS_\alpha^h$ is defined in~\eqref{eq:def_sym_discr}.
    
    \begin{prop}\label{prop:inegfunc_discr}
        For any suitably summable sequence $f=(f_j)_j$, one has 
        \begin{equation}\label{eq:discr_poinc}
            \|f-\Pi_hf\|_{\ell^2_h(M^{-1})}^2\ \lesssim\ \cS_{\alpha}^h(f,f)\,,
        \end{equation}
        where the projection $\Pi_h$ is defined by the formula $(\Pi_hf)_j=M_j\,(\sum_{k\in\ZZ}f_k\,h)/(\sum_{k\in\ZZ}M_k\,h)$.
    \end{prop}
    \begin{proof}
        Let $\gamma^h_j = \mua(hj)$ and $w^h_k\ =\ |hk|^{-(1+\alpha)}$ for $j\in\ZZ$ and $k\in\ZZ^*$. The symmetry assumption~\eqref{eq:symm} is clearly satisfied. 
        Then, we get from Lemma~\ref{lem:inegwang} and Proposition~\ref{prop:bounds_mua} in the appendix that for all~$v,w\in\RR$ with $v\neq w$, one has
        $
        \mua(v)\mua(w)\ \leq\ 2^\alpha\,C_1\,(\mua(v) + \mua(w))|v-w|^{-(1+\alpha)}
        $
        where $C_1=C_1(\alpha)$ is the constant in \eqref{eq:mualpha_est1}. Therefore, \eqref{eq:assump} is also satisfied.
            { Moreover, still using Proposition~\ref{prop:bounds_mua}-\eqref{eq:mualpha_est1}, it is straightforward to check that there exists a constant $C=C(\alpha)>0$ independent of $h$ such that 
    	$
	\sum_{j \in \ZZ} M_j h \geq C.
	$}
	Consequently, Proposition~\ref{prop:general_inegfunc_discr} yields {that if~$g$ is such that $ \sum_{j\in\ZZ} g_j\,\gamma_j^h\,h = 0$, then}
        $
        \|g\|_{\ell^2_h(\gamma^h)}^2\ \lesssim\ |g|_{{ \dot H^{\alpha/2}_h(\gamma^h)}}^2 \, .
        $
        If we take $g = (f-\Pi_hf)/M$, we obtain 
        $
        \|f-\Pi_hf\|_{\ell^2_h(M^{-1})}^2\ \lesssim\ |f/M|_{{ \dot H^{\alpha/2}_h(M)}}^2\,.
        $
        One concludes by observing that, as a consequence of Lemma~\ref{lem:bounds_beta}, the right-hand side is bounded by a constant depending only on $\alpha$ times $\cS_{\alpha}^h(f,f)$. 
    \end{proof}

    \subsection{Regularization estimates} 
    The result of Proposition~\ref{prop:inegfunc_discr} can actually be improved because the dissipation $\cS_\alpha^h(f,f)$  also provides a gain of fractional Sobolev regularity, as shown in the next lemma. 
    \begin{lem} \label{lem:regSalpha}
        There exists a constant $C>0$ independent of $h$ such that for any suitably summable sequence $f=(f_j)_j$, one has:
        \[
        \cS_\alpha^h(f,f)\ \gtrsim\  |fM^{-1/2}|_{\dot H^{\alpha/2}_h}^{2} - C \|  f M^{-1/2} \|_{\ell^2_h}^{2}\,.
        \]
    \end{lem}
    \begin{proof}
        Using that $(a+b)^2\geq a^2/2-b^2$, we have
        $$
        \cS_\alpha^h(f,f) 
        \ge \sum_{j \in \ZZ} \sum_{\substack {k \in \ZZ^* \\ |k| \le 1/h}} \beta_k^h \left(\frac{f_j}{M_j}-\frac{f_{j+k}}{M_{j+k}}\right)^2 M_j h^2
        \geq \frac12I_1 - I_2
        $$
        with 
        $$
        I_1 
        :=\sum_{j \in \ZZ} \sum_{\substack {k \in \ZZ^* \\ |k| \le 1/h}} \beta_k^h \left(\frac{f_j}{\sqrt{M_j}}-\frac{f_{j+k}}{\sqrt{M_{j+k}}}\right)^2 h^2
        $$
        and
        $$
        I_2 := \sum_{j \in \ZZ} \sum_{\substack {k \in \ZZ^* \\ |k| \le 1/h}} \beta_k^h \frac{f^2_{j+k}}{M^2_{j+k}} \left(\sqrt{M_{j+k}}- \sqrt{M_j}\right)^2 h^2 \, .
        $$
        Using Lemma~\ref{lem:bounds_beta}, the first term can be bounded from below by 
        \[
        \begin{array}{rcl}
            I_1 
            &\gtrsim& \ds|  fM^{-1/2} |_{\dot H^{\alpha/2}_h}^{2} -\sum_{j \in \ZZ} \sum_{\substack {k \in \ZZ^* \\ |k| > 1/h}} \beta_k^h \left( \frac{f_j}{\sqrt{M_j}} - \frac{f_{j+k}}{\sqrt{M_{j+k}}}\right)^2 h^2\,.
        \end{array}
        \]
        Moreover, using Lemma~\ref{lem:bounds_beta} and a change of index, we have:
        $$
        \sum_{j \in \ZZ} \sum_{\substack {k \in \ZZ^* \\ |k| > 1/h}} \beta_k^h \left( \frac{f_j}{\sqrt{M_j}} - \frac{f_{j+k}}{\sqrt{M_{j+k}}}\right)^2 h^2 \lesssim \sum_{j \in \ZZ} \frac{f_j^2}{M_j} h \sum_{\substack {k \in \ZZ^* \\ |k| > 1/h}} \frac{h}{|hk|^{1+\alpha}} \lesssim \| fM^{-1/2}\|^2_{\ell^2_h}\, .
        $$
        It implies that for some constant $C>0$ (which does not depend on $h$):
        \begin{equation} \label{eq:I1_est}
            I_1 \gtrsim \ds | f M^{-1/2} |_{\dot H^{\alpha/2}_h}^{2}
            - C \| f M^{-1/2} \|_{\ell^2_h}^{2}\,.
        \end{equation}
        The second term can be rewritten as
        \begin{align*}
            I_2  &= \sum_{j \in \ZZ} \sum_{\substack {k \in \ZZ^* \\ |k| \le 1/h}} \beta_k^h \frac{f^2_{j+k}}{M^2_{j+k}} k^2 \left(\sum_{\ell=0}^{k-1} {1 \over k}\left(\sqrt{M_{j+\ell+1}}- \sqrt{M_{j+\ell}}\right)\right)^2 h^2\,.
        \end{align*}
        Using the convexity of the squared function and Lemma~\ref{lem:bounds_beta}, we get:
        \begin{align*}
            I_2 &\lesssim \sum_{j \in \ZZ} \sum_{\substack {k \in \ZZ^* \\ |k| \le 1/h}} \frac{|hk|^2}{|hk|^{1+\alpha}} {1 \over k} \sum_{\ell=0}^{k-1} \left(D_h^+ \sqrt{M}\right)^2_{j+\ell} \, \frac{f_{j+k}^2}{M^2_{j+k}} \, h^2\\
            &\lesssim \sum_{j \in \ZZ} \sum_{\substack {k \in \ZZ^* \\ |k| \le 1/h}} \frac{|hk|^2}{|hk|^{1+\alpha}} {1 \over k} \sum_{\ell=0}^{k-1} \frac{\left(D_h^+ \sqrt{M}\right)^2_{j+\ell}}{M_{j+\ell}} \, \frac{f_{j+k}^2}{M_{j+k}} \, h^2
        \end{align*}
        where for the last inequality, we used that we sum over the $k$ such that $|hk| \le 1$ and~\eqref{eq:mualpha_est1}, so that we can write the following bounds for any $j$, $k$ and $\ell$, $0 \le \ell \le k-1$:
        $
        M_{j+k}^{-1}  \lesssim M_{j}^{-1} \lesssim M_{j+\ell}^{-1}.
        $
        Using Corollary~\ref{cor:DhM-1/2} and estimate~\eqref{eq:mualpha_est1}, we have that $\left(D_h^+ \sqrt{M}\right)^2_j M_j^{-1} \in \ell^\infty_j$ uniformly in $h$. Consequently, performing a change of index, we obtain:
        \begin{equation} \label{eq:I2_est}
            I_2 \lesssim  \sum_{j \in \ZZ} \sum_{\substack {k \in \ZZ^* \\ |k| \le 1/h}} \frac{|hk|^2}{|hk|^{1+\alpha}} \frac{f_{j+k}^2}{M_{j+k}} h^2 \lesssim \| f M^{-1/2}\|^2_{\ell^2_h}\,.
        \end{equation}
        Combining~\eqref{eq:I1_est} and~\eqref{eq:I2_est}, we obtain the wanted result. 
    \end{proof}
    
    \begin{prop} \label{prop:inegfunc_discr2}
        For any suitably summable sequence $f=(f_j)_j$, one has:
        \[
        \cS_\alpha^h(f,f)\ \gtrsim\  \|(f-\Pi_h f)M^{-1/2}\|_{ H^{\alpha/2}_h}^{2}\,.
        \]
    \end{prop}
    \begin{proof}
        One can notice that $\Pi_h f = \Pi^2_h f$ and
        $$
        \cS_\alpha^h(f,f)\ =\ \cS_\alpha^h(f-\Pi_h f,f-\Pi_h f)\,.
        $$
        As a consequence, an appropriate convex combination of inequalities coming from Proposition~\ref{prop:inegfunc_discr} and Lemma~\ref{lem:regSalpha} applied to $f-\Pi_h f$ shows the wanted inequality.
    \end{proof}

    \subsection{Interpolation and embeddings in discrete fractional spaces} \label{subsec:interp}
    
    In this section, we derive embedding and interpolation inequalities between the previously introduced discrete fractional Sobolev spaces and mixed spaces involving discrete difference operators. { Notice that our proof is only based on nonlocal estimates on the fractional Sobolev norms whereas in the continuous case, Fourier transform and Young inequality straightforwardly give this type of interpolation inequalities. Our proof is thus completely different from the standard ones that are developed the continuous case.}     
    
    \begin{prop}\label{prop:embed}
        One has, uniformly in $h>0$, the following results of continuous embeddings.
        \begin{itemize}
            \item[(i)] For any $0<s_2\leq s_1  <1$, one has
            \[
            \|f\|_{H^{s_2}_h}\ \lesssim\ \|f\|_{H^{s_1}_h}\,.
            \]
            \item[(ii)] For any $0<s<1$, one has
            \[
            \|f\|_{H^{s}_h}\ \lesssim\ \|f\|_{\ell^2_h} + \|D_h^+f\|_{\ell^2_h}\,.
            \]
            \item[(iii)] One has
            \[
            \|D_hf\|_{\ell^2_h}\ \lesssim \ \|D_h^+f\|_{\ell^2_h}\ \leq\ |f|_{\dot{H}^{1}_h}\,.
            \]
        \end{itemize}
    \end{prop}
    \begin{proof}
        For $(i)$, let us notice that if $|hk| \le 1$, then $|hk|^{s_1} \le |hk|^{s_2}$ so that 
        $$
        \sum_{j \in \ZZ} \sum_{\substack {k \in \ZZ^* \\ |k| \le 1/h}} \frac{(f_j-f_{j+k})^2}{|hk|^{1+2s_2}} h^2\le \sum_{j \in \ZZ} \sum_{\substack {k \in \ZZ^* \\ |k| \le 1/h}} \frac{(f_j-f_{j+k})^2}{|hk|^{1+2s_1}} h^2 \le |f|^2_{\dot H^{s_2}_h}\,.
        $$
        Moreover, using the tail estimate
        $$
        \sum_{|k| \ge 1/h} {1 \over |k|^{1+2s_2}} \lesssim h^{2s_2}\,,
        $$
        we have that
        $$
        \sum_{j \in \ZZ} \sum_{\substack {k \in \ZZ^* \\ |k| \ge 1/h}} \frac{(f_j-f_{j+k})^2}{|hk|^{1+2s_2}} h^2\lesssim 
        \sum_{j \in \ZZ} f_j^2 \sum_{\substack {k \in \ZZ^* \\ |k| \ge 1/h}} \frac{h^2}{|hk|^{1+2s_2}} \lesssim \sum_{j \in \ZZ} f_j^2 h = \|f\|^2_{\ell^2_h}\,,
        $$
        which concludes the proof of $(i)$.
        
        Concerning $(ii)$, we split the sum over $k$ into two parts and write that 
        \begin{align*}
            |f|^2_{\dot H^s_h}  &= \sum_{j \in\ZZ} \sum_{\substack {k \in \ZZ^* \\ |k| \le 1/h}} \frac{\left( \sum_{\ell=1}^k {1 \over k} (f_{j+\ell}- f_{j+\ell-1})\right)^2}{|hk|^{1+2s}}|hk|^2 
            +\sum_{j \in\ZZ} \sum_{\substack {k \in \ZZ^* \\ |k| \ge 1/h}} \frac{(f_j-f_{j+k})^2}{|hk|^{1+2s}}h^2\\&=: I_1+I_2\,.
        \end{align*}
        For the first term, we use the convexity of the squared function, Fubini theorem and a change of index, it gives us that: 
        $$
        I_1 \le \sum_{\substack {k \in \ZZ^* \\ |k| \le 1/h}} \sum_{\ell=1}^k \sum_{j \in \ZZ} \frac{{1 \over k} (f_{j}- f_{j+1})^2}{|hk|^{1+2s}}|hk|^2 \le  \sum_{j \in \ZZ} \sum_{\substack {k \in \ZZ^* \\ |k| \le 1/h}} \frac{(f_{j}- f_{j+1})^2}{|hk|^{1+2s}}|hk|^2\,.
        $$
        From this, we get
        $$
        I_1 \le  \sum_{j \in \ZZ} \frac{(f_j-f_{j+1})^2}{h}\sum_{\substack {k \in \ZZ^* \\ |k| \le 1/h}} \frac{h}{|hk|^{1+2s-2}} \lesssim \|D_h^+f\|^2_{\ell^2_h}\,. 
        $$
        The second term is easier to treat, we have:
        $$
        I_2 \lesssim \sum_{j \in\ZZ} \sum_{\substack {k \in \ZZ^* \\ |k| \ge 1/h}} \frac{f_j^2+f_{j+k}^2}{|hk|^{1+2s}}h^2
        $$
        and thus using a change of index in $j$:
        $$
        I_2 \lesssim \sum_{j \in \ZZ} f_j^2 h  \sum_{\substack {k \in \ZZ^* \\ |k| \ge 1/h}} \frac{h}{|hk|^{1+2s}} \lesssim \|f\|^2_{\ell^2_h}\,,
        $$
        which yields the second result. 
        
        Finally, concerning (iii), the first inequality is straightforward. For the second one, notice that just keeping the term corresponding to~$k=1$ in the right hand side term, we have
        $$
        \|D_h^+f\|^2_{\ell^2_h} = \sum_{j \in \ZZ} \frac{(f_j-f_{j+1})^2}{h} \leq \sum_{j \in \ZZ} \sum_{k \in \ZZ^*} \frac{(f_j-f_{j+k})^2}{|hk|^3} h^2\ =\ |f|^2_{\dot H^1_h}\,,
        $$
        which ends the proof.
    \end{proof}

    \begin{lem} \label{lem:interpHs}
        For any positive $s$, $\beta$ such that $0<s-\beta<s+\beta$, one has
        $$
        |f|^2_{\dot H^s_h}\ \lesssim\ |f|_{\dot H^{s+\beta}_h}\, |f|_{\dot H^{s-\beta}_h}\,.
        $$
    \end{lem}
    \begin{proof}
        The result directly comes from Cauchy-Schwarz inequality. 
    \end{proof}
    
    \begin{prop} \label{prop:interpDh+}
        For any $s \in (0,1)$, we have:
        $$
        \|D_h^+ f\|_{\ell^2_h}^2\ \lesssim\ |f|_{\dot H^{1-s}_h} \,|D_h f|_{\dot H^s_h}\,.
        $$
    \end{prop}
    In order to prove Proposition~\ref{prop:interpDh+}, we need the following elementary lemma.
    \begin{lem}\label{lem:1/k2}
        Let $g_k := 1/(1+k^2)$ for $k \in \ZZ$. We have the following inequality:
        $$
        k^{-3} \lesssim -h (D_hg)_k, \quad \forall  k \in \ZZ^*\,. 
        $$
    \end{lem}
    \begin{proof}
        Let $k \in \ZZ^*$. A computation yields 
        $$
        {1 \over k^3} + {1 \over 2} h (D_h g)_k= \frac{4}{k^3((1+(k+1)^2)(1+(k-1)^2)}\le \frac{4}{5k^3}\,,
        $$
        which proves the result.
    \end{proof}
    \begin{proof}[Proof of Proposition~\ref{prop:interpDh+}]
        Let us start by combining Proposition~\ref{prop:embed}-(iii) with Lemma~\ref{lem:1/k2} to get 
        $$
        \|D_h^+f\|^2_{\ell^2_h} \lesssim - \sum_{j \in \ZZ} \sum_{k \in \ZZ^*} (D_h g)_k (f_j-f_{j+k})^2 = - \sum_{j \in \ZZ} \sum_{k \in \ZZ} (D_h g)_k (f_j-f_{j+k})^2\,.
        $$
        We now temporarily fix $j$ and only look at the sum over $k$. Denoting $a_k=f_j-f_{j+k}$, we have 
        \begin{align*}
            (D_h(a^2))_k &= (a_{k+1}+a_{k-1}) \frac{a_{k+1} - a_{k-1}}{2h} \\
            &= - (f_j-f_{j+k+1}+f_j-f_{j+k-1}) (D_h f)_{j+k}
        \end{align*}
        From this, performing an integration by parts in $k$, we deduce that 
        $$
        -\sum_{k \in \ZZ} (D_h g)_k (f_j-f_{j+k})^2 = -\sum_{k \in \ZZ} (D_h g)_k a_k^2
        =  \sum_{k \in \ZZ} g_k (D_h(a^2))_k
        $$
        i.e. 
        $$
        -\sum_{k \in \ZZ} (D_h g)_k (f_j-f_{j+k})^2 = -\sum_{k \in \ZZ} {1 \over {1+k^2}}(f_j-f_{j+k+1}+f_j-f_{j+k-1}) (D_h f)_{j+k}\,. 
        $$
        It implies that
        \begin{align*}
            &\|D_h^+f\|^2_{\ell^2} \\
            &\quad \lesssim - \sum_{j, k \in \ZZ} {1 \over {1+k^2}}(f_j-f_{j+k+1})(D_h f)_{j+k}  - \sum_{j, k \in \ZZ} {1 \over {1+k^2}}(f_j-f_{j+k-1})(D_h f)_{j+k} \\
            &\quad \lesssim - \sum_{j, k \in \ZZ} {1 \over {1+k^2}}(f_{j+k-1}-f_{j})(D_h f)_{j-1}  - \sum_{j, k \in \ZZ} {1 \over {1+k^2}}(f_j-f_{j+k-1})(D_h f)_{j+k} \\
            &\quad  \lesssim - \sum_{j, k \in \ZZ} {1 \over {1+k^2}}(f_j-f_{j+k-1})((D_h f)_{j+k} - (D_h f)_{j-1})
        \end{align*}
        where we performed a change of index in the first sum to get the second inequality. Now, from Cauchy-Schwarz inequality, we get
        \begin{align*}
            &\|D_h^+f\|^2_{\ell^2} \\
            &\quad \lesssim \left( \sum_{j, k \in \ZZ} \frac{(f_j-f_{j+k-1})^2}{(1+k^2)^{1/2+1-s} h^{1+2(1-s)}} h^2 \right)^{1/2} 
            \left( \sum_{j, k \in \ZZ} \frac{ ((D_h f)_{j+k} - (D_h f)_{j-1})^2}{(1+k^2)^{1/2+s} h^{1+2s}} h^2 \right)^{1/2} \,.
        \end{align*} 
        To conclude, performing changes of indices, we remark that 
        \begin{align*}
            \sum_{j, k \in \ZZ} \frac{(f_j-f_{j+k-1})^2}{(1+k^2)^{1/2+1-s} h^{1+2(1-s)}} h^2 &=  \sum_{j \in \ZZ} \sum_{k \in \ZZ^*} \frac{(f_j-f_{j+k})^2}{(1+(k+1)^2)^{1/2+1-s} h^{1+2(1-s)}} h^2 \\
            &\lesssim  \sum_{j \in \ZZ} \sum_{k \in \ZZ^*} \frac{(f_j-f_{j+k})^2}{ |h k|^{1+2(1-s)}} h^2 = |f|^2_{\dot H^{1-s}_h}
        \end{align*}
        and 
        \begin{align*}
            \sum_{j, k \in \ZZ} \frac{ ((D_h f)_{j+k} - (D_h f)_{j-1})^2}{(1+k^2)^{1/2+s} h^{1+2s}} h^2 &= \sum_{j \in \ZZ}\sum_{k \in \ZZ^*} \frac{ ((D_h f)_{j} - (D_h f)_{j+k})^2}{(1+(k-1)^2)^{1/2+s} h^{1+2s}} h^2 \\
            &\lesssim \sum_{j \in \ZZ}\sum_{k \in \ZZ^*} \frac{ ((D_h f)_{j} - (D_h f)_{j+k})^2}{ |hk|^{1+2s}} h^2 = |D_h f |_{\dot H^s_h}^2
        \end{align*}
        where the two last inequalities come from the fact that $1+(k \pm1)^2 \ge k^2/2$. 
    \end{proof}
{ Roughly speaking, thanks to Proposition~\ref{prop:interpDh+}, we are only able to prove ``symmetric interpolation inequalities''. In the next theorem, we are able to extend this result to a ``non symmetric'' framework thanks to a sort of iteration argument, which is, to our knowledge, not standard in this context.}
    
    \begin{theo} \label{theo:interp}
        Let $s \in (0,1)$. There is $\eta_s>0$ such that for any $\eps\in(0,\eta_s)$, there is $K(\eps)>0$ such that  
        \begin{equation} \label{ineq:interpfinal}
            \|D_h^+ f\|^2_{\ell^2_h}\ \lesssim\ \eps\, |D_h f|^2_{\dot H^{s}_h} \,+\,K(\eps)\,\|f\|^2_{H^{s}_h}\,. 
        \end{equation}
    \end{theo}
    \begin{proof}
        Using Proposition~\ref{prop:interpDh+} combined with Young inequality, we get that for any $\eps_1 \in (0,1)$, 
        \begin{equation} \label{eq:ineqeps1}
            \|D_h^+f\|^2_{\ell^2_h} \lesssim \eps_1\, |D_hf|^2_{\dot H^s_h} \,+\, {1 \over \eps_1}\,|f|^2_{\dot H^{1-s}_h}\,. 
        \end{equation}
        If $s \ge 1/2$, then $1-s \le s$ and it is thus enough to conclude since from Proposition~\ref{prop:embed}-(i), we have
        $$
        |f|_{\dot H^{1-s}_h}\ \lesssim\ \|f\|_{H^s_h}\,. 
        $$
        Let us now deal with the case $s \le 1/2$. Consider the minimal integer $q$ such that $1-s-qs/2 \le s$. Coming back to~\eqref{eq:ineqeps1}, we use Lemma~\ref{lem:interpHs} and Young inequality to write that for any $\eps_2\in(0,1)$,
        $$
        \|D_h^+f\|^2_{\ell^2_h} \lesssim \eps_1 |D_hf|^2_{\dot H^s_h} + {1 \over \eps_1} \eps_2 |f|^2_{\dot H^{1-s/2}_h} + {1 \over {\eps_1\eps_2}} |f|^2_{\dot H^{1-3s/2}_h}\,. 
        $$
        From Proposition~\ref{prop:embed}-(ii), we obtain
        \begin{equation} \label{eq:ineqeps2}
            \|D_h^+f\|^2_{\ell^2_h} \lesssim  \eps_1 |D_hf|^2_{\dot H^s_h} + {1 \over \eps_1} \eps_2 (\|f\|^2_{\ell^2_h} + \|D_h^+f\|^2_{\ell^2_h})+ {1 \over {\eps_1\eps_2}} |f|^2_{\dot H^{1-3s/2}_h}\,.
        \end{equation}
        Reiterating the process, we obtain that for any $\eps_1, \dots, \eps_{q+1}\in (0,1)$:
        $$
        \|D_h^+f\|^2_{\ell^2_h} \lesssim  \eps_1 |D_hf|^2_{\dot H^s_h} + \left( \sum_{k=1}^q{1 \over \Pi_{\ell=1}^k \eps_\ell} \eps_{q+1} \right) (\|f\|^2_{\ell^2_h} + \|D_h^+f\|^2_{\ell^2_h})+ {1 \over {\Pi_{\ell=1}^{q+1} \eps_\ell}} |f|^2_{\dot H^{1-s-qs/2}_h}\,.
        $$
        From the definition of $q$ and Proposition~\ref{prop:embed}-(i), we deduce that 
        \begin{equation} \label{eq:ineqepsq}
            \|D_h^+f\|^2_{\ell^2_h} \lesssim  \eps_1 |D_hf|^2_{\dot H^s_h} + \left( \sum_{k=1}^q{1 \over \Pi_{\ell=1}^k \eps_\ell} \eps_{k+1} \right) \|D_h^+f\|^2_{\ell^2_h}+ {1 \over {\Pi_{\ell=1}^{q+1} \eps_\ell}} \|f\|^2_{H^{s}_h}\,.
        \end{equation}
        In order to get the wanted inequality~\eqref{ineq:interpfinal}, we consider $\eps \in (0,1)$ and set $\eps_\ell := \eps^{2^{\ell-1}}$ for any $1 \le \ell \le q+1$. With this choice of $\eps_\ell$, the previous inequality~\eqref{eq:ineqepsq} becomes 
        $$
        \|D_h^+f\|^2_{\ell^2_h} \lesssim  \eps |D_hf|^2_{\dot H^s_h} + q \eps \|D_h^+f\|^2_{\ell^2_h}+\frac{\eps}{\eps^{2^{q+1}}} \|f\|^2_{H^{s}_h}\,.
        $$
        Taking $\eps$ small enough enables us to absorb the second term of the right-hand side into the left-hand side in order to obtain the following inequality
        $$
        \|D_h^+f\|^2_{\ell^2_h} \lesssim  \eps |D_hf|^2_{\dot H^s_h} + \eps^{1-2^{q+1}} \|f\|^2_{H^{s}_h}\,,
        $$
        which concludes the proof. 
    \end{proof}

    \subsection{A discrete interpolation inequality in weighted spaces}
    In the next proposition, we prove a result which will be one of the keystones of the proof of $H^1$-coercivity in Theorem~\ref{theo:coerc_discrLFP} {(and thus also of Theorems~\ref{theo:hypo_discrLFP} and~\ref{theo:hypo_fulldiscrLFP})}. This result is a consequence of Proposition~\ref{prop:inegfunc_discr2} (which comes from Poincar\'e and regularization estimates) and the interpolation inequality obtained in Theorem~\ref{theo:interp}. 
    \begin{prop} \label{prop:useful}
        There exist $\eta>0$ and $h_0>0$ such that for any $h \in (0,h_0)$ and for any $\eps\in(0,\eta)$, there is $K(\eps)>0$ such that  
        \begin{align*}
            &\|f\|_{\ell^2_h(M^{-1})} \|D_hf\|_{\ell^2_h(M^{-1})} + \|D_hf\|_{\ell^2_h(M^{-1})}^2\ \\
            &\quad \leq\ K(\ve) \left(\Symh(f,f)\,+\,\|\Pi_h f\|_{\ell^2_h(M^{-1})}^2\right)+\, \ve\,\Symh(D_hf,D_hf)
        \end{align*}
        where we recall that $(\Pi_hf)_j = M_j\,(\sum_{k\in\ZZ}f_k\,h)/(\sum_{k\in\ZZ}M_k\,h)$.
    \end{prop}
    \begin{proof}
        It is actually enough to prove that 
        \begin{equation} \label{ineq:Dhf}
            \|D_hf\|_{\ell^2_h(M^{-1})}^2\ 
            \leq\ K(\ve) \left(\Symh(f,f)\,+\,\|\Pi_h f\|_{\ell^2_h(M^{-1})}^2\right)+\, \ve\,\Symh(D_hf,D_hf)\,.
        \end{equation}
        Indeed, let us assume that the latter inequality is proved. Then, we can write the following bound:
        $$
        \|f\|_{\ell^2_h(M^{-1})} \|D_hf\|_{\ell^2_h(M^{-1})} 
        \le \left(\|f-\Pi_hf\|_{\ell^2_h(M^{-1})} + \,\|\Pi_h f\|_{\ell^2_h(M^{-1})}  \right) \|D_hf\|_{\ell^2_h(M^{-1})} 
        $$
        so that using~\eqref{eq:discr_poinc} and~\eqref{ineq:Dhf},
        \begin{align*}
            & \|f\|_{\ell^2_h(M^{-1})} \|D_hf\|_{\ell^2_h(M^{-1})} 
            \le \left(\Symh (f,f)^{1/2} + \, \|\Pi_h f\|_{\ell^2_h(M^{-1})}\right) \|D_hf\|_{\ell^2_h(M^{-1})} \\
            &\quad \le K(\ve)^{1/2}\left(\Symh (f,f)^{1/2} + \, \|\Pi_h f\|_{\ell^2_h(M^{-1})}\right)  \left(\Symh(f,f)\,+\,\|\Pi_h f\|_{\ell^2_h(M^{-1})}^2\right)^{1/2}\\
            &\qquad + \ve^{1/2}\left(\Symh (f,f)^{1/2} + \, \|\Pi_h f\|_{\ell^2_h(M^{-1})}\right)\Symh(D_hf,D_hf)\big)^{1/2} \\
            &\quad \lesssim (K(\ve)^{1/2}+1) \left(\Symh(f,f)\,+\,\|\Pi_h f\|_{\ell^2_h(M^{-1})}^2\right) + \, \ve \, \Symh(D_hf,D_hf)\,,
        \end{align*}
        which is exactly the wanted result up to changing the value of $\eps$. 
        
        Let us now come to the proof of~\eqref{ineq:Dhf}. We first notice that:
        \begin{align*}
            &\|D_hf\|_{\ell^2_h(M^{-1})}^2 
            = \sum_{j \in \ZZ} |(D_h f)_j M_{j}^{-1/2}|^2  h \\     & \qquad \lesssim \sum_{j \in \ZZ} \left(|(D_h (fM^{-1/2}))_j|^2  + |f_{j+1}(D_h^+ M^{-1/2})_j|^2 +|f_{j-1}(D_h^+ M^{-1/2})_{j-1}|^2\right) h\,.
        \end{align*}
        Then using Corollary~\ref{cor:DhM-1/2}, we deduce that 
        \[
        \|D_hf\|_{\ell^2_h(M^{-1})}^2 \lesssim \|D_h(fM^{-1/2})\|^2_{\ell^2_h} + \|fM^{-1/2} \|^2_{\ell^2_h}\,. 
        \]
        From Theorem~\ref{theo:interp}, we deduce that there exists $\eta>0$ such that for any $\eps \in (0,\eta)$, there is $K(\eps)>0$ such that  
        \[
        \|D_hf\|_{\ell^2_h(M^{-1})}^2 \lesssim \eps \|D_h(fM^{-1/2})\|^2_{H^{\alpha/2}_h} + K(\eps) \|f M^{-1/2}\|^2_{H^{\alpha/2}_h}\,. 
        \]
        Then, using that $(D_h (fM^{-1/2}))_j = (D_h f)_j M^{-1/2}_{j+1} + f_{j-1} (D_h M^{-1/2})_j$, up to changing $K(\eps)$ and using Corollary~\ref{cor:DhM-1/2}, we get: 
        \[
        \|D_hf\|_{\ell^2_h(M^{-1})}^2 \lesssim \eps \|(D_hf)M^{-1/2}\|^2_{H^{\alpha/2}_h} + K(\eps) \|f M^{-1/2}\|^2_{H^{\alpha/2}_h}\,,
        \]
        where we used that $M_{j+1}^{-1/2} \lesssim M_j^{-1/2} \lesssim M_{j-1}^{-1/2}$ holds for $h$ small enough. Now observe that 
        \[
        \|f M^{-1/2}\|^2_{H^{\alpha/2}_h} \lesssim \|(f-\Pi_h f) M^{-1/2}\|^2_{H^{\alpha/2}_h} + \|(\Pi_h f)M^{-1/2}\|^2_{H^{\alpha/2}_h} \,.
        \]
        Moreover, we have:
        \[
        \|(\Pi_h f)M^{-1/2}\|^2_{H^{\alpha/2}_h} = \|\Pi_h f\|^2_{\ell^2_h(M^{-1})} \|M^{1/2}\|_{H^{\alpha/2}_h}^2\,.
        \]
        We also have from Proposition~\ref{prop:embed}-(ii),~\eqref{eq:mualpha_est1} in Proposition~\ref{prop:bounds_mua} and Corollary~\ref{cor:DhM-1/2},
        \[
        \|M^{1/2}\|_{H^{\alpha/2}_h} \lesssim \|M^{1/2}\|_{\ell^2_h} + \|D_h^+(M^{1/2})\|_{\ell^2_h} \lesssim 1
        \]
        so that 
        \[
        \|(\Pi_h f)M^{-1/2}\|^2_{H^{\alpha/2}_h} \lesssim \|\Pi_h f\|^2_{\ell^2_h(M^{-1})}\,.
        \]
        Finally, to conclude, we remark that $D_h f = D_h f - \Pi_h D_hf$ and we use twice Proposition~\ref{prop:inegfunc_discr2}.
    \end{proof}

    \section{Stability and long-time behavior}\label{sec:stablong}
    \subsection{Main results}
    
    In the first theorem, we look at a semi-discretized version of~\eqref{eq:def_scheme_hom}. Since only the variable $v$ is discretized, to simplify the notations, we note $h=\Delta v$ and we thus study the equation
    \begin{equation} \label{eq:semihom}
        \partial_t f_j \ =\ (L_\alpha^h f)_{j}\,,\quad\forall j \in\ZZ\,,
    \end{equation}
    with some given initial data $(f_j^0)_j$. \\
    
    { \begin{theo}\label{theo:coerc_discrLFP}
            { There exists $h_0>0$ such that} if $f$ is a solution of the semi-discrete L\'evy-Fokker-Planck equation \eqref{eq:semihom} with initial data~$(f_j^0)_j \in H^1_h(M^{-1})$ then, for all $t\geq0$ { and $h \in (0,h_0)$} one has
            \[
            \|f(t) - f^\infty\|_{H^1_h(M^{-1})}\ \leq\ C\,\|f^0 - f^\infty\|_{H^1_h(M^{-1})}\,e^{-\lambda t}\,
            \]
            where \[f^\infty  := \frac{\lla f^0 \rra_h}{\lla M \rra_h} \, M\,\quad\text{with}\quad
            \lla f \rra_h := \sum_{j \in \ZZ} f_j \, h 
            \]
            for some constants $C\geq1$ and $\lambda>0$ depending only on $\alpha$. 
        \end{theo}
        
        \begin{rema} \label{rem:semihom}
            Recall that the projection $\Pi_h$ is defined through $(\Pi_hf)_j=M_j\,\lla f \rra_h / \lla M \rra_h$ so that $f^\infty =\Pi_hf^0$. Remark that as in Proposition~\ref{prop:elementary}, one can prove that the global mass is conserved by the equation: if $f$ is a  solution to~\eqref{eq:semihom}, then for all $t \geq 0$, one has 
            $
            \lla f(t) \rra_{h} = \lla f^0 \rra_{h}
            $.
        \end{rema}}

    In the second theorem, we consider a semi-discretized version of \eqref{eq:def_scheme_kin}. Now, both $x$ and $v$ variables appear in the equation, we thus study the equation 
    \begin{equation} 
        \label{eq:semikin} 
        \partial_t f_{i,j} + \left(T^\dx f\right)_{i,j} = \left(L^\dv_\alpha f\right)_{i,j}, \;\forall (i,j)\in\ZZ/N_x\ZZ \times \ZZ\,,
    \end{equation}
    where $L^\dv_\alpha$ and $T^\dx$ are respectively defined in \eqref{def_LFPdiscr} and \eqref{eq:def_transport},
    with some given initial data $(f^0_{i,j})_{i,j}$. In order to prove hypocoercivity estimates for \eqref{eq:semikin}, the scalar product 
    \[
    \lla f,g \rra_\lm = \sum\limits_{i\in\ZZ/N_x \ZZ} \sum\limits_{j\in\ZZ} \frac{f_{i,j} g_{i,j}}{M_j}  \dx \dv \, ,
    \]
    for $f=(f_{i,j})_{(i,j)\in\ZZ/N_x\ZZ \times \ZZ}$ and $g=(g_{i,j})_{(i,j)\in\ZZ/N_x\ZZ \times \ZZ}$ is introduced as well as the associated norm which is denoted by $\|\cdot\|_\lm$. In what follows, we prove that the solution $(f(t))_{t\ge 0}$ of \eqref{eq:semikin} has exponential decay to equilibrium for a discrete $H^1$ weighted norm defined by 
    \[
    \| f\|^2_\hm = \|f\|^2_\lm + \| D_\dx f\|^2_\lm + \|D_\dv f \|^2_\lm\,,
    \]
    where $D_\dx$ and $D_\dv$ stand for centered finite differences in space and velocity
    \[
    \forall (i,j)\in\ZZ/N_x\ZZ\times \ZZ, \; (D_\dx f)_{i,j} = \frac{f_{i+1,j}-f_{i-1,j}}{2\dx}, \;\;  (D_\dv f)_{i,j}=\frac{f_{i,j+1}-f_{i,j-1}}{2\dv}\,.
    \]

    \begin{theo}\label{theo:hypo_discrLFP}
        Suppose that $N_x$ is odd. There exists $\Delta v_0>0$ such that  if $f$ is  solution of the semi-discrete kinetic Lévy-Fokker-Planck equation \eqref{eq:semikin} with initial data $(f^0_{i,j})_{i,j}\in\hm$ then, for all $\dv<\dv_0$ and for all $t\ge 0$, one has 
        \[
        \|f(t) - f^\infty\|_\hm 
        \le C \|f^0-f^\infty\|_\hm\;e^{-\lambda t},
        \]
        where  
        \[
        f^\infty  := \frac{\lla f^0 \rra_{\dx,\dv}}{\lla M \rra_{\dx,\dv}} \, M\,\quad\text{with}\quad
        \lla f \rra_{\dx,\dv} := \sum_{(i,j)\in\ZZ/N_x\ZZ \times \ZZ}  f_{i,j} \, \dx \dv
        \]
        for some constant $C\ge 1$  and $\lambda>0$ depending only on $\alpha$.
    \end{theo}
    
    {\begin{rema} \label{rem:semikin}
            In the inhomogeneous setting one can define the macroscopic density $\rho_i := \sum_{j \in \ZZ} f_{i,j} \, \dv$. It relates to the projection $\Pi_{\Delta v}$ by the relation $\rho_i\,M_j\ =\  \lla M \rra_{\Delta v}(\Pi_{\Delta v}f_{i,\cdot})_j$. Remark also that as in Proposition~\ref{prop:elementary}, one can prove that the global mass is conserved by the equation: if $f$ is a  solution to~\eqref{eq:semikin}, then for all $t \geq 0$, one has $\lla f(t) \rra_{\dx,\dv} = \lla f^0 \rra_{\dx,\dv}$.
        \end{rema}}
    {Notice that Theorem~\ref{theo:coerc_discrLFP} can be seen as a special case of Theorem~\ref{theo:hypo_discrLFP}. However, we choose to present Theorem~\ref{theo:coerc_discrLFP} as well as its proof in order to highlight the main arguments that allow us to treat the collision operator $\Lah$ in our energy estimates. Indeed, the said arguments may be somewhat hidden in the proof of the kinetic case in which we face additionnal difficulties coming from the transport operator and the $H^1$-{\em hypo}coercivity method.}

    Finally, in the last theorem, we consider the fully discrete implicit in time discretization of \eqref{eq:kinfracFP}, that we recall here,
    \begin{equation} 
        \label{eq:fulldiscretekin} 
        \frac{f_{i,j}^{n+1} - f_{i,j}^{n}}{\Delta t} + \left(T^\dx f\right)_{i,j}^{n+1} = \left(L^\dv_\alpha f\right)_{i,j}^{n+1}, \;\forall (i,j)\in\ZZ/N_x\ZZ \times \ZZ, n \in \mathbb{N}.
    \end{equation}
    Then, we obtain the following result.
    \begin{theo}\label{theo:hypo_fulldiscrLFP}
        Suppose that $N_x$ is odd. There exists $\Delta v_0>0$ such that if $f$ is a  solution of the  discrete kinetic Lévy-Fokker-Planck equation \eqref{eq:fulldiscretekin} with initial data $(f^0_{i,j})_{i,j}\in\hm$, then for all $\dv<\dv_0$ and for all $n \in \mathbb{N}$, one has 
        \[
        \|f^n - f^\infty\|_\hm 
        \le C \|f^0-f^\infty\|_\hm\;(1+2  \lambda \Delta t)^{- \frac {n}{2}},
        \]
        for some constants $C\ge 1$ and  $\lambda>0$  depending only on $\alpha$. The global equilibrium $f^\infty$ is the same as in Theorem~\ref{theo:hypo_discrLFP}.
    \end{theo}
    
    \begin{rema}
     It can be noticed, that the constant $\lambda>0$ in Theorem \ref{theo:hypo_fulldiscrLFP} is the same as the one in Theorem \ref{theo:hypo_discrLFP}.
    \end{rema}

    \subsection{A technical lemma}
    Before starting the proof of our results, we prove an estimate for the commutator $[D_h, L^h_\alpha]$ which naturally arises in our computations. It will be useful in the proof of Theorem~\ref{theo:coerc_discrLFP} and it is worth remarking that the proof is similar in the non-homogeneous case, so that it will also be used in the proof of Theorem \ref{theo:hypo_discrLFP}. In the continuous setting, the corresponding commutator, $[\partial_v, L_\alpha] = \partial_v$, is easily computed. While the discrete case is more intricate, we are still able to obtain nice estimates on the scalar product $\lla [D_h,L^h_\alpha] f, g\rra_{\ell^2_h(M^{-1})}$.

    \begin{lem} 
        \label{lem:estimate_commutator}
        For any $f,g\in\ell^2_h(M^{-1})$, 
        \begin{equation}
            \label{eq:estimate_commutator} 
            \left| \lla [D_h,L^h_\alpha] f, g\rra_{\ell^2_h(M^{-1})}\right| \lesssim \|f\|_{\ell^2_h(M^{-1})}\|g\|_{\ell^2_h(M^{-1})} + \|D_h f\|_{\ell^2_h(M^{-1})} \| g\|_{\ell^2_h(M^{-1})}\,.
        \end{equation}
    \end{lem}
    
    \begin{proof}
        Let $f$, $g\in\ell^2_h(M^{-1})$. The proof of Lemma \ref{lem:estimate_commutator} relies on an appropriate rewriting of $\lla [D_h, L^h_\alpha] f, g\rra$, and on the bounds on the discrete equilibrium given in Section \ref{subsec:BoundsOnTheDiscreteEquilibrium}. First of all, let us notice that we clearly have   $[D_h, \Lambda^h_\alpha]=0$. Therefore
        \[
        \lla [D_h, L^h_\alpha] f, g\rra_{\ell^2_h(M^{-1})} 
        = \lla [D_h, \Gamma^h_\alpha] f, g\rra_{\ell^2_h(M^{-1})}
        = \lla D_h \Gamma^h_\alpha f,g\rra_{\ell^2_h(M^{-1})} - \lla \Gamma^h_\alpha D_h f, g\rra_{\ell^2_h(M^{-1})}\,.
        \]
        Let us study each term of the right-hand side. For the first term, one has 
        \[
        \lla D_h \Gamma^h_\alpha f,g\rra_{\ell^2_h(M^{-1})}
        = \sum\limits_{j\in\ZZ} \frac{h}{M_j} \frac{(\Gamma^h_\alpha f)_{j+1}-(\Gamma^h_\alpha f)_{j-1}}{2h} g_j\,,
        \]
        which yields, according to the definition of $\Gamma^h_\alpha$ in \eqref{def:Gamma}-\eqref{def:flux},
        \begin{align}
            \label{eq:AnotherUsefulEstimate_A}
            \lla D_h  \Gamma^h_\alpha f,g \rra_{\ell^2_h(M^{-1})} &= \sum\limits_{j\in\ZZ} \frac{h}{M_j} (VM)_{j+3/2} \frac{f_{j+1}/M_{j+1} + f_{j+2}/M_{j+2}}{(2h)^2} g_j
            \\&\quad
            - \sum\limits_{j\in\ZZ} \frac{h}{M_j} (VM)_{j+1/2} \frac{f_j/M_j + f_{j+1}/M_{j+1}}{(2h)^2} g_j \nonumber 
            \\&\quad
            - \sum\limits_{j\in\ZZ} \frac{h}{M_j} (VM)_{j-1/2} \frac{f_{j-1} /M_{j-1}+ f_j/M_j}{(2h)^2} g_j \nonumber
            \\&\quad
            + \sum\limits_{j\in\ZZ} \frac{h}{M_j} (VM)_{j-3/2}  \frac{f_{j-2}/M_{j-2} + f_{j-1}/M_{j-1}}{(2h)^2} g_j\,. \nonumber
        \end{align}
        For the second term, we have
        \begin{align}
            \label{eq:AnotherUsefulEstimate_B}
            \lla \Gamma^h_\alpha D_h f, g\rra_{\ell^2_h(M^{-1})} &=
            \sum\limits_{j\in\ZZ} \frac{h}{M_j} (VM)_{j+1/2} \left( \frac{1}{M_{j+1}} \frac{f_{j+2} -f_j}{(2h)^2} + \frac{1}{M_j} \frac{f_{j+1}-f_{j-1}}{(2h)^2} \right) g_j
            \\&\quad
            - \sum\limits_{j\in\ZZ} \frac{h}{M_j}(VM)_{j-1/2} \left( \frac{1}{M_j} \frac{f_{j+1}-f_{j-1}}{(2h)^2} +\frac{1}{M_{j-1}} \frac{f_j-f_{j-2}}{(2h)^2} \right) g_j\,. \nonumber
        \end{align}
        Now, we sum  \eqref{eq:AnotherUsefulEstimate_A} and \eqref{eq:AnotherUsefulEstimate_B}, and we reorganize it to gather the terms in $f_j, f_{j-1}, f_{j+1}$ and $f_{j+2}$ or~$f_{j-2}$. Hence
        \[
        \lla [D_h,\Gamma^h_\alpha] f,g\rra_{\ell^2_h(M^{-1})} = C_0 + C_{-1} + C_1 + C_{-2,2}\,,
        \]
        where 
        \begin{align}
            \label{eq:AnotherUsefulEstimate_C0} 
            C_0&= \sum\limits_{j\in\ZZ} \frac{h}{M_j} \frac{f_j g_j}{(2h)^2} \left( 
            \frac{(VM)_{j-1/2}}{M_{j-1}} - \frac{(VM)_{j-1/2}}{M_j} + \frac{(VM)_{j+1/2}}{M_{j+1}} -\frac{(VM)_{j+1/2}}{M_j}
            \right) 
            \\
            \label{eq:AnotherUsefulEstimate_Cmoins1} 
            C_{-1} &= \sum\limits_{j\in\ZZ} \frac{h}{M_j} \frac{f_{j-1} g_j}{(2h)^2} \left( 
            \frac{(VM)_{j+1/2}}{M_j}- \frac{(VM)_{j-1/2}}{M_j}
            -\frac{(VM)_{j-1/2}}{M_{j-1}}+ \frac{(VM)_{j-3/2}}{M_{j-1}} 
            \right)
            \\
            \label{eq:AnotherUsefulEstimate_C1}
            C_1&= \sum\limits_{j\in\ZZ} \frac{h}{M_j} \frac{f_{j+1}g_j}{(2h)^2} \left( 
            \frac{(VM)_{j-1/2}}{M_j} - \frac{(VM)_{j+1/2}}{M_j}- \frac{(VM)_{j+1/2}}{M_{j+1}} + \frac{(VM)_{j+3/2}}{M_{j+1}}
            \right)
            \\
            \label{eq:AnotherUsefulEstimate_C2}
            C_{-2,2} &=  \sum\limits_{j\in\ZZ} \frac{h}{M_j} \frac{f_{j+2} g_j}{(2h)^2} \left(  
            \frac{(VM)_{j+3/2}}{M_{j+2}} - \frac{(VM)_{j+1/2}}{M_{j+1}}
            \right) 
            \\ \nonumber &\quad
            + \sum\limits_{j\in\ZZ} \frac{h}{M_j} \frac{f_{j-2} g_j}{(2h)^2} \left( 
            \frac{(VM)_{j-3/2}}{M_{j-2}} - \frac{(VM)_{j-1/2}}{M_{j-1}}
            \right)\,,
        \end{align}
        and we estimate all the terms separately. Let us start with $C_0$, and rewrite it as 
        \begin{align*}
            C_0&= \sum\limits_{j\in\ZZ} \frac{h}{4M_j}\; f_j\; g_j \;\frac{(VM)_{j+1/2}-(VM)_{j-1/2}}{h}\; \frac{1}{h}\left( \frac{1}{M_{j+1}}-\frac{1}{M_j}\right)
            \\&\quad
            + \sum\limits_{j\in\ZZ} \frac{h}{4M_j}\; f_j \; g_j\; (VM)_{j-1/2}\; \frac{1}{h^2} \left( \frac{1}{M_{j+1}}-\frac{2}{M_j}+ \frac{1}{M_{j-1}}\right)\,.
        \end{align*}
       Thanks to Corollaries \ref{cor:DhM-1/2} and \ref{cor:D2h1/M}, and to Lemmas~\ref{lem:VM} and \ref{lem:DhVM}, the following inequality holds
        \[
        |C_0|\lesssim \sum\limits_{j\in\ZZ} \frac{h}{M_j} \; |f_j|\;|g_j|\,,
        \]
        which yields $|C_0|\lesssim \|f\|_{\ell^2_h(M^{-1})} \|g\|_{\ell^2_h(M^{-1})}$ with Cauchy-Schwarz inequality.
        Going on with $C_1$ and $C_{-1}$, we rewrite them as 
        \begin{align*}
            C_1&=  \sum\limits_{j\in\ZZ} \frac{h}{4M_j}\; f_{j+1} g_j \; \frac{(VM)_{j+3/2}-(VM)_{j+1/2}}{h} \;\frac{1}{h}\left( \frac{1}{M_{j+1}}- \frac{1}{M_j} \right)
            \\&\quad
            + \sum\limits_{j\in\ZZ} \frac{h}{4M_j} \; f_{j+1}\; g_j \; \frac{1}{M_j} \;\frac{ (VM)_{j+3/2} - 2 (VM)_{j+1/2}+ (VM)_{j-1/2} }{h^2} =: C_{1}^a + C_{1}^b\,,
        \end{align*}
        { and similarly
            \begin{align*}
                C_{-1}&=  \sum\limits_{j\in\ZZ} \frac{h}{4M_j}\; f_{j-1} g_j \; \frac{(VM)_{j-3/2}-(VM)_{j-1/2}}{h} \;\frac{1}{h}\left( \frac{1}{M_{j-1}}- \frac{1}{M_j} \right)
                \\&\quad
                + \sum\limits_{j\in\ZZ} \frac{h}{4M_j} \; f_{j-1}\; g_j \; \frac{1}{M_j} \;\frac{ (VM)_{j+1/2} - 2 (VM)_{j-1/2}+ (VM)_{j-3/2} }{h^2} =: C_{-1}^a + C_{-1}^b \, .
            \end{align*}
            Thanks to Corollary \ref{cor:DhM-1/2} and Lemma \ref{lem:DhVM}, one has
            \[
            |C_{1}^a|\lesssim \sum\limits_{j\in\ZZ} \frac{h}{M_j}\; |f_{j+1}|\; |g_j| \lesssim \left( \sum\limits_{j\in\ZZ} \frac{h}{M_{j+1}} \;\frac{M_{j+1}}{M_j} f^2_{j+1}\right)^{1/2} \;\|g\|_{\ell^2_h(M^{-1})}\,.
            \]
            The estimate 
            \[
            |C_{1}^a|\lesssim \|f\|_{\ell^2_h(M^{-1})} \;\|g\|_{\ell^2_h(M^{-1})}
            \]
            comes thanks to Lemma \ref{lem:DhM} and the identity $M_{j+1}/M_j= 1+ h (D_h^+ M)_j/M_j$. Similarly, one has 
            \[
            |C_{-1}^a|\lesssim \|f\|_{\ell^2_h(M^{-1})} \;\|g\|_{\ell^2_h(M^{-1})}\,.
            \]
            Then, to handle $C_{1}^b+C_{-1}^b$, we first remark that 
            \begin{align*}
                &(VM)_{j+3/2} - 2 (VM)_{j+1/2}+ (VM)_{j-1/2} \\
                &\quad = (VM)_{j+3/2} - (VM)_{j+1/2} - (VM)_{j-1/2} + (VM)_{j-3/2} \\
                &\qquad - \left((VM)_{j+1/2} - 2 (VM)_{j-1/2}+ (VM)_{j-3/2}\right)\, .
            \end{align*}
            From this, we deduce that 
            \begin{align*}
                C_{1}^b &= \sum\limits_{j\in\ZZ} \frac{h}{4M_j} \; f_{j+1}\; g_j \; \frac{1}{M_j} \;\frac{ (VM)_{j+3/2} -  (VM)_{j+1/2} -(VM)_{j-1/2} +(VM)_{j-3/2}}{h^2} \\
                &\quad - \sum\limits_{j\in\ZZ} \frac{h}{4M_j} \; f_{j+1}\; g_j \; \frac{1}{M_j} \;\frac{ (VM)_{j+1/2} - 2 (VM)_{j-1/2}+ (VM)_{j-3/2}}{h^2} =: C_{1}^c + C_{1}^d\,.
            \end{align*}
            As previously, Lemma~\ref{lem:D2hVM} implies that 
            \[
            |C_{1}^c| \lesssim \|f\|_{\ell^2_h(M^{-1})} \;\|g\|_{\ell^2_h(M^{-1})}\,. 
            \]
            Finally, we write that 
            \[
            C_{1}^d + C_{-1}^b =  \sum\limits_{j\in\ZZ} \frac{h}{2M_j} \; \frac{f_{j-1}-f_{j+1}}{2h}\; g_j  \; \frac{1}{M_j} \;\frac{ (VM)_{j+1/2} - 2 (VM)_{j-1/2}+ (VM)_{j-3/2} }{h} 
            \]
            and remark that 
            \begin{align*}
            &(VM)_{j+1/2} - 2 (VM)_{j-1/2}+ (VM)_{j-3/2} \\
            &\quad= (VM)_{j+1/2} -  (VM)_{j-1/2}-\left((VM)_{j-1/2} - (VM)_{j-3/2}\right) \, . 
            \end{align*}
            Then, Lemma~\ref{lem:DhVM} allows us to conclude that 
            \[
            \left|C_{1}^d + C_{-1}^b \right| \lesssim\|D_h f\|_{\ell^2_h(M^{-1})} \| g\|_{\ell^2_h(M^{-1})}\, . 
            \]
        }
        
        Eventually, $C_{-2,2}$ is once again decomposed as 
        \[
        C_{-2,2}=C_{-2,2}^a+ C_{-2,2}^b\,,
        \]
        where 
        \begin{align}
            \label{eq:AnotherUsefulEstimate_C2a} 
            C_{-2,2}^a&=\sum\limits_{j\in\ZZ} \frac{h}{M_j}\;\frac{g_j}{(2h)^2}  
            \frac{f_{j+2}}{M_{j+2}} \left( (VM)_{j+3/2} - (VM)_{j+1/2} \right)
            \\ \nonumber &\quad - \sum\limits_{j\in\ZZ} \frac{h}{M_j}\;\frac{g_j}{(2h)^2} \frac{f_{j-2}}{M_{j-2}} \left( (VM)_{j-1/2} - (VM)_{j-3/2}
            \right) 
            \end{align}
and
            \begin{align}
            \label{eq:AnotherUsefulEstimate_C2b}
            C_{-2,2}^b &= \sum\limits_{j\in\ZZ} \frac{h}{M_j} \;\frac{f_{j+2} g_j}{(2h)^2} \; (VM)_{j+1/2} \left( \frac{1}{M_{j+2}} - \frac{1}{M_{j+1}}\right) 
            \\ \nonumber &\quad + \sum\limits_{j\in\ZZ} \frac{h}{M_j} \frac{f_{j-2} g_j}{(2h)^2} \; (VM)_{j-1/2} \left( \frac{1}{M_{j-2}}- \frac{1}{M_{j-1}}\right)\,,
        \end{align}
        and each term is considered independently. Let us remark that 
        \begin{align*}
            C_{-2,2}^a&= \sum\limits_{j\in\ZZ} \frac{h}{4M_j}\; g_{j} \; \frac{1}{M_{j+2}} \;\frac{(VM)_{j+3/2}-(VM)_{j+1/2}}{h}\;  \frac{f_{j+2}-f_{j-2}}{h} 
            \\&\quad + \sum\limits_{j\in\ZZ} \frac{h}{4M_j} \; g_j \;f_{j-2}\; \frac{(VM)_{j+3/2}- (VM)_{j+1/2}}{h} \;\frac{1}{h}\left( \frac{1}{M_{j+2}}- \frac{1}{M_{j-2}}\right)
            \\&\quad + \sum\limits_{j\in\ZZ} \frac{h}{4M_j}\; g_j \; f_{j-2} \; \frac{1}{M_{j-2}} \frac{(VM)_{j+3/2} - (VM)_{j+1/2} - (VM)_{j-1/2} + (VM)_{j-3/2}}{h^2}\,,
        \end{align*}
        and since $(f_{j+2}-f_{j-2})/h = (D_h f)_{j+1} + (D_h f)_{j-1}$ and $(M_{j+2}^{-1} - M_{j-2}^{-1})/h= \sum_{k=-2}^1 (D_h^+ M^{-1})_{j+k}$, Corollary \ref{cor:DhM-1/2} and Lemmas \ref{lem:DhVM} and \ref{lem:D2hVM} yield 
        \begin{align*}
            |C_{-2,2}^a|\lesssim \sum_{j\in\ZZ} \frac{h}{M_j} \; |g_j| \; \left(|(S_+ D_hf)_j | + |(S_- D_h f)_j|+ |(S_-S_- f)_j|\right)\,,
        \end{align*}
        where $(S_+ f)_j = f_{j+1}$ and $(S_- f)_j = f_{j-1}$. The conclusion comes from
        \[
        |C_{-2,2}^a| \lesssim \|g\|_{\ell^2_h(M^{-1})} \left( \|S_+D_h f\|_{\ell^2_h(M^{-1})} + \| S_- D_h f\|_{\ell^2_h(M^{-1})} + \|S_- S_- f\|_{\ell^2_h(M^{-1})}  \right)\,,
        \]
        and 
        \[
        \|S_+ f \|^2_{\ell^2_h(M^{-1})} - \|f\|^2_{\ell^2_h(M^{-1})} = \sum\limits_{j\in\ZZ} \frac{h}{M_j } f^2_j  \; h\frac{(D^+_hM)_{j-1}}{M_{j-1}}\,.
        \]
        Indeed, thanks to Lemma \ref{lem:DhM}, the previous identity yields
        \[
        \left|\|S_+ f \|^2_{\ell^2_h(M^{-1})} - \|f\|^2_{\ell^2_h(M^{-1})}\right|\lesssim h \|f\|^2_{\ell^2_h(M^{-1})}\,,
        \]
        so that $\|S_+ f\|_{\ell^2_h(M^{-1})} \lesssim \|f\|_{\ell^2_h(M^{-1})}$. The same result holds for $S_-$, and we eventually have
        \[
        |C_{-2,2}^a|\lesssim \|g\|_{\ell^2_h(M^{-1})} \|f\|_{\ell^2_h(M^{-1})} + \|g\|_{\ell^2_h(M^{-1})} \|D_h f\|_{\ell^2_h(M^{-1})}\,.
        \]
        To conclude this proof, $C_{-2,2}^b$ is rewritten as 
        \begin{align*}
            C_{-2,2}^b &= \sum\limits_{j\in\ZZ} \frac{h}{4M_j}\; g_j \; (VM)_{j+1/2} \;\frac{f_{j+2}-f_{j-2}}{h}\;\frac{1}{h}\left(\frac{1}{M_{j+2}} - \frac{1}{M_{j+1}} \right) 
            \\ &\quad + \sum\limits_{j\in\ZZ} \frac{h}{4M_j}\;g_j\; f_{j-2}\; \frac{(VM)_{j+1/2}-(VM)_{j-1/2}}{h} \;\frac{1}{h} \left( \frac{1}{M_{j+2}}- \frac{1}{M_{j+1}} \right)
            \\&\quad +\sum\limits_{j\in\ZZ} \frac{h}{4M_j}\; g_j\; f_{j-2}\; (VM)_{j-1/2}\;  \frac{1}{h^2}\left( \frac{1}{M_{j+2}} - \frac{1}{M_{j+1}} - \frac{1}{M_{j-1}} + \frac{1}{M_{j-2}} \right)\,,
        \end{align*}
        and since 
        \begin{align*}
            \frac{1}{M_{j+2}} - \frac{1}{M_{j+1}} - \frac{1}{M_{j-1}} + \frac{1}{M_{j-2}}= \sum\limits_{k=-1}^1 \left(\frac{1}{M_{j+k+1}}-\frac{2}{M_{j+k}}+ \frac{1}{M_{j+k-1}}\right)\,,
        \end{align*}
        we obtain using Lemmas \ref{lem:VM}-\ref{lem:DhVM} and  Corollaries \ref{cor:DhM-1/2}-\ref{cor:D2h1/M}
        that
        \[
        |C_{-2,2}^b|\lesssim \sum\limits_{j\in\ZZ} \frac{h}{M_j} |g_j| \left( |(S_+ D_h f)_j | + |(S_- D_h f)_j| + |(S_-S_- f)_j|   \right)\,.
        \]
        Eventually, Cauchy-Schwarz inequality and the above estimate for $\|S_\pm f\|_{\ell^2_h(M^{-1})}$ yield
        \[
        |C_{-2,2}^b |\lesssim \|g\|_{\ell^2_h(M^{-1})} \|f\|_{\ell^2_h(M^{-1})}+\|g\|_{\ell^2_h(M^{-1})}\|D_h f\|_{\ell^2_h(M^{-1})}\, .
        \]
    \end{proof}
    
    \subsection{$H^1$-coercivity in the homogenous case: proof of Theorem~\ref{theo:coerc_discrLFP}.} \label{subsec:coercivity}
    
    In the next proposition, we give an estimate on the evolution of the norm of the solution to equation~\eqref{eq:semihom} and on the evolution of its derivative. Notice that in the continuous case, this type of estimates comes from commutator equalities that are quite simple. Here, the commutators are much more complicated but some simplifications occur when we compute the evolution of the derivative of the solution and we are still able to get nice bounds. 
    
    \begin{prop} \label{prop:f&Dhf}
        Let $f$ be a solution to~\eqref{eq:semihom} with $L_\alpha$ defined in~\eqref{def_LFPdiscr}-\eqref{def:VM}. Then we have:
        $$
        {\frac12} \frac{\dd}{\dd t} \|f\|^2_{\ell^2_h(M^{-1})} = - \Symh (f,f)
        $$
        and there exists $C>0$ (depending only on $\alpha$) such that
        \begin{align*}
            {\frac 12} \frac{\dd}{\dd t} \|D_h f\|^2_{\ell^2_h(M^{-1})} \le &- \Symh (D_h f,D_h f) +C \|D_h f\|^2_{\ell^2_h(M^{-1})} 
             + C \|f\|_{\ell^2_h(M^{-1})} \|D_hf\|_{\ell^2_h(M^{-1})} \, .
        \end{align*}
    \end{prop}
    
    \begin{proof}
        The first equality comes directly from the decomposition given in~\eqref{eq:Lalphabilinear} since 
        $$
        {\frac12} \frac{\dd}{\dd t} \|f\|^2_{\ell^2_h(M^{-1})} = \lla\Lah f,f \rra_{\ell^2_h(M^{-1})}\,.
        $$ 
        
        Concerning the second part of the proposition, we start by writing that 
        \begin{align*}
            &{\frac12} \frac{\dd}{\dd t} \|D_hf\|^2_{\ell^2_h(M^{-1})} \\
            \quad &= \lla D_h\Lah f,D_h f \rra_{\ell^2_h(M^{-1})} = \lla [D_h,\Lah] f,D_hf \rra_{\ell^2_h(M^{-1})} + \lla \Lah D_h f, D_h f \rra_{\ell^2_h(M^{-1})}\,.
        \end{align*}
        The second term is computed exactly as previously: 
        $$
        \lla \Lah D_h f, D_h f \rra = - \Symh(D_h f ,D_h f)\,,
        $$
        while the first one is estimated with Lemma \ref{lem:estimate_commutator} to obtain the wanted result.
    \end{proof}
    
    We are now able to conclude the proof of the main result of this part. 
    
    \smallskip
    \noindent {\it Proof of Theorem~\ref{theo:coerc_discrLFP}.}
    As in the continuous case, we introduce a functional which is going to be an entropy for our equation:
    \[
    \mathcal{F}(f) := \|f\|^2_{\ell^2_h(M^{-1})} + a \|D_h f\|^2_{\ell^2_h(M^{-1})}
    \]
    for some positive constant $a$ which will be chosen later on. Notice first that we clearly have the equivalence $\mathcal{F}(f) \sim \|f\|^2_{H^1_h(M^{-1})}$. 
    
    Without loss of generality, since~\eqref{eq:semihom} is linear, we consider an initial data $(f^0_j)_j$ with vanishing mass and $f(t)$ the associated solution of~\eqref{eq:semihom} is such that for $t \ge 0$, $\Pi_h f(t) =0$ { since~\eqref{eq:semihom} preserves mass (see Remark~\ref{rem:semihom})}. Then, from Propositions~\ref{prop:f&Dhf} and~\ref{prop:useful}, there exist constants $C>0$ and $\eta>0$ such that for any $\eps \in (0,\eta)$, there is $K(\eps)>0$ such that:
    \begin{align*}
        {1 \over 2} {\dd \over \dd t} \mathcal{F}(f(t)) &\le - \Symh (f,f) - a \, \Symh (D_h f,D_h f) \\
        &\quad + a\, C \, \|D_h f\|^2_{\ell^2_h(M^{-1})} + a\, C\,  \|f\|_{\ell^2_h(M^{-1})} \|D_hf\|_{\ell^2_h(M^{-1})}\\
        &\le - \Symh (f,f) - a \, \Symh (D_h f, D_h f)  \\
        &\quad + a \, C \, K(\ve) \Symh(f,f) +a \, C \, \ve\,\Symh(D_hf,D_hf)\,.
    \end{align*}
    Choosing first $\eps$ small enough so that $1-C \, \eps \ge 1/2$ and then $a$ small enough so that $1-a\, C \, K(\eps) \ge 1/2$, we obtain that 
    \[
    {\dd \over \dd t} \mathcal{F}(f(t)) \le- \Symh (f,f) - a \, \Symh (D_h f, D_h f) \,.
    \]
    Proposition~\ref{prop:inegfunc_discr} implies that 
    \[
    {\dd \over \dd t} \mathcal{F}(f(t)) \lesssim -\mathcal{F}(f(t))
    \]
    and we can thus conclude thanks to a Gronwall type argument. \qed
    
    \medskip
    \subsection{Hypocoercivity in the inhomogeneous case: proofs of Theorems~\ref{theo:hypo_discrLFP} and~\ref{theo:hypo_fulldiscrLFP}.}
    
    To show the long-time behavior stated in Theorem \ref{theo:hypo_discrLFP}, we introduce 
    \begin{align}
        \label{eq:calH}
        \mathcal{H}(f,f) &= \|f\|^2_\lm + a \|D_\dx f\|^2_\lm + b\|D_\dv f\|^2_\lm \\
        &\quad + 2c\lla D_\dx f, D_\dv f\rra_\lm. \nonumber
    \end{align}
    The positive constants $a$, $b$, and $c$ will be determined in the sequel such that $c^2<ab$. With such hypothesis,  $\mathcal{H}^{1/2}$ is a norm equivalent to $\|\cdot\|_\hm$.  In addition, since \eqref{eq:semikin} is linear, there is no loss of generality in supposing that $\lla f^0\rra_{\dx,\dv}=0$ {so that for any $t \geq 0$, $\lla f(t) \rra_{\dx,\dv}=0$ since \eqref{eq:semikin} preserves the global mass (see Remark~\ref{rem:semikin})}. 
    
    In what follows, shifts in velocity will be denoted $S_+$ and $S_-$
    \[
    \forall (i,j)\in\ZZ/N_x\ZZ\times\ZZ, \; (S_-f)_{i,j} = f_{i,j-1}, \; (S_+f)_{i,j}=f_{i,j+1}\,. 
    \]
    One can notice that $D_\dx$ commutes with $T^\dx$ and $S_\pm$, however 
    \begin{equation}
        \label{eq:hypo_dvT}
        D_\dv T^\dx = T^\dx D_\dv + D_\dx\frac{S_++S_-}{2}\,. 
    \end{equation}
    Eventually, it is worth noticing that $T^\dx$ is skew-symmetric for $\lla \cdot,\cdot\rra_\lm$. We denote $\mathcal{S}^{\dx,\dv}_\alpha$ the following operator $$ \mathcal{S}^{\dx,\dv}_\alpha(f,g)  := \sum_{i \in\ZZ/N_x\ZZ} (\mathcal{S}^{\dv}_\alpha (f,g)))_i \dx.$$ The proof of Theorem \ref{theo:hypo_discrLFP} relies on the three following lemmas.
    
    \begin{lem}
        \label{lem:hypo_equalities}
        Under the assumptions of Theorem \ref{theo:hypo_discrLFP}, the following equalities hold
        \begin{align}
            \label{eq:lem_hypo_equalities_f} \frac{1}{2}\frac{\mathrm{d}}{\mathrm{d}t}&\| f\|^2_\lm = -\mathcal{S}_\alpha^{\dx,\dv}(f,f) \; =: \mathcal{E}_1(f)\\
            \label{eq:lem_hypo_equalities_dxf} \frac{1}{2}\frac{\mathrm{d}}{\mathrm{d}t}&\| D_\dx f\|^2_\lm = 
            -\mathcal{S}_\alpha^{\dx,\dv} (D_\dx f,D_\dx f) \;  =: \mathcal{E}_2(f)
            \\
            \label{eq:lem_hypo_equalities_dvf}\frac{1}{2}\frac{\mathrm{d}}{\mathrm{d}t}&\| D_\dv f\|^2_\lm = 
            \lla \left( I-\frac{S_++S_-}{2}\right) \Pi_\dv D_\dx f, D_\dv f \rra_\lm  \\
            \nonumber & - \lla \left(I-\frac{S_++S_-}{2}\right) (\Pi_\dv -I)D_\dx f, D_\dv f\rra_\lm 
            \\ 
            \nonumber & +\lla (\Pi_\dv -I) D_\dx f,D_\dv f\rra_\lm  
            \\ 
            \nonumber & + \lla [D_\dv,L^\dv_\alpha] f, D_\dv f\rra_\lm  -\mathcal{S}^{\dx,\dv}_\alpha(D_\dv f,D_\dv f)  \;  =: \mathcal{E}_3(f)
            \\
            \label{eq:lem_hypo_equalities_dxdvf} \frac{\mathrm{d}}{\mathrm{d}t}& \lla D_\dx f,D_\dv f\rra_\lm = - \| D_\dx f\|^2_\lm 
            -2 \mathcal{S}^{\dx,\dv}_\alpha(D_\dx f,D_\dv f) 
            \\ \nonumber &- \lla \left(I-\frac{S_++S_-}{2}\right) (\Pi_\dv -I) D_\dx f, D_\dx f \rra_\lm  
            \\ \nonumber & + \lla \left (I-\frac{S_++S_-}{2}\right) \Pi_\dv D_\dx f, D_\dx f\rra_\lm 
            \\ \nonumber & +  \lla [ D_\dv, L^\dv_\alpha] f, D_\dx f\rra_\lm \; =: \mathcal{E}_4(f)\,.
        \end{align}
    \end{lem}
    
    \begin{proof}
        The expression \eqref{eq:lem_hypo_equalities_f} is obtained by injecting \eqref{eq:semikin} in $\mathrm{d}/\mathrm{d}t \|f\|^2_\lm = 2\lla \partial_t f, f\rra_\lm$. Since $T^\dx$ is skew-symmetric, the result is straightforward using \eqref{eq:Lalphabilinear}. Moreover, since $D_\dx f$ satisfies the relation \eqref{eq:semikin}, the equality \eqref{eq:lem_hypo_equalities_dxf} holds true. The following equalities are obtained in the same way, using \eqref{eq:Lalphabilinear}, \eqref{eq:hypo_dvT} {and the fact that $\lla \Pi_\dv D_\dx f, D_\dv f \rra_\lm =0$ for~\eqref{eq:lem_hypo_equalities_dvf}}.
    \end{proof}
    
    \begin{lem}
        \label{lem:hypo_estimates_Poincare}
        Let $N_x\in\NN$ be odd, and $f\in\lm$ { such that $\lla f \rra_{\dx,\dv} = 0$}. The following estimate holds
        \begin{equation}
            \label{eq:lem_hypo_estimates_Poincare}  \|\Pi_\dv f\|_\lm \lesssim \|D_\dx f\|_\lm\,.
        \end{equation}
    \end{lem}
    
    \begin{proof} This Lemma is a consequence of the following discrete Poincaré inequality proven in \cite{bessemoulin_2020_hypocoercivity}: for $N_x \in \NN$ odd, $\dx=1/N_x$,
        \begin{equation}
            \label{eq:lem_hypo_DiscrPoincare}
            \sum\limits_{i\in\mathbb{Z}/N_x\mathbb{Z}}  \beta_i^2 \lesssim \sum\limits_{i\in\mathbb{Z}/N_x\mathbb{Z}}  [D_\dx \beta]_i^2\, , \quad \forall (\beta_i)_{i\in\mathbb{Z}/N_x\mathbb{Z}} \quad \text{s.t.} \,\,\, \sum_{i\in\mathbb{Z}/N_x\mathbb{Z}}  \beta_i=0\,.
        \end{equation}
         Recalling that $(\dpi f)_{i,j}= \rho_i M_j \, \lla M \rra_{\dv}^{-1}$ (see Remark~\ref{rem:semikin}),  we have
        \[
        \|\dpi f\|^2_\lm= \sum\limits_{i\in\mathbb{Z}/N_x \mathbb{Z}} \dx \rho_i^2  \, \frac{1}{ \lla M\rra_\dv}\,, 
        \]
        since $\sum_{j\in\mathbb{Z}} \dv M_j= \lla M \rra_\dv$.  Moreover, we have that $\sum_{i\in\mathbb{Z}/N_x\mathbb{Z}} \rho_i = 0$ since $\sum_{i\in\mathbb{Z}/N_x\mathbb{Z}} \rho_i \dx = \lla f \rra_{\dx,\dv} = 0$ (see Remark~\ref{rem:semikin}). Thus, thanks 
        to 
        \eqref{eq:lem_hypo_DiscrPoincare},
        we can obtain the estimate \eqref{eq:lem_hypo_estimates_Poincare} with Cauchy-Schwarz inequality. Indeed, 
        \begin{align*}
            \sum\limits_{i\in\mathbb{Z}/N_x \mathbb{Z}} \dx (D_\dx \rho)^2_i   \, \frac{1}{ \lla M\rra_\dv}&= 
            \sum\limits_{i\in\mathbb{Z}/N_x\mathbb{Z}} \dx \left( \sum\limits_{j\in\mathbb{Z}} (D_\dx f)_{i,j}  \frac{1}{M_j^{1/2}} \dv \frac{M_j^{1/2}}{\lla M \rra_\dv^{1/2}}\right)^2
            \\&
            \le \sum\limits_{i\in\mathbb{Z}/N_x\mathbb{Z}} \dx \sum\limits_{j\in\mathbb{Z}} (D_\dx f)_{i,j}^2 \frac{1}{M_j}\dv\,,
        \end{align*}
        and the last term is equal to $\|D_\dx f\|^2_\lm$.
    \end{proof}
    \begin{lem}
        \label{lem:hypo_estimates} Let $g\in\lm$. Under the assumptions of Theorem \ref{theo:hypo_discrLFP}, the following estimates hold
        \begin{align}
            \label{eq:lem_hypo_estimates_SPi} &\left|\lla \left( I-\frac{S_++S_-}{2}\right) \Pi_\dv D_\dx f, g \rra_\lm\right| \lesssim \dv \|g\|_\lm \|\Pi_\dv D_\dx f\|_\lm 
            \\ 
            \label{eq:lem_hypo_estimates_SPi-I} & \left| \lla \left(I-\frac{S_++S_-}{2}\right) (\Pi_\dv -I)D_\dx f, g\rra_\lm\right|\lesssim \|g\|_\lm \mathcal{S}^{\dx,\dv}_\alpha(D_\dx f, D_\dx f)^{ 1/2}
            \\ \label{eq:lem_hypo_estimates_DvLDv} & \left| \lla [D_\dv,L^\dv_\alpha] f, D_\dv f\rra_\lm \right| \lesssim \|D_\dv f\|^2_\lm + \|f\|_\lm \|D_\dv f\|_\lm
            \\ \label{eq:lem_hypo_estimates_DvLDx} & \left| \lla [ D_\dv, L^\dv_\alpha] f, D_\dx f\rra_\lm\right| \lesssim  
            \|(\Pi_\dv - I) D_\dx f\|_\lm \; \|f\|_\lm 
            \\ \nonumber & \hspace{20pt} +\|(\Pi_\dv - I) D_\dx f\|_\lm \; \| D_\dv f \|_\lm \,.
        \end{align}
    \end{lem}
    \begin{proof}
        As a preliminary result, let us remark that thanks to Lemma \ref{lem:DhM}, 
        \[
        \forall j\in\ZZ, 
        \; \left| \frac{(D^+_\dv M)_{j}}{M_j}\right|\lesssim \lla j\dv \rra^{-1}\,. 
        \]
        {Noticing that $\Pi_\dv D_\dx f = D_\dx \rho M  \lla M \rra_\dv^{-1}$ (where we recall that $\rho$ is the local mass of $f$)}, one has
        \[
        \lla (I-S_+) { \Pi_\dv} D_\dx f, g \rra_\lm = {\dx\dv\sum\limits_{i\in\ZZ/N_x\ZZ}\sum\limits_{j\in\ZZ}  ( D_\dx \rho)_i  \frac{ M_j}{\lla M \rra_\dv}  g_{i,j} \left( -\dv \frac{(D^+_\dv M)_j}{M_j}\right) \frac{1}{M_{j}}}\,,
        \]
        which yields \eqref{eq:lem_hypo_estimates_SPi}. Similarly, since 
        \[
        \|S_+ f \|^2_\lm - \|f\|^2_\lm = \dx\dv\sum\limits_{i\in\ZZ/N_x\ZZ}\sum\limits_{j\in\ZZ} f^2_{i,j} \dv\frac{(D_\dv^+ M)_{j-1}}{M_{j-1}} \frac{1}{M_j}\,,
        \]
        we obtain that $\|S_+f\|_\lm\lesssim \|f\|_\lm$. The same estimate holds for $S_-$, and it gives \eqref{eq:lem_hypo_estimates_SPi-I} with \eqref{eq:discr_poinc} { and Cauchy-Schwarz inequality}.
        
        Eventually, the last two estimates are obtained using Lemma \ref{lem:estimate_commutator}. Indeed, as the proof relies only on computations on the $v$ variable, it is still true with $\lla\cdot,\cdot\rra_\lm$ instead of $\lla\cdot, \cdot\rra_{\ell^2_h(M^{-1})}$. The estimate~\eqref{eq:lem_hypo_estimates_DvLDv} is straightforward, while \eqref{eq:lem_hypo_estimates_DvLDx} comes from the equality 
        \[
        \lla [D_\dv, L_\alpha^\dv ] f, D_\dx f \rra_\lm = \lla [D_\dv, L_\alpha^\dv] f, (I-\Pi_\dv) D_\dx f\rra_\lm\,. 
        \]
        Indeed, still denoting $\Pi_\dv f= \rho M \lla M \rra_\dv^{-1}$, one has
        \begin{align*}
            \lla [D_\dv,L^\dv_\alpha] f, \Pi_\dv D_\dx f \rra_\lm 
            &= \dx\dv \sum\limits_{i\in\ZZ/N_x\ZZ} \sum\limits_{j\in\ZZ} \left([D_\dv, L^\dv_\alpha ] f\right)_{i,j} (D_\dx \rho)_i  \frac{1}{\lla M \rra_\dv}\,,
        \end{align*}
        but for all $i\in\ZZ/N_x\ZZ$,
        \begin{align*}
            &\sum\limits_{j\in\ZZ} \left([D_\dv,L^\dv_\alpha] f\right)_{i,j} = \sum\limits_{j\in\ZZ} \left(D_\dv (L^\dv_\alpha f)\right)_{i,j}- \sum\limits_{j\in\ZZ} \left( \Lambda_\alpha^\dv (D_\dv f)\right)_{i,j} - \sum\limits_{j\in\ZZ} \left(\Gamma_\alpha^\dv (D_\dv f) \right)_{i,j}\,,
        \end{align*}
        and the three terms vanish, respectively because of the definition of $D_\dv$, and thanks to \eqref{eq:mass_discrLap}  and~\eqref{def:Gamma}.
    \end{proof}
    
    The proof of Theorem \ref{theo:hypo_discrLFP} relies on these lemmas, and ad-hoc use of Young's inequality. 
    
    \begin{proof}[Proof of Theorem \ref{theo:hypo_discrLFP}]
        Thanks to  Lemmas \ref{lem:hypo_equalities}-\ref{lem:hypo_estimates} and to \eqref{eq:discr_poinc}, there exists a constant $K>0$ such that 
        \begin{align}
            \nonumber \frac{1}{2}\frac{\mathrm{d}}{\mathrm{d}t}\mathcal{H}(f,f)
            & +\mathcal{S}_\alpha^{\dx,\dv} (f,f) + a \mathcal{S}_\alpha^{\dx,\dv}(D_\dx f, D_\dx f) + b \mathcal{S}_\alpha^{\dx,\dv} (D_\dv f,D_\dv f)
            \\ \nonumber &+ \frac{c}{2} \|\Pi_\dv D_\dx f\|^2_\lm + \frac{c}{2}  \|(\Pi_\dv-I) D_\dx f \|^2_\lm 
            \\ \nonumber &+\frac{c}{4}  \|D_\dx f\|^2_\lm + \frac{c}{16}  \| \Pi_\dv f\|^2_\lm.     
            \\ \nonumber  \le & \;
            K b \dv \|\Pi_\dv D_\dx f\|_\lm \|D_\dv f\|_\lm 
            \\& +\nonumber 2Kb \mathcal{S}_\alpha^{\dx,\dv} (D_\dx f,D_\dx f)^{1/2}  \|D_\dv f\|_\lm
            \\ \label{eq:VUE_1} &+Kb \left(\|f\|_\lm \|D_\dv f\|_\lm + \|D_\dv f\|^2_\lm  \right)
            \\ \label{eq:CS} &+2K\left(c^2 \mathcal{S}_\alpha^{\dx,\dv}(D_\dx f,D_\dx f)\right)^{1/2} \left( \mathcal{S}_\alpha^{\dx,\dv}(D_\dv f,D_\dv f)\right)^{1/2} 
            \\ \nonumber &+cK \mathcal{S}_\alpha^{\dx,\dv} (D_\dx f,D_\dx f)^{1/2} \|D_\dx f\|_\lm 
            \\ \nonumber & + cK \dv \|\Pi_\dv D_\dx f\|_\lm \|D_\dx f\|_\lm 
            \\ \label{eq:VUE_2} & + cK \mathcal{S}_\alpha^{\dx,\dv}(D_\dx f,D_\dx f)^{1/2}\left( \|f\|_\lm +  \|D_\dv f\|_\lm\right)\,.
        \end{align}
        Up to a transformation with Cauchy-Schwarz inequality in \eqref{eq:CS}, all lines of the previous inequality come from Lemma \ref{lem:hypo_estimates}. To conclude, Young's inequality is applied, except for \eqref{eq:VUE_1} which is bounded with Proposition \ref{prop:useful}. This proposition is also employed to bound all the $\|D_\dv f\|^2_\lm$ that appear after the use of Young's inequality. Hence, there exists $\eta>0$ and $\Delta v_0>0$ such that for any $\varepsilon \in (0,\eta)$, $\Delta v \in (0, \Delta v_0)$ and any positive constants $\mu$, $\delta$, $\nu$  and $\gamma$, there is $K(\eps) >0$   such that 
        \begin{equation*}
            \frac{1}{2}\frac{\mathrm{d}}{\mathrm{d}t}\mathcal{H}(f,f) + \mathcal{D}(f,f)  \le 0
        \end{equation*}
        where 
        \begin{align*}
            \mathcal{D}(f,f) := &A\;\mathcal{S}_\alpha^{\dx,\dv} (f,f)+  B\;\mathcal{S}_\alpha^{\dx,\dv}(D_\dx f, D_\dx f) + C\;\mathcal{S}_\alpha^{\dx,\dv} (D_\dv f,D_\dv f)
            \\& 
            + D\;\|\Pi_\dv D_\dx f\|^2_\lm+  \frac{c}{2} \|(\Pi_\dv-I) D_\dx f \|^2_\lm 
            \\ &
            +E\;  \|D_\dx f\|^2_\lm +   F\;
            \| \Pi_\dv f\|^2_\lm \le 0\,,
        \end{align*}
        with
        \[
        \begin{array}{l}
            \ds A= 1- KbK(\eps) \left(2+\frac{\dv}{4}\right) - \frac{cK}{2}\left( \frac{K(\eps)}{\mu} + \frac{1}{\delta } \right) 
            \\ 
            \ds B = a -Kb - \frac{Kc}{2}\left( \frac{2c}{\nu}+\frac{1}{\gamma} + \delta + \mu \right) 
            \\ 
            \ds C= b -Kb\eps \left( \frac{\dv}{4}+2 \right) - K\nu - \frac{cK\eps}{2\mu} 
            \\ \ds D=  \frac{c}{2}- Kb\dv-\frac{cK\dv}{2} 
            \\ \ds E =  \frac{c}{4} \left( 1 - 2 K\left( \gamma+{\dv}\right)\right) 
            \\ \ds F =  \frac{c}{16}-
            KbK(\eps) \left(2+\frac{\dv}{4}  \right)- \frac{cK}{2}\left( \frac{1}{\delta}+ \frac{K(\eps)}{\mu} \right)\,.
        \end{array}
        \]
        Let us take \begin{align*}
            &b=\frac{c}{128KK(\eps)}, \;\;  \nu = \frac{b}{2K},\;\;    \gamma=\frac{1}{4K},  \;\; \delta=32K, \;\;   \mu=32KK(\eps)\,.
        \end{align*}
        Thus, we can rewrite $C$, $E$ and $F$ as follows:
        \[
        \begin{array}{l}
            \ds C= \frac{b}{2} -Kb\eps \left( \frac{\dv}{4}+2 \right)  - \frac{cK\eps}{2\mu}  \\
            \ds E =  \frac{c}{4} \left( \frac{1}{2} - 2 K\dv\right)
            \\ \ds F =  \frac{c}{64}-
            \frac{c}{128} \frac{\dv}{4}\,.
        \end{array}
        \] 
        We consecutively choose $\dv$ and $\varepsilon$ small enough so that $D,E,F$ and $C$ are strictly non-negative. Finally,  we take $c$ small enough so that $A>0$ and $a$ large enough so that $B>0$ and $c^2 < ab$. It yields that the dissipation is non-negative.
        Eventually, thanks to Proposition~\ref{prop:inegfunc_discr} and the fact that $\Pi_\dv D_\dv f = 0$, there exists $\lambda >0$ such that for all $f\in\hm$, 
        \begin{align*}
            \lambda \mathcal{H}(f,f)  \leq \mathcal{D}(f,f)\,.
        \end{align*}
        As a consequence, the following inequality holds
        \[
        \frac{1}{2}\frac{\mathrm{d}}{\mathrm{d}t} \mathcal{H}(f,f) + {\lambda} \mathcal{H}(f,f) \le 0\,.
        \]
        Theorem \ref{theo:hypo_discrLFP} follows immediately by a Gronwall type argument and the fact that $\mathcal{H} \sim \|\cdot\|_\hm^2$.
    \end{proof}
    
    We finally can prove the result in the fully-discrete case.
    \begin{proof}[Proof of Theorem \ref{theo:hypo_fulldiscrLFP}]
        We deduce from \eqref{eq:fulldiscretekin} that
        \begin{equation}
            \label{eq:fnp}
            {f_{i,j}^{n+1} =f_{i,j}^{n}} - \Delta t  \left(T^\dx f\right)_{i,j}^{n+1}  +\Delta t   \left(L^\dv_\alpha f\right)_{i,j}^{n+1}, \;\forall (i,j)\in\ZZ/N_x\ZZ \times \ZZ, \, n \in \mathbb{N}\,.
        \end{equation} 
        Thus, we obtain
        \begin{align*}
            \|f^{n+1}\|^2_\lm  & = \lla f^{n+1}, f^{n+1} \rra_\lm
            \\& 
            = \lla f^{n+1}, f^{n} \rra_\lm - \Delta t \lla f^{n+1}, \left(T^\dx f\right)^{n+1}   \rra_\lm
            \\ &  \quad + \Delta t \lla f^{n+1}, \left(L^\dv_\alpha f\right)^{n+1} \rra_\lm
            \\& 
            = \lla f^{n+1}, f^{n} \rra_\lm  - \Delta t \, \mathcal{S}_\alpha^{\dx,\dv} (f^{n+1},f^{n+1})
            \\& 
            = \lla f^{n+1}, f^{n} \rra_\lm  + \Delta t  \; \mathcal{E}_1(f^{n+1})
        \end{align*}
        and similarly,
        \begin{align*}
            \|D_\dx  f^{n+1}\|^2_\lm  & 
            = \lla D_\dx  f^{n+1}, D_\dx  f^{n} \rra_\lm  + \Delta t  \; \mathcal{E}_2(f^{n+1})
        \end{align*}
        where the quantities $\mathcal{E}_i(\cdot)$, $i=1,\dots,4$ are the ones introduced in Lemma \ref{lem:hypo_equalities}. Using again \eqref{eq:fnp}, we deduce that 
        \begin{align*}
            \|D_\dv  f^{n+1}\|^2_\lm  & 
            = \lla D_\dv  f^{n+1}, D_\dv  f^{n} \rra_\lm  + \Delta t  \; \mathcal{E}_3(f^{n+1})
        \end{align*} 
        and 
        \begin{align*}
            &2 \lla D_\dv  f^{n+1}, D_\dx f^{n+1} \rra_\lm \\
            &\quad=  \lla D_\dv  f^{n+1}, D_\dx f^{n}\rra_\lm   +  \lla D_\dx  f^{n+1}, D_\dv f^{n}  \rra_\lm  + \Delta t  \; \mathcal{E}_4(f^{n+1})\,.
        \end{align*} 
        Thus, 
        \begin{equation*}
            \mathcal{H}(f^{n+1},f^{n+1}) = \varphi(f^{n+1},f^n) + \Delta t \left( \mathcal{E}_1(f^{n+1}) + { a} \, \mathcal{E}_2(f^{n+1})+ { b} \,\mathcal{E}_3(f^{n+1})+{ c} \,\mathcal{E}_4(f^{n+1})\right) 
        \end{equation*}
        where \begin{align*}
            \varphi(f^{n+1},f^n)  =&   \lla f^{n+1}, f^{n} \rra_\lm  \\
            &+   a \lla D_\dx  f^{n+1}, D_\dx  f^{n} \rra_\lm  +b  \lla D_\dv  f^{n+1}, D_\dv  f^{n} \rra_\lm  \\
            &+  c \lla D_\dv  f^{n+1}, D_\dx f^{n}\rra_\lm   +  c \lla D_\dx  f^{n+1}, D_\dv f^{n}  \rra_\lm\,.
        \end{align*} 
        We notice that  $\varphi$ is a scalar product on $\hm$ and the associated norm is $\left(\mathcal{H}(\cdot)\right)^{1/2}$. In particular, we have
        \begin{align*}
            |\varphi(f^{n+1},f^n) | & \leq \left(\mathcal{H}(f^{n+1},f^{n+1})\right)^{1/2} \left(\mathcal{H}(f^{n},f^n)\right)^{1/2}\\
            & \leq \frac{1}{2}\mathcal{H}(f^{n+1},f^{n+1}) + \frac{1}{2}\mathcal{H}(f^{n},f^n)\, .
        \end{align*}
        Thus, we obtain
        \begin{equation*}
            \frac{1}{2} \mathcal{H}(f^{n+1},f^{n+1}) \leq \frac{1}{2} \mathcal{H}(f^{n},f^n)  + \Delta t \left( \mathcal{E}_1(f^{n+1}) +{a} \,\mathcal{E}_2(f^{n+1})+ {b} \, \mathcal{E}_3(f^{n+1})+ {c} \, \mathcal{E}_4(f^{n+1})\right)\,.
        \end{equation*}
        The second term of the right-hand side of this last inequality is exactly the same one has in the previous semi-discrete case. Thus, using the exact same inequalities {and making the same choice of constants}, we obtain similarly as before  that there exists  $\lambda >0$ such that 
        \begin{align*}
            \frac{1}{2} \mathcal{H}(f^{n+1},f^{n+1}) & \leq \frac{1}{2} \mathcal{H}(f^{n},f^n)  - \Delta t \, \mathcal{D}(f^{n+1},f^{n+1}) \\
            &  \leq \frac{1}{2} \mathcal{H}(f^{n},f^n)  - \Delta t \, \lambda \,\mathcal{H}(f^{n+1},f^{n+1})\,. 
        \end{align*}
        Finally, we deduce that 
        \begin{align*}
            \mathcal{H}(f^{n+1},f^{n+1}) & \leq (1+ 2  \lambda \Delta t)^{-1}\mathcal{H}(f^{n},f^n)  
        \end{align*}
        which implies 
        \begin{align*}
            \mathcal{H}(f^{n+1},f^{n+1}) & \leq (1+ 2  \lambda \Delta t)^{-n}\mathcal{H}(f^{0},f^0) \, .
        \end{align*}
    \end{proof}

    \section{Numerical simulations}\label{sec:numsim}
    \subsection{Implementation}
    
    The implemention of the schemes have been done in Matlab and the code is available at \url{gitlab.inria.fr/herda/fpfrac}. Here, we give some details on the main implementation task which is the assembling of the matrix $\Laht = \Dht + \Ght$. 
    
    First, one assembles the matrix $\Dht$ of the fractional Laplacian. In practice, we choose the integral truncation parameter such that $K = 10J+1$. One first computes the coefficients $\beta^h_k$ for $k = -K,\dots, K$ using formula \eqref{eq:def_beta} for $k=-K+1,\dots, K-1$ and \eqref{eq:def_beta_last} for $\beta_{\pm K}^h$. From there, the coefficients of the matrix are given by \eqref{eq:def_discrLap_trunc1}-\eqref{eq:def_discrLap_trunc2}.
    
    Then one needs to assemble the matrix $\Ght$ of the drift term following \eqref{def:Gamma_trunc1}-\eqref{def:Gamma_trunc5}. This requires~$\Dht$ as well as the discrete local equilibrium $M_j \approx \mu_\alpha(jh)$ for $j=-J,\dots, J$. The evaluation of $M_j$ is not trivial when $\alpha\neq 1$, because $\mu_\alpha$ is defined by an oscillatory integral. The problem enters the larger framework of the numerical calculation of stable densities which has interested many authors (see \cite{ament_2018_accurate} and references therein). Here we use the efficient method of Ament and O'Neil \cite{ament_2018_accurate} which relies on several different representation formulas (integrals and series) and asymptotics for stable densities. We use their Matlab code available at \url{gitlab.com/s_ament/qastable}.

    \subsection{Test case 1: convergence of the scheme}
    
    We solve the homogeneous fractional Fokker-Planck equation 
    $
    \partial_tf  = L_\alpha f
    $.
    First, we focus on convergence properties of the scheme for various values of $\alpha>0$. Our reference solution is 
    \[
    f(t,v)\ =\ \sum_{i=1}^2  \frac{\theta_i}{(1-e^{-(t+1)\alpha})^{\frac1\alpha}}\mu_\alpha\left(\frac{v-w_ie^{-(t+1)}}{(1-e^{-(t+1)\alpha})^{\frac1\alpha}}\right)\,,
    \]
    with $\theta_1 = 3/4$, $\theta_2 = 1/4$, $w_1 = 2$ and $w_2 = -6$. The truncated velocity domain is $[-L,L]$ with $L=16$ and the time domain of simulation is $[0,T]$ with $T=0.5$. As we focus on the convergence of the scheme in the velocity variable, the time step is appropriately refined at each refinement of the velocity step.  
    
    On Figure~\ref{fig:cv_homogeneous}, we report the error of approximation with respect to the mesh size in $L^\infty_tL^2_v$ and $L^\infty_{t,v}$ norms. We observe that the experimental rate of convergence in both norms is equal to $2$. This is the expected rate when $\alpha \leq 1$. In theory, the rate could be worse ($3-\alpha$) when $\alpha>1$. However, it is common in practice that the experimental order of convergence for the Huang and Oberman discretization of the fractional Laplacian is better than the theoretical $3-\alpha$ (see \cite[Fig2. and Fig 3.]{huang_2014_discretization}), which may explain the improved rate here. On other test cases concerning the fractional heat equation with the same implementation and that we do not report here, we observed a slightly worsened rate of convergence for values of $\alpha$ greater than $1$.

    \begin{figure}
%
%
\begin{tikzpicture}

\begin{axis}[
width=.4\textwidth,
at={(0,0)},
scale only axis,
xmode=log,
xmin=0.01,
xmax=1,
xminorticks=true,
xlabel style={font=\color{white!15!black}},
xlabel={Meshsize},
ymode=log,
ymin=0.00001,
ymax=1,
yminorticks=true,
ylabel style={font=\color{white!15!black}},
title={Error in $L^\infty_tL^2(\mu_\alpha^{-1}\dd v)$ norm},
axis background/.style={fill=white}
]
\addplot [color=red, mark=+, mark options={solid, red}, forget plot]
  table[row sep=crcr]{%
1	0.1308892320513\\
0.5	0.0893854869978074\\
0.25	0.0317621402055704\\
0.125	0.00957849730597995\\
0.0625	0.00257709819744217\\
0.03125	0.000660503532892379\\
};
\addplot [color=green, mark=o, mark options={solid, green}, forget plot]
  table[row sep=crcr]{%
1	0.0858632105113834\\
0.5	0.0403015619739594\\
0.25	0.0123815047792652\\
0.125	0.0033912378572473\\
0.0625	0.000872161814163415\\
0.03125	0.000237928501014633\\
};
\addplot [color=blue, mark=asterisk, mark options={solid, blue}, forget plot]
  table[row sep=crcr]{%
1	0.0899748047880546\\
0.5	0.0285048330821137\\
0.25	0.0077183775776165\\
0.125	0.00197569917853812\\
0.0625	0.000497017182788124\\
0.03125	0.0001244523679925\\
};

\logLogSlopeTriangle{.8}{.4}{.1}{2}{black};

\end{axis}

\begin{axis}[%
width=.4\textwidth,
at={(.5\textwidth,0)},
scale only axis,
xmode=log,
xmin=0.01,
xmax=1,
xminorticks=true,
xlabel style={font=\color{white!15!black}},
xlabel={Meshsize},
ymode=log,
ymin=0.00001,
ymax=1,
yminorticks=true,
ylabel style={font=\color{white!15!black}},
title={Error in $L^\infty_{t,v}$ norm},
axis background/.style={fill=white},
legend style={at={(0.97,0.03)}, anchor=south east, legend cell align=left, align=left, draw=white!15!black}
]
\addplot [color=red, mark=+, mark options={solid, red}]
  table[row sep=crcr]{%
1	0.121996923201134\\
0.5	0.124635683910726\\
0.25	0.0366016900130719\\
0.125	0.0121939309864063\\
0.0625	0.0034549228968605\\
0.03125	0.000899711554205798\\
};
\addlegendentry{$\alpha = 0.8$}

\addplot [color=green, mark=o, mark options={solid, green}]
  table[row sep=crcr]{%
1	0.0681276557438044\\
0.5	0.0421722050006725\\
0.25	0.0137223940635329\\
0.125	0.00412831504386369\\
0.0625	0.00110562522373399\\
0.03125	0.000281351386434087\\
};
\addlegendentry{$\alpha = 1$}

\addplot [color=blue, mark=asterisk, mark options={solid, blue}]
  table[row sep=crcr]{%
1	0.0709922605511245\\
0.5	0.0221148733377137\\
0.25	0.00603464578557304\\
0.125	0.00160122611556378\\
0.0625	0.000413326122330688\\
0.03125	0.000105399772253787\\
};
\addlegendentry{$\alpha = 1.5$}

\end{axis}
\end{tikzpicture}%
        \caption{\textbf{Test case 1.}  Error in $L^\infty_tL^2(\mu_\alpha^{-1}\dd v)$ (left) and $L^\infty_{t,v}$ (right) norm between approximate and analytical solution.} \label{fig:cv_homogeneous}
    \end{figure}
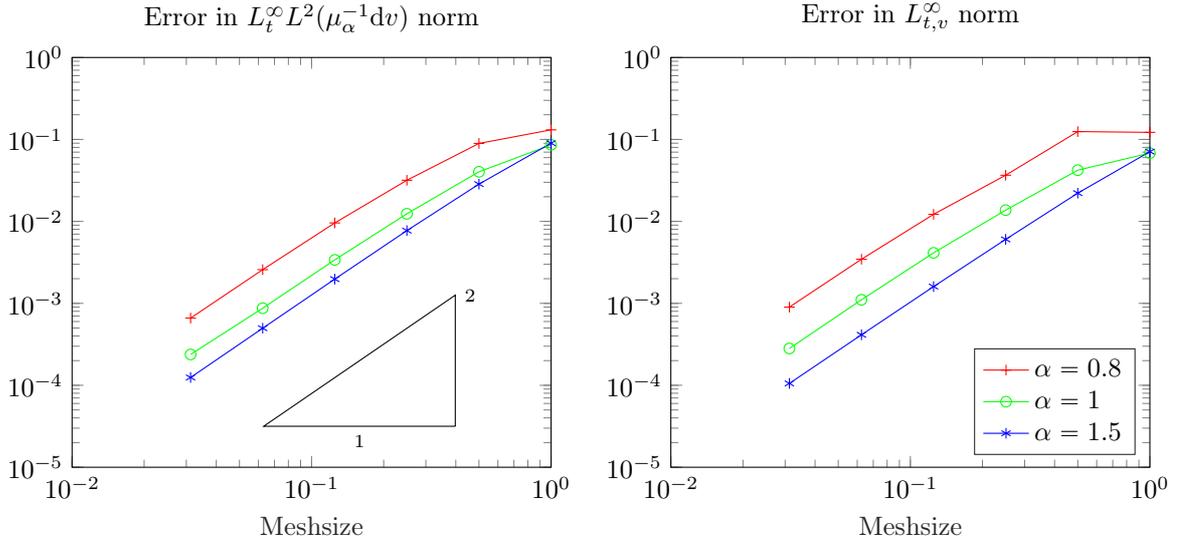
    
    \subsection{Test case 2: heavy tails}
    For this second test case, we illustrate the preservation by the scheme of the heavy-tails of solutions to the homogeneous fractional Fokker-Planck equation 
    $
    \partial_tf  = L_\alpha f
    $. The truncated velocity domain is $[-L,L]$ with $L=20$ and $1025$ mesh points, the time step is $\Delta t = 10^{-2}$ and the initial data is given by 
    \[
    f(0,v)\ =\ \frac{1}{2}\chi_{[-3,-1]}(v) + \frac{1}{4}\chi_{[0,4]}(v)\,,
    \]
    where $\chi_I$ is the indicator function of the set $I$.
    
    On Figure~\ref{fig:dens_homogeneous}, we plot the computed densities at different times in regular and logarithmic scales. We observe that the computed densities develop heavy tails with the algebraic decay $O(|v|^{-1-\alpha})$ as expected, even if the domain is truncated and the initial data is compactly supported.

    \begin{figure}
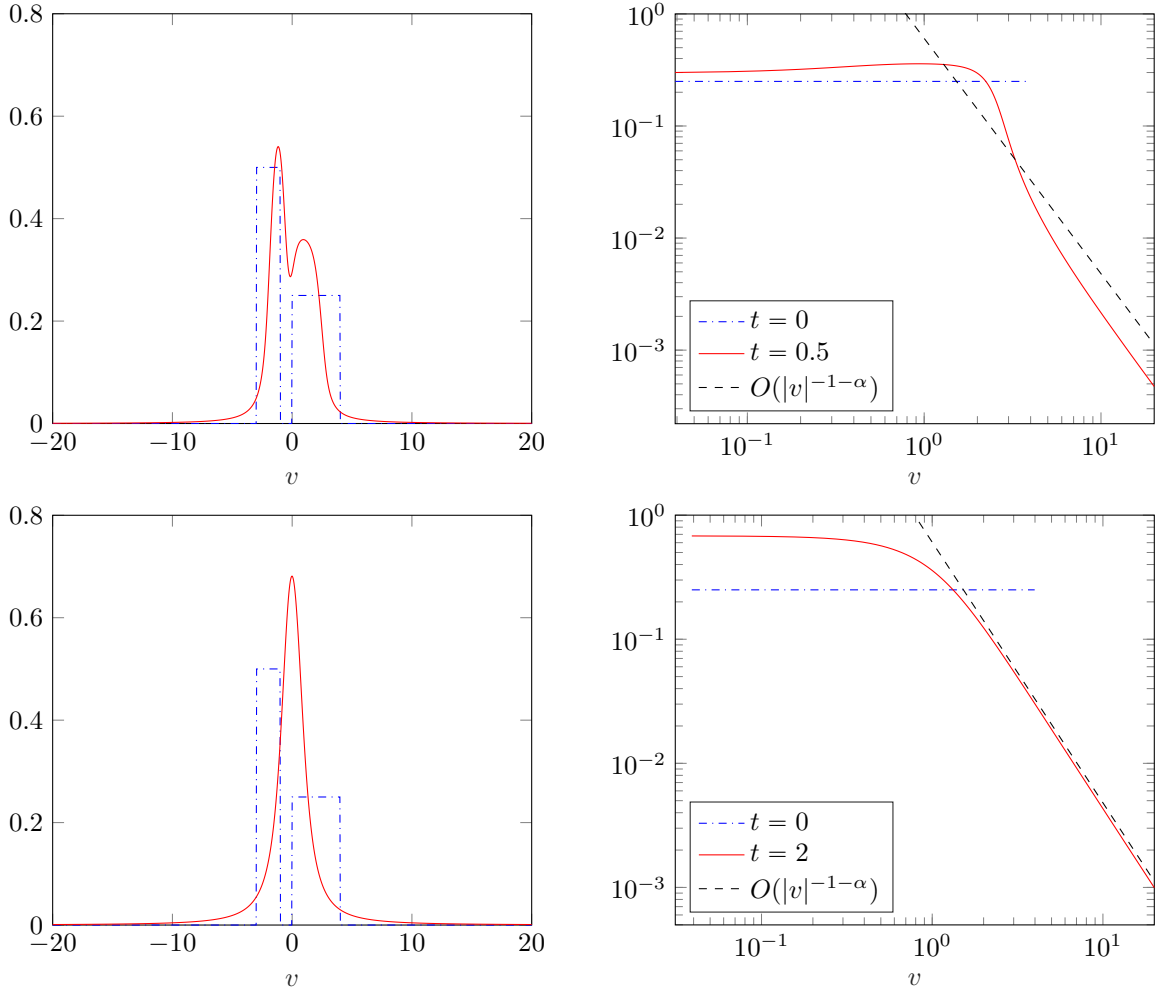

        \input{densities1.tex}
        \input{densities2.tex}
        \caption{\textbf{Test case 2.} Approximate densities at $t=0.5$ (top) and $t=2$ (bottom). On the right the logarithmic scale allows to see the heavy-tail decay. Here $\alpha =1.1$.} \label{fig:dens_homogeneous}
    \end{figure}

    \subsection{Test case 3: numerical hypocoercivity (long time behavior)}
    
    In this section, we illustrate the exponential time stability of our scheme for the kinetic fractional Fokker-Planck equation on a numerical example. In order to certify the results, we first seek a non trivial  analytical solutions of the  equation.
    
    We consider 
    $
    \partial_tf + v\cdot\nabla_xf = L_1f                       
    $
    with $x\in\RR/(2\pi\ZZ)$ and $v\in\RR$. Analytical solutions can be computed by solving the equation in Fourier variable. One family of solutions is parametrized by  $t_0>0$, $v_0\in\RR$ and $x_0\in\RR\in\RR/(2\pi\ZZ)$ and given by
    \begin{equation}\label{eq:ref_solution}
        f(t-t_0,x-x_0,v) = \frac{\tau(t)}{\pi(\tau(t)^2+w(v,v_0,t)^2)} + \frac{1}{\pi}\int_\RR e^{-g(t,\xi)}\cos(\xi w(v,v_0,t) + y(x,v_0,t))\dd \xi\,,
    \end{equation}
    where $\tau(t)\ =\ 1-e^{-t}$, $\eta(t) = t-\tau(t)$, $w(v,v_0,t) = v-v_0e^{-t}$, $y(x,v_0,t) = x-v_0\tau(t)$ and 
    $
    g(t,\xi)\ =\ \int_0^t|\xi e^{-s} + (1-e^{-s})|\dd s\,.
    $
    Using that 
    \[
    g(t,\xi)\ =\ 
    \left\{
    \begin{array}{ll}
        \xi\tau(t)+\eta(t)&\text{if}\ \xi\geq 0\\[.5em]
        \xi(2-\tau(t))+\eta(t)-2\ln(1-\xi)&\text{if}\ \xi \in[-\frac{\tau(t)}{1-\tau(t)},0]\\[.5em]
        -\xi\tau(t)-\eta(t)&\text{if}\ \xi \leq -\frac{\tau(t)}{1-\tau(t)}
    \end{array}
    \right.
    \]
    the integral in the expression of $f$ can be computed explicitly. As the resulting expression is quite lengthy, we do not report it here. 
    
    We run the scheme of Section~\ref{sec:full_num_scheme}, with the truncated operator of Section~\ref{sec:truncation}. The reference solution is given by \eqref{eq:ref_solution} with the parameters $t_0 = 0.5$, $x_0 = 0$ and $v_0 = 1$. The velocity domain is truncated at $L=16$ and discretized $65$ points ($J = 32$). The space domain, of size $2\pi$, is discretized with $128$ points. The time step is $\Delta t = 10^{-2}$ and the final time is $T=35$. With these parameters we report an error of $4.5\cdot10^{-2}$ in $L^\infty_{t,x,v}$ norm between the computed solution and the reference solution. On Figure~\ref{fig:longtime_inhomogeneous}, we plot the distance between solutions and the equilibrium in the natural $L^2_{x,v}(\mu_\alpha^{-1}\dd v\dd x)$ norm. We do observe exponential decay as predicted and the rate matches that of the reference solution. On coarser meshes, the experimental rate tends to be smaller than the exact rate.

    \begin{figure}
        \input{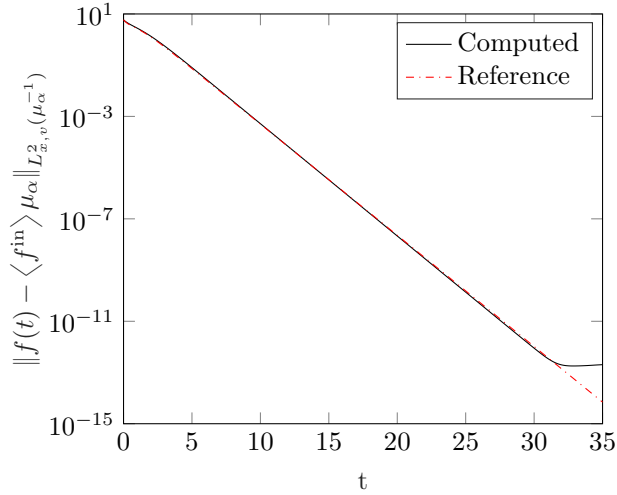}
        \caption{\textbf{Test case 3.} Time evolution of the distance between the steady state and the approximate and reference densities in $L^2_{x,v}(\mu_\alpha^{-1}\dd v\dd x)$ norm.} \label{fig:longtime_inhomogeneous}
    \end{figure}

    \subsection{Semi-Lagrangian version of the scheme}\label{sec:semilag}
    
    In this last part, we propose a modification of the scheme for the kinetic fractional Fokker-Planck equation. It experimentally conserves the same structure-preserving properties, but with a considerably lowered computational cost and increased accuracy. It is based on a Strang splitting approach for solving transport and collisions. The transport step is done with a backward semi-Lagrangian scheme. These methods are standard for kinetic equations and we refer to \cite{DimarcoPareschi15} and references therein for details. We choose piecewise Hermite polynomial function for the reconstruction. More precisely, for any sequence $u = (u_i)_{i\in\ZZ/N_x\ZZ}$, the reconstruction $\Pi_x(u)$ is a $C^1$ function such that if $x\in[x_i,x_{i+1}]$, $\Pi_x(u)(x)$ is the polynomial interpolating $u$ and centered finite difference approximation of its derivatives and  at $x_i$ and $x_{i+1}$. Collisions are solved using our discrete fractional Fokker-Planck operator on truncated domain. In order to improve the order of accuracy in time, we use a Crank-Nicolson approach for the collision step. The scheme reads as follows
    \begin{itemize}
        \item Start from $(f^n_{ij})_{i\in\ZZ/N_x\ZZ, j\in\{-J,\dots,J\}}$.
        \item Compute the transport over a time step $\Delta t/2$, 
        \[
        f^\dagger_{i,j} = \Pi_x((f^n_{k,j})_k)(x_i-v_j\Delta t/2)\,.
        \]
        \item Compute the collisions over a time step $\Delta t$, 
        \[
        (f^{\ddagger}_{i,j})_j = (I-\Delta t/2\Laht)^{-1}(I+\Delta t/2\Laht)(f^\dagger_{i,j})_j\,,
        \]
        where $I$ is the identity matrix.
        \item Compute the transport over a time step $\Delta t/2$, 
        \[
        f^{n+1}_{i,j} = \Pi_x((f^\ddagger_{k,j})_k)(x_i-v_j\Delta t/2)\,.
        \]
    \end{itemize}
    Further computational improvement could be easily obtained by parallelizing the transport step. We try this scheme on the test case of the previous section. The space and velocity steps are progressively refined with a fixed time step. On Table~\ref{tab:semiLag}, we report the duration of the computation as well as the errors between the computed and analytic solutions in absolute norm. The latter is obtained on a Dell Latitude 5490 laptop with an 8th gen Intel Core i7 CPU. We observe experimental convergence. Concerning the computational effort, the computation time on the $1024\times257$  mesh with the semi-Lagrangian version of the scheme is of the same order than the computation time of the original Eulerian scheme on a $128\times 64$ mesh. The structural properties (conservation of mass, long-time behavior, heavy-tails ...) are also preserved by the semi-Lagrangian scheme as in the previous numerical experiments.  We do not report it here for conciseness.
    
    \begin{table}
        \begin{tabular}{llll}
            $N_x$&$N_v$& Error& Duration \\
            \hline\hline
            $128$&$33$&$ 2.8\cdot10^{-1}$&$ 14$ sec\\\hline 
            $256$&$65$&$ 4.6\cdot10^{-2}$&$ 40$ sec\\\hline 
            $512$&$129$&$ 4.1\cdot10^{-2}$&$ 147$ sec\\\hline 
            $1024$&$257$&$ 1.4\cdot10^{-2}$&$ 593$ sec\\\hline 
            $2048$&$513$&$ 4.0\cdot10^{-3}$&$ 2488$ sec\\
        \end{tabular}
        \caption{\textbf{Semi-Lagrangian version of the scheme.} Experimental convergence result and calculation time for various mesh sizes. The time step is constant $\Delta t = 10^{-3}$ and the final time is $T=3.5$. The error is taken in $L^\infty_{t,x,v}$ norm.} \label{tab:semiLag}
    \end{table}

    \section*{Acknowledgements}
    Maxime Herda thanks the LabEx CEMPI (ANR-11-LABX-0007-01). 
    Isabelle Tristani thanks the ANR EFI:  ANR-17-CE40-0030  and the ANR SALVE: ANR-19-CE40-0004 for their support.
    
    This work is part of a collaborative research project that was initiated for the Junior Trimester Program in Kinetic Theory at the Hausdorff Research Institute for Mathematics in Bonn. Part of the work was carried out during the time spent at the institute. The authors are grateful for this opportunity and warmly acknowledge the HIM for the financial support and the hospitality they benefited during their stay.


    \appendix
    
    \section{Estimates on the continuous and discrete local equilibrium}    
    
    \subsection{Bounds on $\mu_\alpha$ and its derivatives}\label{sec:boundsmua}
    
    In this section, we provide some estimates on the stable density $\mu_\alpha$. The results of this section are not new but are scattered among the existing literature \cite{hawkes_1971_lower, doetsch_1974_introduction, uchaikin_1999_chance}. We gather them here. 
    
    As $\mu_\alpha$ is defined in \eqref{eq:defmua} by an oscillatory integral, it is not easy to derive pointwise estimates. For that purpose another representation formula is more appropriate. It reads
    \begin{equation}\label{eq:subordinate_Levy}
        \mua(v)\ =\  \int_0^\infty \frac{1}{(4\,\pi\,u)^{d/2}}\,\exp\left(-\frac{|v|^2}{ 4u}\right)\,\eta_{\alpha,\alpha^{-1}}(u)\,\dd u\,,
    \end{equation}
    where $\eta_{\alpha,\alpha^{-1}}(u)$ is a probability density which is characterized by its Laplace transform 
    \begin{equation}\label{eq:laplace_eta}
        \int_0^\infty e^{-\lambda\,u}\,\eta_{\alpha,t}(u)\,\dd u\ =\ e^{-t\,\lambda^{\alpha/2}}\,,\quad\forall \lambda\geq0\,.
    \end{equation}
    The formula \eqref{eq:subordinate_Levy} is obtained by taking $t = \alpha^{-1}$ and {$\lambda = |\xi|^2$}, for $\xi\in\RR^d$ in \eqref{eq:laplace_eta} and applying the inverse Fourier transform. The density $\eta_{\alpha,t}$ has the following properties.
    \begin{prop}\label{prop:def_subord}
        Let $\alpha\in(0,2)$ and $t>0$. The density $\eta_{\alpha,t}$ is smooth and uniquely defined by \eqref{eq:laplace_eta}.  For any $\mu,u>0$, one has the scaling property
        \begin{equation}\label{eq:scaling_eta}
            \eta_{\alpha,t}(u)\ =\ \mu^{\frac2\alpha}\,\eta_{\alpha,t\mu}(\mu^{\frac2\alpha}\,u)\,.
        \end{equation}
        Moreover, there are positive constants $c_1,c_2$ and $c_3$ depending only on $\alpha$ such that 
        \begin{equation}\label{eq:liminf_eta}
            \lim_{u\to\infty}\eta_{\alpha,1}(u)\,u^{1+\frac\alpha2}\ =\ c_1
        \end{equation}
        and
        \begin{equation}\label{eq:lim0_eta}
            \lim_{u\to0}\eta_{\alpha,1}(u)\,u^{\frac{4-\alpha}{4-2\alpha}}\,\exp\left(c_2\,u^{-\frac{\alpha}{2-\alpha}}\right)\ =\ c_3\,.
        \end{equation}
    \end{prop}
    \begin{proof}
        Using the inverse Laplace transform (or Bromwich transform) one has that 
        \begin{equation}\label{eq:bromwich}
            \eta_{\alpha,t}(u)\ =\ \frac{1}{2\,i\,\pi}\int_{b+i\,\RR} e^{u\,\lambda - t\,\lambda^{\alpha/2}}\dd \lambda\,, \quad t>0\,,\ u\in(0,\infty)
        \end{equation}
        where $b>0$ is arbitrary and $\lambda\to\lambda^{\alpha/2}$ is the analytic continuation of its real counterpart on $\mathbb{C}\setminus(-\infty,0]$. 
        
        The scaling property is easily obtained from \eqref{eq:laplace_eta}. 
        
        It follows from \cite[Theorem 37.1]{doetsch_1974_introduction} that the series expansion $\exp(-\lambda^{\alpha/2}) = \sum_{n=0}^\infty (-1)^n\lambda^{n\alpha/2}/n!$ yields the asymptotic expansion
        \[
        \eta_{\alpha, 1}(u)\approx_{u\to\infty}\frac{\alpha}{2\Gamma\left(1-\frac\alpha2\right)u^{\frac\alpha2+1}}
        \]
 since $\Gamma(-\alpha/2) = -2\Gamma(1-\alpha/2)/\alpha$.       
       
        Finally for the asymptotic expansion at $u\to0$, one can use \eqref{eq:bromwich} with $b = (\frac{\alpha}{2u})^{2/(2-\alpha)}$ to write after a change of variable
        \[
        \eta_{\alpha,1}(u)\ =\ \frac{1}{2\,i\,\pi}\,\left(\frac{\alpha}{2u}\right)^{\frac{2}{2-\alpha}}\int_{1+i\RR} e^{\phi_{\alpha,u}(\lambda)}\dd\lambda
        \]
        where $\phi_{\alpha,u}(\lambda) = (\alpha/2)^{\frac{2}{2-\alpha}}u^{-\frac{\alpha}{2-\alpha}}\left(\lambda - 2\lambda^{\alpha/2}/\alpha\right)$. Then, one uses the saddle point approximation and obtains  an equivalent of the integral when $u\to0$ by replacing $\phi_{\alpha,u}(1+i\tau)$ by its equivalent when $\tau\to0$. It shows \eqref{eq:lim0_eta} (see \cite[Lemma 1]{hawkes_1971_lower} for the expressions of $c_2$ and $c_3$).
    \end{proof}

    \begin{prop}\label{prop:bounds_mua}
        The stable density $\mua$ satisfies
        \begin{equation}\label{eq:mualpha_est1}
            C_1^{-1}\ \leq\ (|v|^{d+\alpha}+1)\mu_\alpha(v)\ \leq\ C_1\,,
        \end{equation}
        \begin{equation}\label{eq:mualpha_est2}
            C_2^{-1}|v_i|\,\ \leq\ (|v|^{2+d+\alpha}+1)\partial_{v_i}\mu_\alpha(v)\,\mathrm{sgn}(-v_i)\ \leq\ C_2|v_i|\,,
        \end{equation}
        and
        \begin{equation}\label{eq:mualpha_est3}
            (|v|^{n+d+\alpha}+1)|\partial_{v_{i_1}\,\dots\,v_{i_n}}^n\mu_\alpha(v)|\ \leq\ C_3\,,
        \end{equation}
         for all $v\in\RR^d$ for some positive constants $C_1$, $C_2$ depending only on $\alpha$ and $d$ and $C_3$ depending additionally on $n$.
    \end{prop}
    \begin{proof}
        We shall use the asymptotic bounds on $\eta_{\alpha,\alpha^{-1}}$ in order to obtain the bounds on $\mua$ and its derivatives.
        Let us define, for $\beta>0$,
        \[
        \nu_{\alpha,\beta}(v)\ =\ \int_0^\infty\frac{1}{u^\beta}\,\exp\left(-\frac{|v|^2}{ 4u}\right)\,\eta_{\alpha,\alpha^{-1}}(u)\,\dd u\,.
        \]
        Observe that $\mua$ and its derivatives can be easily expressed in terms of $\nu_{\alpha,\beta}$ for appropriate $\beta$. Let us derive upper and lower bounds on this quantity. We know from Proposition~\ref{prop:def_subord} that $\eta_{\alpha,\alpha^{-1}}(u)\lesssim_{\alpha}u^{-1-\alpha/2}$.  
        Using this inequality in the expression of $\nu_{\alpha,\beta}$ yields for $v\neq0$ that
        \[
        \nu_{\alpha,\beta}(v)\ \lesssim_{\alpha,\beta}\ |v|^{-2\beta-\alpha}\,.
        \]
        Concerning the lower bounds, we split the cases of small and large $|v|$. First, if $|v|\leq1$, one has 
        \[
        \nu_{\alpha,\beta}(v)\ \geq\ e^{-1/4}\int_{1}^\infty\frac{1}{u^\beta}\,\eta_{\alpha,\alpha^{-1}}(u)\,\dd u\,,\quad \text{for}\ |v|\leq1 \, .
        \]
        For $|v|\geq1$, we use that there is $u_0(\alpha)$ such that for all $u\geq u_0(\alpha)$, one has $\eta_{\alpha,\alpha^{-1}}(u)\gtrsim_\alpha\,u^{-1-\alpha/2}$, implying that 
        \[
        \nu_{\alpha,\beta}(v)\ \gtrsim_{\alpha,\beta}\ |v|^{-2\beta-\alpha}\int_{\frac{ 4u_0(\alpha)}{|v|^2}}^\infty \frac{1}{u^{\beta+1+\alpha/2}}\,\exp\left(-\frac{1}{u}\right)\,\dd u\,,\quad \text{for}\ |v|\geq1\,.
        \]
        Observe that the integral is uniformly bounded from below since $|v|\geq1$. By regrouping everything, we have showed that for some constant $k(\alpha,\beta)$, one has 
        \[
        0<k(\alpha,\beta)^{-1}\ \leq\ (|v|^{2\beta+\alpha}+1)\nu_{\alpha,\beta}(v)\ \leq\ k(\alpha,\beta)\, .
        \]
        This bound can be used repeatedly with appropriate exponents $\beta>0$ to prove  \eqref{eq:mualpha_est1}, \eqref{eq:mualpha_est2} and \eqref{eq:mualpha_est3} from the representation formula \eqref{eq:subordinate_Levy}.

    \end{proof}

    \subsection{Bounds on the discrete equilibrium $M_j$ and its derivatives}
    \label{subsec:BoundsOnTheDiscreteEquilibrium}
    In this subsection, we prove a series of lemmas about decay properties of the equilibrium 
    $
    M_j\ =\ \mu_\alpha(v_j)
    $
    and its discrete derivatives which directly come from the estimates on the continuous equilibrium $\mu_\alpha$ and its derivatives staten in Section~\ref{sec:boundsmua}. { We introduce the notation $\langle \cdot \rangle := (1+|\cdot|^2)^{1/2}$ and notice that one has $(1+|\cdot|)/2 \leq \langle \cdot \rangle \leq 1+|\cdot|$.}
    
    {In all the results of this subsection and the next one, the multiplicative constants that appear in our estimates are all uniform in $j \in \ZZ$ but not in $m \in \ZZ$. As already mentioned, they are always uniform in the mesh size $h$, which has to be taken small in some cases. } 
    
    \begin{lem} \label{lem:DhM}
        For any $j \in \ZZ$, any $m \in \ZZ$, we have:
        $$
        |(D_h M)_{j+m}| \lesssim {1 \over \lla hj \rra^{2+\alpha}} \quad \text{and} \quad |(D_h^{+} M)_{j+m}| \lesssim {1 \over \lla hj \rra^{2+\alpha}} \, .
        $$
    \end{lem}
    
    \begin{proof}
        We only prove the first inequality, the second one is proven exactly in the same way. Using Taylor formula and~\eqref{eq:mualpha_est3} in Proposition~\ref{prop:bounds_mua}, we have:
        $$
        |(D_h M)_{j+m}| \lesssim \int_0^1 \left|\mu_\alpha'(h(j+m-1)h+2hs)\right| \, \dd s \lesssim \int_0^1 {1 \over \lla h(j+m-1)+2hs \rra^{2+\alpha}} \, \dd s\,.
        $$
        Let $J_m:=4+2|m-1|$ then for any $|j| \ge J_m$ and any $s \in [0,1]$, $|(j+m-1)+2s| \ge |j|/2$. Consequently, for $|j| \ge J_m$, we get 
        $
        |(D_h M)_{j+m}| \lesssim {\lla hj \rra^{-2-\alpha}}\,. 
        $
        To conclude, we just remark that for $|j| \le J_m$, we have  
        $
        |(D_h M)_{j+m}| \lesssim 1 \lesssim {\lla hj \rra^{-2-\alpha}}\,.
        $
    \end{proof}
    
    \begin{cor} \label{cor:DhM-1/2}
        For any $j \in \ZZ$, any $m \in \ZZ$, we have:
        \begin{align*}
            &|(D_h^{+} M^{1/2})_{j+m}| \lesssim {1 \over \lla hj \rra^{(3+\alpha)/2}}, \quad |(D_h^{+} M^{-1/2})_{j+m}| \lesssim \lla hj \rra^{(-1+\alpha)/2} \\
            &\hspace{3cm} \text{and} \quad |(D_h^{+} M^{-1})_{j+m}| \lesssim \lla hj \rra^{\alpha}\,.
        \end{align*}
    \end{cor}
    
    \begin{proof}
        We only prove the first estimate, the others are proven exactly in the same way. We write that:
        $
        (D_h^{+} M^{1/2})_{j+m} = (D_h^+ M)_{j+m}(M^{1/2}_{j+m} + M^{1/2}_{j+m+1})^{-1}
        $ 
        so that using Lemma~\ref{lem:DhM} and~\eqref{eq:mualpha_est1} in Proposition~\ref{prop:bounds_mua}, we get 
        $
        \left|(D_h^{+} M^{1/2})_{j+m}\right| \lesssim {\lla hj \rra^{-2-\alpha}} \lla hj \rra^{(1+\alpha)/2} \lesssim {\lla hj \rra^{(-3-\alpha)/2}}\, .
        $
    \end{proof}
    
    \begin{lem} \label{lem:D2hM}
        For any $j \in \ZZ$, any $m \in \ZZ$, we have:
        $$
        |(D_h^2 M)_{j+m}| \lesssim {1 \over \lla hj \rra^{3+\alpha}}
\qquad \text{and} \qquad
            \left| \frac{M_{j+m+1}+M_{j+m-1}-2M_{j+m}}{h^2}\right|\lesssim  { 1 \over \lla hj \rra^{3+\alpha}}\, .
	$$
    \end{lem}
    
    \begin{proof}
        The proof is similar to the previous one, indeed, using again Taylor formula and~\eqref{eq:mualpha_est3} in Proposition~\ref{prop:bounds_mua}, we have 
        \begin{align*}
            |(D_h^2 M)_{j+m}| &\lesssim \int_0^1 \left(\left|\mu_\alpha''(h(j+m) + 2hs)\right| + \left|\mu_\alpha''(h(j+m) - 2hs)\right|\right)  \dd s \\
            &\lesssim \int_0^1 \left( {1 \over \lla h(j+m) + 2hs \rra^{3 + \alpha}} +  {1 \over \lla h(j+m) - 2hs \rra^{3 + \alpha}} \right)  \dd s\,,
        \end{align*}
        we can thus conclude in the same way. The proof of the second inequality is similar.
    \end{proof}
    \begin{cor} \label{cor:D2h1/M}
        For any $j \in \ZZ$, any $m \in \ZZ$, we have:
        $$
        |(D_h^2 \left(M^{-1}))_{j+m}\right| \lesssim \lla hj \rra^{\alpha-1}
\qquad \text{and} \qquad 
        \left| \frac{1}{h^2} \left( \frac{1}{M_{j+m+1}} + \frac{1}{M_{j+m-1}} - 2\frac{1}{M_{j+m} } \right)\right| \lesssim \lla hj \rra^{\alpha-1}\, .
        $$
    \end{cor}

    \begin{proof}
        We first write that 
        \begin{align*}
            (D^2_h(M^{-1}))_{j+m} &=  \frac{1}{4h^2} \frac{2M_{j+m} - M_{j+m-2} - M_{j+m+2}}{M_{j+m-2}M_{j+m}}\\
            &\quad + \frac{1}{4h^2} \frac{(M_{j+m-2}-M_{j+m} + M_{j+m} -M_{j+m+2})(M_{j{\color{vert} +m}}-M_{j+m+2})}{M_{j+m-2}M_{j+m}M_{j+m+2}}\,.
        \end{align*}
        The first term is treated thanks to Lemma~\ref{lem:D2hM} and~\eqref{eq:mualpha_est1} in Proposition~\ref{prop:bounds_mua}:
        \begin{align*}
            &\left| \frac{1}{4h^2} \frac{2M_{j+m} - M_{j+m-2} - M_{j+m+2}}{M_{j+m-2}M_{j+m}} \right|
            \lesssim \frac{\left|(D^2_h M)_{j+m}\right|}{M_{j+m-2}M_{j+m}} \\
            &\qquad \lesssim {1 \over \lla hj \rra^{3+\alpha}} \lla h(j+m-2) \rra ^{1+\alpha} \lla h(j+m) \rra^{1+\alpha} \lesssim \lla hj \rra^{\alpha-1}\,.
        \end{align*}
        The second term is treated similarly using Lemma~\ref{lem:DhM} and~\eqref{eq:mualpha_est1} in Proposition~\ref{prop:bounds_mua}, we get:
        $$
        \left|\frac{1}{4h^2} \frac{(M_{j+m-2}-M_{j+m+2})(M_{j +m}-M_{j+m+2})}{M_{j+m-2}M_{j+m}M_{j+m+2}}\right| \lesssim {\lla hj \rra^{3(1+\alpha)} \over \lla hj \rra^{2(2+\alpha)}} \lesssim \lla hj \rra^{\alpha-1}\,. 
        $$
    \end{proof}
    
    \subsection{Bounds on $(VM)_{j+1/2}$}
    \label{subsec:BoundsOnTheDiscreteVM}
    
    We give here some decay estimates on $(VM)$ and its derivatives. Notice that those properties do not come straightfowardly from the estimates on the continuous equilibrium $\mu_\alpha$. Indeed, the formula which gives~$(VM)$ in~\eqref{def:VM} corresponds in the continuous case to write that 
    $$
    v \mu_\alpha(v) = {1 \over 2} \left(\int_{-\infty}^v (-\Delta)^{\alpha/2} \mu_\alpha(w) \, \dd w - \int_v^{+\infty} (-\Delta)^{\alpha/2} \mu_\alpha(w) \, \dd w\right)\,.
    $$
    This formulation of the operator is not so favorable to get decay estimates as $|v| \to \infty$. However, using the fact that $\int_{\RR}(-\Delta)^{\alpha/2} \mu_\alpha(w) \, \dd w=0$, one can remark that 
    $$
    v \mu_\alpha(v) =\int_{-\infty}^v (-\Delta)^{\alpha/2} \mu_\alpha(w) \, \dd w =-  \int_v^{+\infty} (-\Delta)^{\alpha/2} \mu_\alpha(w) \, \dd w\,.
    $$
    According to the sign of $v$, we can choose one or another of the two previous equalities and it allows to study the behavior of $v \mu_\alpha$ at infinity. We use the same type of ideas in what follows to get bounds on $(VM)$ in the discrete framework.

    \begin{lem} \label{lem:VM}
        There exists $h_0>0$ such that for any $j \in \ZZ$, any odd $m \in \ZZ$ and any~$h \le h_0$, we have:
        $$
        \left|(VM)_{j+{m \over 2}}\right| \lesssim {1 \over \lla hj \rra^\alpha} \,.
        $$
    \end{lem}
    \begin{proof}
        \noindent {\it Step 0.} We introduce the following notations:
        \begin{align*}
            &\widetilde S_1 := \sum_{k \in \ZZ} \sum_{\substack{\ell \in \ZZ^* \\|h\ell| \le 1/\sqrt{2}}} |\beta_\ell^h| \left|2M_k-M_{k+\ell}-M_{k-\ell}\right| h^2\,, \\
            &\widetilde S_2 := \sum_{k \in \ZZ} \sum_{\substack{\ell \in \ZZ^* \\|h\ell| > 1/\sqrt{2}}} |\beta_\ell^h| \left|2M_k-M_{k+\ell}-M_{k-\ell}\right| h^2
        \end{align*}
        and we are going to prove that $\widetilde S_1 + \widetilde S_2 \lesssim 1$, which will be useful in the sequel of the proof. 
        In order to bound $\widetilde S_1$, we use Taylor formula, Lemma~\ref{lem:bounds_beta} and~\eqref{eq:mualpha_est3} in Proposition~\ref{prop:bounds_mua}: 
        \begin{align*}
            \widetilde S_1 &\lesssim \sum_{k \in \ZZ} \sum_{\substack{\ell \in \ZZ^* \\|h\ell| \le 1/\sqrt{2}}} {1 \over |h\ell|^{1+\alpha-2}} \int_0^1 \left|\mu_\alpha''(hk+sh\ell) + \mu_\alpha''(hk-sh\ell)\right| \, \dd s \, h^2 \\
            &\lesssim \sum_{k \in \ZZ} \sum_{\substack{\ell \in \ZZ^* \\|h\ell| \le 1/\sqrt{2}}} {1 \over |h\ell|^{1+\alpha-2}} \int_0^1 \left({1 \over \lla hk+sh\ell \rra^{3+\alpha}} + {1 \over \lla hk-sh\ell \rra^{3+\alpha}}\right) \, \dd s \, h^2\, .
        \end{align*}
        Then, we notice that for $|h\ell| \le 1/\sqrt{2}$ and $s \in [0,1]$,
        $$
        1+|hk\pm s h\ell|^2 \ge 1+\frac{|hk|^2}{2}-|sh\ell|^2 \ge {1 \over 2} (1+|hk|^2)
        $$
        so that $\lla hk\pm sh\ell \rra^{-(3+\alpha)} \lesssim  \lla hk \rra^{-(3+\alpha)}$. Consequently, we obtain
        $$
        \widetilde S_1 \lesssim \sum_{k \in \ZZ} \frac{h}{\lla hk \rra^{3+\alpha}} \sum_{\substack{\ell \in \ZZ^* \\|h\ell| \le 1/\sqrt{2}}} \frac{h}{|h\ell|^{1+\alpha-2}} \lesssim 1\,. 
        $$
        For $\widetilde S_2$, we just use changes of indices in $k$ and~\eqref{eq:mualpha_est1} in Proposition~\ref{prop:bounds_mua} after Fubini-Tonelli's theorem to get:
        $$
        \widetilde S_2 \lesssim \sum_{k \in \ZZ} \frac{h}{\lla hk \rra^{1+\alpha}} \sum_{\substack{\ell \in \ZZ^* \\|h\ell| > 1/\sqrt{2}}} \frac{h}{|h\ell|^{1+\alpha}} \lesssim 1\,. 
        $$
        Thanks to this, we can deduce that
        $$
        \sum_{k \in \ZZ} \sum_{\ell \in \ZZ^*} \beta_\ell^h \left(2M_k-M_{k+\ell}-M_{k-\ell}\right) h^2 =0
        $$
        (which is the equivalent of $\int_{\RR}(-\Delta)^{\alpha/2} \mu_\alpha(w) \, \dd w=0$ in the continuous case). 
        Indeed, the bounds above on $\widetilde S_1$ and $\widetilde S_2$ allow us to use Fubini's theorem. Then, changes of indices in $k$ give the result.
        
        Coming back to the definition of $(VM)$ in~\eqref{def:VM} and denoting $j_m^\pm := j + (m\pm1)/2$ {and using the fact that the mass of $\Lambda_\alpha^h M$ vanishes}, it allows us to write that 
        \begin{equation} \label{def:VMbis}
            \begin{aligned}
                (VM)_{j+{m \over 2}} &= \frac12 \sum_{k=-\infty}^{j_m^-} \sum_{\ell \in \ZZ^*} \beta_\ell^h \left(2M_k-M_{k+\ell}-M_{k-\ell}\right) h^2  
                \\
                &= - \frac12 \sum_{k=j_m^+}^{+\infty} \sum_{\ell \in \ZZ^*} \beta_\ell^h \left(2M_k-M_{k+\ell}-M_{k-\ell}\right) h^2\,. 
            \end{aligned}
        \end{equation}
        In the forthcoming analysis, the idea is then to use either one of these formulas according to the sign of $j_m^\pm$ for large $|j|$. More precisely, we will use the first equality when $j_m^-< 0$ i.e. $j \le - (m+1)/2$ and the second one when $j_m^+> 0$ i.e. $j \ge -(m+1)/2+1$. The two cases being handled exactly in the same way, in what follows, we only treat the case $j_m^+ > 0$. 
        
        \smallskip
        \noindent {\it Step 1.}
        Let $J_m:= 4+|m+1|$. In this part, we consider $j \in \ZZ$ such that $|j|\le J_m$. In order to prove the wanted bound on~$(VM)_{j+\frac{m}{2}}$, it is actually enough to prove that $(VM)_{j+\frac{m}{2}}$ is bounded. But using the definition of $(VM)$ in~\eqref{def:VM}, we directly have for any $|j| \le J$:
        $$
        \left|(VM)_{j+\frac{m}{2}}\right| \lesssim \widetilde S_1 + \widetilde S_2 \lesssim 1 \lesssim {1 \over \lla hj \rra^\alpha}\,.
        $$
        
        \smallskip
        \noindent {\it Step 2.} Consider now $|j| \ge J_m$. We first split $(VM)_{j+m/2}$ into two parts:
        \begin{align*}
            (VM)_{j+{m \over 2}} &= - \frac12 \sum_{k\ge j_m^+}\sum_{\substack{\ell \in \ZZ^* \\ |h\ell| \le 1/\sqrt{2}}}  \beta_\ell^h \left(2M_k-M_{k+\ell}-M_{k-\ell}\right) h^2 \\
            &\quad - \frac12 \sum_{k\ge j_m^+}\sum_{\substack{\ell \in \ZZ^* \\ |h\ell| > 1/\sqrt{2}}}  \beta_\ell^h \left(2M_k-M_{k+\ell}-M_{k-\ell}\right) h^2 =: S^1_{j,m} + S^2_{j,m}\,. 
        \end{align*}
        To estimate $S^1_{j,m}$, we use the same method as in Step~0 to bound $\widetilde S_1$ and we obtain:
        \[
        |S^1_{j,m}| \lesssim \sum_{k \ge j_m^+} \frac{h}{\lla hk \rra^{3+\alpha}} \sum_{\substack{\ell \in \ZZ^* \\|h\ell| \le 1/\sqrt{2}}} \frac{h}{|h\ell|^{1+\alpha-2}}\,.
        \]
        Comparing the series in $k$ with an integral and using that $|j| \ge J_m$, we obtain that:
        \[
        |S^1_{j,m}| \lesssim {1 \over \lla hj \rra^{2+\alpha}}\,. 
        \]
        To deal with $S^2_{j,m}$, we separate it into two parts: 
        \begin{align*}
            S^2_{j,m} =
            &- \frac12 \sum_{k\ge j_m^+}\sum_{\substack{|h\ell| > 1/\sqrt{2} \\ |\ell|\le j_m^+/2}}  \beta_\ell^h \left(2M_k-M_{k+\ell}-M_{k-\ell}\right) h^2 \\
            &- \frac12 \sum_{k\ge j_m^+}\sum_{\substack{|h\ell| > 1/\sqrt{2} \\ |\ell|> j_m^+/2}}  \beta_\ell^h \left(2M_k-M_{k+\ell}-M_{k-\ell}\right) h^2 =: S^{21}_{j,m} + S^{22}_{j,m}\,. 
        \end{align*}
        For the first term $S^{21}_{j,m}$, we use~\eqref{eq:mualpha_est1} in Proposition~\ref{prop:bounds_mua} and the fact that for $|\ell| \le j_m^+/2$ and $k \ge j_m^+$, we have $|k \pm \ell| \ge |k|/2$. We obtain:
        \[
        |S^{21}_{j,m}| \lesssim  \sum_{k\ge j_m^+} {h \over \lla hk \rra^{1+\alpha}} \sum_{\substack{\ell \in \ZZ^* \\ |h\ell| > 1/\sqrt{2}}} {h \over |h\ell|^{1+\alpha}}\,.
        \]
        Once more, using a comparison between series and integrals and the fact that $|j| \ge J_m$, we get that:
        \[
        |S^{21}_{j,m}| \lesssim {1 \over \lla hj \rra^{\alpha}}\,. 
        \]
        For $S^{22}_{j,m}$, using changes of indices in $k$ and~\eqref{eq:mualpha_est1} in Proposition~\ref{prop:bounds_mua}, we have:
        \[
        |S^{22}_{j,m}| \lesssim \sum_{k \in \ZZ} {h \over \lla hk \rra^{1+\alpha}} \sum_{\substack{|h\ell| > 1/\sqrt{2} \\ |\ell|> j_m^+/2}} {h \over |h\ell|^{1+\alpha}} \lesssim \sum_{\substack{\ell \in \ZZ^* \\ |\ell|> j_m^+/2}} {h \over \lla h\ell \rra^{1+\alpha}}\,. 
        \]
        As previously, we can conclude that for any $|j| \ge J_m$,
        \[
        |S^{22}_{j,m}| \lesssim {1 \over \lla hj \rra^\alpha}\,.
        \]
        We have thus obtained that for $|j| \ge J_m$, 
        \[
        \left|(VM)_{j+{m \over 2}}\right| \lesssim  {1 \over \lla hj \rra^\alpha}\,.
        \]
        
        \smallskip
        \noindent{\it Conclusion in the case $j_m^+ > 0$.} Step 1 gives us the wanted result for $|j| \le J_m$ while Step 2 allows us to conclude when $|j| \ge J_m$. 
    \end{proof}
    
    \begin{lem} \label{lem:DhVM}
        For any $j \in \ZZ$ and any odd $m \in \ZZ$, we have:
        $$
        {1 \over h} \left|(VM)_{j+{m \over 2}}-(VM)_{j+{m \over 2}-1}\right| \lesssim {1 \over \lla hj \rra^{1+\alpha}}\,.
        $$
    \end{lem}
    \begin{proof}
        In the subsequent proof, we note $j_m:=j+(m-1)/2$. Using the definition of $(VM)$ in~\eqref{def:VM}, we have:
        \begin{align*}
            &{1 \over h} \left((VM)_{j+{m \over 2}}-(VM)_{j+{m \over 2}-1}\right) \\
            & = {1 \over {4h}} \sum_{k \in \ZZ} \sum_{\ell \in \ZZ^*} \beta_\ell^h (2M_k-M_{k+\ell}-M_{k-\ell}) \\
            &\hspace{3cm} \left(\mathrm{sgn}\left(j_m+{1\over2}-k\right) - \mathrm{sgn}\left(j_m-{1\over 2}-k\right)\right)h^2 \\
            &= \frac12 \sum_{\ell \in \ZZ^*} \beta_\ell^h (2M_{j_m}-M_{j_m+\ell}-M_{j_m-\ell}) \, h\,.  
        \end{align*}
        As in the proof of Lemma~\ref{lem:VM}, we separate the analysis of the cases of small and large $|j|$. Let $J_m:=|m-1|$. 
        
        \noindent {\it Step 1.} The proof of the wanted estimate for $|j|$ follows the Steps~0 and~1 of the proof of Lemma~\ref{lem:VM}. We consider $j \in \ZZ$ such that $|j|\le J_m$. Introducing the notations 
        \begin{align*}
            &\widetilde S^1_{j,m} := \sum_{\substack{\ell \in \ZZ^* \\ |h\ell| \le 1/\sqrt{2}}} |\beta_\ell^h| \left|2M_{j_m}-M_{j_m+\ell}-M_{j_m-\ell}\right| h\,, \\
            &\widetilde S^2_{j,m} := \sum_{\substack{\ell \in \ZZ^* \\ |h\ell| > 1/\sqrt{2}}} |\beta_\ell^h| \left|2M_{j_m}-M_{j_m+\ell}-M_{j_m-\ell}\right| h\,,
        \end{align*}
        we get:
        \begin{align*}
            &{1 \over h} \left|(VM)_{j+{m \over 2}}-(VM)_{j+{m \over 2}-1}\right| \lesssim \widetilde S^1_{j,m} +  \widetilde S^2_{j,m}  
            \lesssim 1 \lesssim {1 \over \lla hj \rra^{1+\alpha}}\,. 
        \end{align*}
        
        \noindent {\it Step 2.} For $|j| \ge J_m$, we have $| j_m | \ge |j|/2$ so that doing as in the Step 0 of the proof of Lemma~\ref{lem:VM}, we get:
        $$
        \widetilde S^1_{j,m}  \lesssim {1 \over \lla hj_m\rra^{3+\alpha}} \lesssim {1 \over \lla hj\rra^{3+\alpha}}\,. 
        $$ 
        Concerning $\widetilde S^2_{j,m}$, we split it into two parts following ideas of the proof of Lemma~\ref{lem:VM}:
        \begin{align*}
            \widetilde S^2_{j,m} &=
            \sum_{\substack{|h\ell| > 1/\sqrt{2} \\ |\ell| \le |j_m|/2}} |\beta_\ell^h|\left|2M_{j_m}-M_{j_m+\ell}-M_{j_m-\ell}\right| h \\
            &\quad +  \sum_{\substack{|h\ell| > 1/\sqrt{2} \\ |\ell| > |j_m|/2}} |\beta_\ell^h| \left|2M_{j_m}-M_{j_m+\ell}-M_{j_m-\ell}\right|h 
            =: \widetilde S^{21}_{j,m} + \widetilde S^{22}_{j,m}\,.
        \end{align*}
        The first part is easily bounded because when $|\ell| \le |j_m|/2$, we have:
        $$
        |2M_{j_m}-M_{j_m+\ell}-M_{j_m-\ell}| \lesssim {1 \over \lla hj_m \rra^{1+\alpha}} \lesssim {1 \over \lla hj \rra^{1+\alpha}}
        $$
        where we used that $|j| \ge J_m$ to write the last inequality, and because
        $$
        \sum_{|h \ell|>{1/\sqrt{2}}} |\beta_\ell^h| h \lesssim 1\,. 
        $$
        For the second one, using that for $|h\ell| > 1/\sqrt{2}$ and $|\ell|>|j_m|/2$, we have 
        $$
        |h\ell|^{1+\alpha} \gtrsim \lla h \ell \rra^{1+\alpha} \gtrsim \lla hj_m\rra^{1+\alpha}\, ,
        $$
        so that using changes of indices, we get:
        \begin{align*}
            \widetilde S^{22}_{j,m} &\lesssim M_{j_m}  \sum_{|h\ell| > 1/\sqrt{2}} {h \over |h\ell|^{1+\alpha}}
            + 
            {1 \over \lla hj_m\rra^{1+\alpha}} \sum_{\ell \in \ZZ} (M_{j_m+\ell}+M_{j_m-\ell})h \\
            &\lesssim {1 \over \lla hj_m \rra^{1+\alpha}} + {1 \over \lla hj_m \rra^{1+\alpha}} \sum_{\ell \in \ZZ} {h \over \lla h\ell \rra^{1+\alpha}} \lesssim {1 \over \lla hj_m \rra^{1+\alpha}} \lesssim {1 \over \lla hj \rra^{1+\alpha}}\,,
        \end{align*}
        which yields the wanted result. 
    \end{proof}
    
    \begin{lem} \label{lem:D2hVM}
        There exists $h_0>0$ such that for any $j \in \ZZ$, any odd $m \in \ZZ$ and any~$h \le h_0$, we have:
        $$
        {1 \over h^2} \left|(VM)_{j+{m \over 2}}- (VM)_{j+{m \over 2}-1}-(VM)_{j+{m \over 2}-2}+(VM)_{j+{m \over 2}-3}\right| \lesssim {1 \over \lla hj \rra^{2+\alpha}}\, .
        $$ 	
    \end{lem}
    
    \begin{proof}
        In the subsequent proof, we note $j_m:=j+(m-1)/2$. Using the definition of $(VM)$ in~\eqref{def:VM}, we have:
        \begin{align*}
            &{1 \over h^2} \left((VM)_{j+{m \over 2}}- (VM)_{j+{m \over 2}-1}-(VM)_{j+{m \over 2}-2}+(VM)_{j+{m \over 2}-3}\right) \\
            &\quad = {1 \over {h^2}} \sum_{\substack{\ell \in \ZZ^* \\ |h\ell| \le 1/\sqrt{2}}} \beta^h_\ell \big(M_{j_m}-M_{j_m+\ell}- M_{j_m-2}+M_{j_m-2+\ell}\big)h^2 \\
            &\qquad +{1 \over {h^2}} \sum_{\substack{\ell \in \ZZ^* \\ |h\ell| > 1/\sqrt{2}}} \beta^h_\ell \big(M_{j_m}-M_{j_m+\ell}- M_{j_m-2}+M_{j_m-2+\ell} \big)h^2 =: S^1_{j,m}+S^2_{j,m}\,. 
        \end{align*}
        To deal with $S^1_{j,m}$, we first notice that 
        \begin{align*}
            S^1_{j,m} &= {1 \over {{ 2} h^2}} \sum_{\substack{\ell \in \ZZ^* \\ |h\ell| \le 1/\sqrt{2}}} \beta^h_\ell \Big(\left(2M_{j_m}-M_{j_m+\ell}-M_{j_m-\ell}\right) \\[-\baselineskip]
            &\hspace{3,5cm} - \left(2M_{j_m-2}-M_{j_m-2+\ell}-M_{j_m-2-\ell}\right) \Big)h^2\,. 
        \end{align*}
        We use Taylor formula to write:
        \begin{align*}
            |S^{1}_{j,m}| 
            &\lesssim \sum_{\substack{\ell \in \ZZ^* \\ |h\ell| \le 1/\sqrt{2}}} \beta_\ell^h \int_0^1 \left|\mu_\alpha''(hj_m+sh\ell) - \mu_\alpha''(h(j_m-2)+sh\ell) \right| \, \dd s \, |h\ell|^2.
        \end{align*}
        Using once more Taylor formula, Lemma~\ref{lem:bounds_beta} and~\eqref{eq:mualpha_est3} in Proposition~\ref{prop:bounds_mua}, we obtain:
        \begin{align*}
            |S^{1}_{j,m}| &\lesssim  \sum_{\substack{\ell \in \ZZ^* \\ |h\ell| \le 1/\sqrt{2}}} \beta_\ell^h \int_0^1 \int_0^2 \left|\mu_\alpha'''(h(j_m-2)+sh\ell+ht) \right| \, \dd t \, \dd s \, h |h\ell|^2 \\
            &\lesssim   \sum_{\substack{\ell \in \ZZ^* \\ |h\ell| \le 1/\sqrt{2}}}  \frac{h}{|h\ell|^{1+\alpha-2}} \int_0^1 \int_0^2 {1 \over \lla h(j_m-2)+sh\ell+ht \rra^{4+\alpha}} \, \dd t \, \dd s\, .
        \end{align*}
        At this point defining $J_m:=8+|m-1|$, 
        we see that if $|j| \le J_m$, we directly have:
        $$
        |S^{1}_{j,m}| \lesssim 1 \lesssim {1 \over \lla hj \rra^{2+\alpha}}\,.
        $$
        If $|j| \ge J_m$, then one can show that for $s \in [0,1]$, $t \in [0,2]$ and $|h\ell| \le 1/\sqrt{2}$, we have:
        $$
        {1 \over \lla h(j_m-2)+sh\ell+ht \rra^{4+\alpha}} \lesssim {1 \over \lla hj\rra^{4+\alpha}}\,,
        $$
        which provides the wanted estimate for $S^{1}_{j,m}$ for large $|j|$. We thus have obtained that for any $j$, 
        \begin{equation} \label{eq:S1_est}
            |S^{1}_{j,m}|  \lesssim  {1 \over \lla hj \rra^{2+\alpha}}\,.
        \end{equation}
        
        We now come to the analysis of $S^2_{j,m}$ and define
        \begin{align*}
            &S^{21}_{j,m} := {1 \over {2h^2}} \sum_{\substack{\ell \in \ZZ^* \\ |h\ell| > 1/\sqrt{2}}} \beta^h_\ell \big(M_{j_m}-M_{j_m-2}\big)\,h^2\,, \\ 
            &S^{22}_{j,m} := {1 \over {2h^2}} \sum_{\substack{\ell \in \ZZ^* \\ |h\ell| > 1/\sqrt{2}}} \beta^h_\ell \big(M_{j_m+\ell}-M_{j_m+\ell-2}\big)\,h^2\,.
        \end{align*}
        Using Taylor formula, Lemma~\ref{lem:bounds_beta} and~\eqref{eq:mualpha_est3} in Proposition~\ref{prop:bounds_mua}, we have:
        \begin{align*}
            |S^{21}_{j,m}| &\lesssim \sum_{\substack{\ell \in \ZZ^* \\ |h\ell| > 1/\sqrt{2}}} \frac{h}{|h\ell|^{1+\alpha}} \int_0^2 |\mu_\alpha'(h(j_m-2)+sh)| \, \dd s \\
            &\lesssim \sum_{\substack{\ell \in \ZZ^* \\ |h\ell| > 1/\sqrt{2}}} \frac{h}{|h\ell|^{1+\alpha}} \int_0^2 {1 \over \lla h(j_m-2)  + sh \rra^{2+\alpha}} \, \dd s\,.
        \end{align*}
        Recalling that $J_m=8+|m-1|$ and separating the cases $|j| \le J_m$ and $|j| \ge J_m$, as previously, we obtain the wanted estimate: for any $j$, we have
        \begin{equation} \label{eq:S21_est}
            |S^{21}_{j,m}| \lesssim { 1 \over \lla hj \rra^{2+\alpha}}\,. 
        \end{equation}
        The analysis of $S^{22}_{j,m}$ is trickier and requires more attention. 
        We first separate it into two parts: noting $N = N(j,m,h) := \lfloor \max\left(1/(\sqrt{2}h),|j_m|/2\right) \rfloor +1$, 
        \begin{align*}
            S^{22}_{j,m} &= {1 \over {2}} \sum_{\substack{|\ell| \le |j_m|/2 \\ |h\ell| > 1/\sqrt{2}}} \beta^h_\ell \big(M_{j_m+\ell}-M_{j_m+\ell-2}\big) + {1 \over {2}} \sum_{|\ell| \ge N} \beta^h_\ell \big(M_{j_m+\ell}-M_{j_m+\ell-2}\big) \\
            &=: S^{221}_{j,m} + S^{222}_{j,m}\,.
        \end{align*}
        The first part is simply treated thanks to Taylor formula, Lemma~\ref{lem:bounds_beta} and~\eqref{eq:mualpha_est3} in Proposition~\ref{prop:bounds_mua}: 
        \[
        \left|S^{221}_{j,m}\right| \lesssim \sum_{\substack{|\ell| \le |j_m|/2 \\ |h\ell| > 1/\sqrt{2}}} \frac{h}{|h\ell|^{1+\alpha}} \int_0^2 {1 \over \lla h(j_m+\ell-2)+hs \rra^{2+\alpha}} \, \dd s\,.
        \]
        We note $\widetilde J_m := 16 + |m-1|$, then separating the cases $|j| \le \widetilde J_m$ and $|j| \ge \widetilde J_m$ (similarly as what we did for~$S^{21}_{j,m}$) and using that $|\ell| \le |j_m|/2$, we get that for any~$j$:
        \begin{equation} \label{eq:S221_est}
            \left|S^{221}_{j,m}\right| \lesssim {1 \over \lla hj \rra^{2+\alpha}}\,. 
        \end{equation}
        It remains to deal with the most complicated part $S^{222}_{j,m}$. We first perform an integration by parts:
        \begin{equation} \label{eq:def_S222}
            \begin{aligned} 
                2S^{222}_{j,m} &= \sum_{\ell \ge N} M_{j_m+\ell} (\beta^h_\ell - \beta^h_{\ell+2}) + \sum_{\ell \le -N-2} M_{j_m+\ell} (\beta^h_\ell - \beta^h_{\ell+2}) \\
                &\quad + \beta_N^h(M_{j_m-N}-M_{j_m+N-2}) - \beta_{N+1}^h(M_{j_m-N-1}-M_{j_m+N-1})
            \end{aligned}
        \end{equation}
        where we used that $\beta_j^h=\beta_{-j}^h$ for any $j \in \ZZ$ to rewrite the last two terms, on which we will first concentrate. If $|j| \le \widetilde J_m$, using Lemma~\ref{lem:bounds_beta} and~\eqref{eq:mualpha_est1} in Proposition~\ref{prop:bounds_mua}, we directly have:
        \[
        \left|  \beta_N^h(M_{j_m-N}-M_{j_m+N-2}) \right| 
        \lesssim 1 \lesssim { 1 \over \lla hj \rra^{2+\alpha}}
        \]
        where we used that $hN \gtrsim 1$ to bound $|hN|^{-1-\alpha}$ by $1$.
        On the other hand, if $|j| \ge \widetilde J_m$, using Taylor formula, Lemma~\ref{lem:bounds_beta}  and~\eqref{eq:mualpha_est3} in Proposition~\ref{prop:bounds_mua}, we have:
        \begin{equation} \label{eq:bordIPP_est1}
            \left|  \beta_N^h(M_{j_m-N}-M_{j_m+N-2}) \right| 
            \lesssim {h \over \lla hN \rra^{1+\alpha}} \int_0^{2(N-1)} {1 \over \lla h(j_m-N+s) \rra^{2+\alpha}} \, \dd s
        \end{equation}
        where we used that $hN \gtrsim 1$ to bound $|hN|^{-1-\alpha}$ by $\lla hN \rra^{-1-\alpha}$. 
        
        \noindent Let us first suppose that $N = \lfloor 1/(\sqrt{2}h) \rfloor +1$, then for $h$ small enough, $N \lesssim 1/h$. Then, since $N \ge |j_m|/2$, we have: 
        \begin{align*}
            \left|  \beta_N^h(M_{j_m-N}-M_{j_m+N-2}) \right| 
            &\lesssim {h \over \lla hj_m \rra^{1+\alpha}} \int_0^{2(N-1)} {1 \over \lla h(j_m-N+s) \rra} \, \dd s\,. 
        \end{align*}
        Moreover, using Peetre's inequality and the fact that for $s \in [0, 2(N-1)]$, $|s-N| \le N$, we get:
        \[
        \left|  \beta_N^h(M_{j_m-N}-M_{j_m+N-2}) \right| 
        \lesssim {1 \over \lla hj_m \rra^{2+\alpha}} 2h(N-1) \lla hN \rra \lesssim {1 \over \lla hj_m \rra^{2+\alpha}} \lesssim {1 \over \lla hj \rra^{2+\alpha}}
        \]
        since $hN \lesssim 1$ and $|j| \ge \widetilde J_m$. 
        
        \noindent We now suppose that $N = \lfloor |j_m|/2 \rfloor +1 \in[|j_m|/2, |j_m|/2+1]$. We come back to~\eqref{eq:bordIPP_est1} and notice that for $s \in [0,2(N-1)]$ and $|j| \ge \widetilde J_m$, we have 
        \[
        {1 \over \lla h(j_m-N+s) \rra^{2+\alpha}} \lesssim {1 \over \lla hj_m \rra^{2+\alpha}}
        \]
        so that 
        \begin{align*}
            &\left|  \beta_N^h(M_{j_m-N}-M_{j_m+N-2}) \right| \lesssim {1 \over \lla hj_m \rra^{1+\alpha}} 2h(N-1) {1 \over \lla hj_m \rra^{2+\alpha}} \\
            &\qquad \qquad \lesssim  {1 \over \lla hj_m \rra^{1+\alpha}} \lla hj_m \rra {1 \over \lla hj_m \rra^{2+\alpha}} \lesssim {1 \over \lla hj_m \rra^{2+2\alpha}} \lesssim {1 \over \lla  hj \rra^{2+2\alpha}}\,.
        \end{align*}
        In the end, we have obtained that for any $j$:
        \begin{equation} \label{eq:bordIPP_est2}
            \left|  \beta_N^h(M_{j_m-N}-M_{j_m+N-2}) \right| \lesssim {1 \over \lla hj \rra^{2+\alpha}}\,.
        \end{equation}
        Exactly in the same way, we can also prove that for any $j$, we have:
        \begin{equation} \label{eq:bordIPP_est3}
            \left| \beta_{N+1}^h(M_{j_m-N-1}-M_{j_m+N-1}) \right| \lesssim {1 \over \lla hj \rra^{2+\alpha}}\,.
        \end{equation}
        Coming back to~\eqref{eq:def_S222}, it remains to deal with the first two sums. We are first going to give an estimate of 
        $\beta^h_{\ell+2}-\beta^h_\ell$ for any $\ell \ge N$. Notice that $N \ge {1 / (\sqrt{2}h})$ so that for $h$ small enough, we have $N \ge 2$. We can thus use the definitions of $\beta_\ell^h$ given in~\eqref{eq:defbeta_even} and~\eqref{eq:defbeta_odd}. Let us restrict to the case where $\ell$ is even, the case where $\ell$ is odd is handled similarly. If $\ell$ is even, using Taylor formula, we have:
        \[
        \left|\beta^h_{\ell+2}-\beta^h_\ell\right| 
        = \frac{C_{1,\alpha}}{h^{1+\alpha}} \left| \int_0^1 (1-t^2) \int_0^2 \left(\varphi_\alpha^{(4)}(\ell+t+s) + \varphi_\alpha^{(4)}(\ell-t+s) \right) \, \dd s \, \dd t\right|\,.
        \]
        Recalling that $\varphi_\alpha^{(3)}(t)= t^{-1-\alpha}$, we deduce that
        \begin{equation} \label{eq:Dbeta_est1}
            \left|\beta^h_{\ell+2}-\beta^h_\ell\right| \lesssim h \int_0^1 \int_0^2 \left( {1 \over (h|\ell + t +s|)^{2+\alpha}} + {1 \over (h|\ell - t +s|)^{2+\alpha}}\right) \, \dd s \, \dd t\,.
        \end{equation}
        Consider $|j| \le \widetilde J_m$. Since $\ell \ge N \ge 1/(\sqrt{2}h)$, for any $t \in [0,1]$, $s \in [0,2]$, for $h$ small enough, we have:
        \[
        |h(\ell \pm t +s)| \ge 1/\sqrt{2} - 4h \gtrsim 1
        \]
        so that 
        \begin{equation} \label{eq:Dbeta_est2}
            \left|\beta^h_{\ell+2}-\beta^h_\ell\right| \lesssim h\,.
        \end{equation}
        Using~\eqref{eq:defbeta_odd}, we can prove that this estimate also holds when $\ell$ is odd. Then, from~\eqref{eq:Dbeta_est2}, we deduce that 
        \[
        \left|\sum_{\ell \ge N} M_{j_m+\ell} (\beta^h_\ell - \beta^h_{\ell+2}) \right| \lesssim \sum_{\ell \ge N} M_{j_m+\ell} h 
        \lesssim \sum_{\ell \in \ZZ} \frac{h}{\lla h \ell \rra^{1+\alpha}} \lesssim 1 \lesssim {1 \over \lla hj \rra^{2+\alpha}}\,.
        \]
        Consider now $|j| \ge \widetilde J_m$. Then, since $N \ge |j_m|/2$, for $\ell \ge N$, for any $s \in [0,2]$ and $t \in [0,1]$, we have:
        \[
        |\ell \pm  t + s| \gtrsim |\ell|\,. 
        \]
        Coming back to~\eqref{eq:Dbeta_est1}, we thus have that
        \begin{equation} \label{eq:Dbeta_est3}
            \left|\beta^h_{\ell+2}-\beta^h_\ell\right| \lesssim \frac{h}{|h \ell|^{2+\alpha}} \lesssim \frac{h}{\lla hj \rra^{2+\alpha}}
        \end{equation}
        where we used that $|h\ell| \ge hN \gtrsim 1$ so that $|h \ell|^{-2-\alpha} \lesssim \lla h \ell \rra^{-2-\alpha}$ and $|\ell| \ge N \gtrsim |j_m| \gtrsim |j|$ because $|j| \ge \widetilde J_m$. Using~\eqref{eq:defbeta_odd}, we can prove that this estimate also holds when $\ell$ is odd. Then, from~\eqref{eq:Dbeta_est3}, we deduce that 
        \begin{equation} \label{eq:sumIPP1}
            \left|\sum_{\ell \ge N} M_{j_m+\ell} (\beta^h_\ell - \beta^h_{\ell+2}) \right| \lesssim {1 \over \lla hj \rra^{2+\alpha}} \sum_{\ell \in \ZZ} \frac{h}{\lla h \ell \rra^{1+\alpha}} \lesssim {1 \over \lla hj \rra^{2+\alpha}}\,.
        \end{equation}
        Similarly, we can prove that 
        \begin{equation} \label{eq:sumIPP2}
            \left|\sum_{\ell \le -N-2} M_{j_m+\ell} (\beta^h_\ell - \beta^h_{\ell+2}) \right| 
            \lesssim {1 \over \lla hj \rra^{2+\alpha}}\,.
        \end{equation}
        Coming back to~\eqref{eq:def_S222}, gathering~\eqref{eq:bordIPP_est2}-\eqref{eq:bordIPP_est3} and~\eqref{eq:sumIPP1}-\eqref{eq:sumIPP2}, we can conclude that for any $j$: 
        \begin{equation} \label{eq:S222_est}
            \left| S^{222}_{j,m} \right| \lesssim  {1 \over \lla hj \rra^{2+\alpha}}\,.
        \end{equation}
        Estimates~\eqref{eq:S221_est} and~\eqref{eq:S222_est} give the wanted result for $S^{22}_{j,m}$, for any $j$, we have: 
        \begin{equation} \label{eq:S22_est}
            |S^{22}_{j,m}| \lesssim { 1 \over \lla hj \rra^{2+\alpha}}\,. 
        \end{equation}
        The bounds obtained in~\eqref{eq:S1_est},~\eqref{eq:S21_est} and~\eqref{eq:S22_est} yield the final result. \\
        
    \end{proof}

    \bibliographystyle{plain}
    \bibliography{bibli}
    
\end{document}